\documentclass[a4paper]{amsart}
\usepackage{amsmath,amssymb,times, amscd, graphicx, 
xspace,mathrsfs,a4wide,hyperref,url,chngcntr}
\usepackage[title]{appendix}
\hypersetup{
  colorlinks   = true, 
  urlcolor     = blue, 
  linkcolor    = blue, 
  citecolor   = blue 
}

\newcommand{\Z}{\ensuremath{\mathbb{Z}}\xspace}
\newcommand{\Q}{\ensuremath{\mathbb{Q}}\xspace}
\newcommand{\R}{\ensuremath{\mathbb{R}}\xspace}

\newcommand{\A}{\ensuremath{\mathbb{A}}\xspace}
\newcommand{\F}{\ensuremath{\mathbb{F}}\xspace}
\newcommand{\Qp}{\ensuremath{\mathbb{Q}_{p}}\xspace}
\newcommand{\Zp}{\ensuremath{\mathbb{Z}_{p}}\xspace}
\newcommand{\Fp}{\ensuremath{\mathbb{F}_{p}}\xspace}

\newcommand{\D}{\mathcal{D}}
\newcommand{\De}{\Delta}

\newcommand{\m}{\ensuremath{\mathfrak{m}}\xspace}

\newcommand{\n}{\ensuremath{\mathfrak{n}}\xspace}
\newcommand{\nbar}{\ensuremath{\overline{\mathfrak{n}}}\xspace}

\newcommand{\OO}{\ensuremath{\mathcal{O}}\xspace}
\newcommand{\U}{\ensuremath{\mathcal{U}}\xspace}

\newcommand{\comment}[1]{}

\DeclareMathOperator{\End}{End}
\DeclareMathOperator{\Hom}{Hom}

\DeclareMathOperator{\Spec}{Spec}
\DeclareMathOperator{\Spf}{Spf}

\DeclareMathOperator{\Spa}{Spa}
\DeclareMathOperator{\Tor}{Tor}
\DeclareMathOperator{\Fil}{Fil}
\DeclareMathOperator{\Ker}{Ker}
\DeclareMathOperator{\Res}{Res}
\DeclareMathOperator{\rad}{rad}
\newcommand{\T}{\ensuremath{\mathbb{T}}\xspace}

\newcommand{\GL}{\ensuremath{\mathrm{GL}}\xspace}

\newcommand{\mbf}{\mathbf}
\newcommand{\mb}{\mathbb}
\newcommand{\mc}{\mathcal}
\newcommand{\ms}{\mathscr}
\newcommand{\mf}{\mathfrak}
\newcommand{\vp}{\varpi}
\newcommand{\ra}{\rightarrow}
\newcommand{\ctens}{\widehat{\otimes}}
\newcommand{\sub}{\subseteq}
\newcommand{\oo}{\mathcal{O}}
\newcommand{\ol}{\overline}
\newcommand{\bu}{\bullet}
\newcommand{\ka}{\kappa}
\newcommand{\al}{\alpha}
\newcommand{\wh}{\widehat}
\newcommand{\wt}{\widetilde}

\newtheorem{thmx}{Theorem}
\newtheorem{theorem}{Theorem}[subsection]
\newtheorem{proposition}[theorem]{Proposition}
\newtheorem{corollary}[theorem]{Corollary}
\newtheorem{lemma}[theorem]{Lemma}

\newtheorem*{conjecture}{Conjecture}

\theoremstyle{definition}
\newtheorem{definition}[theorem]{Definition}
\newtheorem{question}[theorem]{Question}
\newtheorem{remark}[theorem]{Remark}

\mathchardef\mhyphen="2D
\title{Extended eigenvarieties for overconvergent cohomology}
\author{Christian Johansson and James Newton}

\address{Department of Mathematical Sciences, Chalmers University of Technology and the University of Gothenburg, Gothenburg, Sweden}
\email{chrjohv@chalmers.se}
\address{Department of Mathematics, King's College London, Strand, London WC2R 
2LS, UK}
\email{j.newton@kcl.ac.uk}
\input xy
\xyoption{all}
\usepackage{amscd}

\subjclass[2010]{11F33,11F80}
\keywords{$p$-adic automorphic forms, eigenvarieties}

\begin{document}
\maketitle
\paragraph{\textbf{Abstract}} Recently, Andreatta, Iovita and Pilloni have 
constructed 
spaces of overconvergent modular forms in characteristic $p$, together with a 
natural extension of the Coleman--Mazur eigencurve over a compactified (adic) 
weight space. Similar ideas have also been used by Liu, Wan and Xiao to study 
the boundary of the eigencurve. This all goes back to an idea of Coleman.
	
In this article, we construct natural extensions of eigenvarieties for arbitrary reductive groups $\mbf{G}$ over a number field which are split at all places above $p$. If $\mbf{G}$ is $\GL_2/\Q$, then we obtain a new construction of the extended eigencurve of Andreatta--Iovita--Pilloni. If $\mbf{G}$ is an inner form of $\GL_2$ associated to a definite quaternion algebra, our work gives a new perspective on some of the results of Liu--Wan--Xiao.
	
We build our extended eigenvarieties following Hansen's construction using overconvergent cohomology. One key ingredient is a definition of locally analytic distribution modules which permits coefficients of characteristic $p$ (and mixed characteristic). When $\mbf{G}$ is $\GL_{n}$ over a totally real or CM number field, we also construct a family of Galois representations over the reduced extended eigenvariety.

\emph{This is an updated version of the paper, including a correction to Lemma \ref{norm} which was incorrect in the published version}. 

\counterwithin{equation}{subsection}
\section{Introduction}

\subsection{The Halo conjecture}
The \emph{eigencurve}, introduced by Coleman and Mazur \cite{cm}, is a rigid 
analytic curve $\ms{E}^{rig}$ over $\Qp$ which parametrizes systems of Hecke 
eigenvalues of finite slope overconvergent modular forms. It comes equipped 
with a morphism $\ms{E}^{rig} \ra \mc{W}_{0}^{rig}$, called the weight map, 
whose target is known as \emph{weight space}. $\mc{W}_{0}^{rig}$ parametrizes 
continuous characters $\ka : \Zp^{\times} \ra \ol{\Q}_{p}^\times$ and is a 
disjoint union of a finite number of open unit discs. There is also a morphism 
$\ms{E}^{rig} \ra \mb{G}_{m}^{rig}$ which sends a system of Hecke eigenvalues 
to the $U_{p}$-eigenvalue; the $p$-adic valuation of the $U_{p}$-eigenvalue is 
known as the \emph{slope}. The geometry of $\ms{E}^{rig}$ encodes a wealth of 
information about congruences between finite slope overconvergent modular 
forms, and it is therefore not surprising that its study remains a difficult 
topic. In particular, we know very little about the global geometry of 
$\ms{E}^{rig}$ (for example, it is not known whether the number of irreducible 
components of $\ms{E}^{rig}$ is finite or not).

\medskip

Let $q=p$ if $p\neq 2$ and $q=4$ if $p=2$. The components of $\mc{W}_{0}^{rig}$ are parametrized by the characters $(\Z/q\Z)^{\times} \ra (\Z/q\Z)^{\times}$, and if we define $X:=X_{\ka}=\ka(\exp(q))-1$,  $X$ defines a parameter on each component. Very little is known about the global geometry of $\ms{E}^{rig}$ over the centre $|X|\leq q^{-1}$ and it seems likely to be rather complicated. Near the boundary, however, the situation turns out to be rather simple. Coleman and Mazur raised the question of whether the slope tends to zero as one moves along a component of $\ms{E}^{rig}$ towards the boundary of $\mc{W}_{0}^{rig}$. Buzzard and Kilford \cite{bk} investigated this question for $p=2$ (and tame level $1$) and proved a striking structure theorem: $\ms{E}^{rig}|_{\{ |X| >1/8 \}}$ is a disjoint union of connected components $(E_{i})_{i=0}^{\infty}$, and the weight map $E_{i} \ra \{ |X| >1/8 \}$ is an isomorphism for every $i$. Moreover, the slope of a point on $E_{i}$ with parameter $X$ is $i.v_{2}(X)$, where $v_{p}$ is the $p$-adic valuation (normalized so that $v_{p}(p)=1$). Out of this came a folklore conjecture; the following version is essentially \cite[Conjecture 1.2]{lwx}:

\begin{conjecture}
For $r\in (0,1)$ sufficiently close to $1$, $\ms{E}^{rig}|_{\{|X|>r\}}$ is a 
disjoint union of connected components $(E_{i})_{i=0}^{\infty}$ such that each 
$E_{i}$ is finite over $\{ |X| > r \}$. Moreover, there exists constants 
$\lambda_{i}\in \R_{\geq 0}$ for $i=0,1,\ldots$, strictly increasing and 
tending to infinity, such that if $x$ is a point on $E_{i}$ with weight 
parameter $X$, then the slope of $x$ is $\lambda_{i}v_{p}(X)$.  The sequence 
$(\lambda_{i})_{i=0}^{\infty}$ is a finite union of arithmetic progressions, 
after perhaps removing a finite number of terms.
\end{conjecture}

We will loosely refer to this as the `Halo conjecture' (the `halo' in question 
is the (disjoint union of) annuli $\{ |X| >r\}$). Let us assume $p\neq 2$ for 
simplicity. If $\ka$ is a point of $\mc{W}_{0}^{rig}$ then $U_{p}$ acts 
compactly on the space of overconvergent modular forms $M_{\ka}^{\dagger}$. The 
Fredholm determininants $\det(1-T.U_{p}\mid M_{\ka}^{\dagger}) \in 
\ol{\Q}_{p}[[T]]$ interpolate to an entire series 
$F=\sum_{n=0}^{\infty}a_{n}T^{n}$ with coefficients in 
$\Zp[[\Zp^{\times}]]=\oo(\mc{W}_{0}^{rig})^{\circ}$. Fix a character $\eta : 
\Zp^{\times} \ra \Fp^{\times}$ and consider the ideal $I_{\eta}=(p, [n]-\eta(n) 
\mid n=1,\ldots,p-1 )$. The quotient ring $\Zp[[\Zp^{\times}]]/I_{\eta}$ is 
isomorphic to $\Fp[[X]]$ via the map sending $[\exp(p)]$ to $1+X$. We may 
consider the reduction $\ol{F}_{\eta}$ of $F$ modulo $I_{\eta}$ and the 
character $\ol{\ka}_{\eta} : \Zp^{\times} \ra 
(\Zp[[\Zp^{\times}]]/I_{\eta})^{\times} \cong \Fp[[X]]^{\times}$. In an 
unpublished note, Coleman conjectured that there should exist an 
$\Fp((X))$-Banach space $\ol{M}_{\ol{\ka}_{\eta}}^{\dagger}$ of `overconvergent 
modular forms of weight $\ol{\ka}_{\eta}$' with a compact $U_{p}$-action such 
that $\det(1- T.U_{p} \mid \ol{M}_{\ol{\ka}_{\eta}}^{\dagger})=\ol{F}_{\eta}$, 
and promoted the idea that one should study the Halo conjecture via integral 
models of the eigencurve near the boundary of weight space.

\medskip

In \cite{aip}, Andreatta, Iovita and Pilloni proved Coleman's conjecture on the existence of $\ol{M}_{\ol{\ka}_{\eta}}^{\dagger}$ and constructed an integral model $\ms{E}^{\prime}$ of $\ms{E}^{rig}$, which lives in the category of analytic adic spaces \cite{hu2}. $\mc{W}_{0}^{rig}$ has a natural formal scheme model $\Spf \Zp[[\Zp^{\times}]]$, which may be viewed as an adic space $\mf{W}_{0}$ over the affinoid ring $(\Zp,\Zp)$. Apart from the points corresponding to the adic incarnation of $\mc{W}_{0}^{rig}$, $\mf{W}_{0}$ contains an additional $2(p-1)$ points in characteristic $p$, corresponding to the characters $\eta$ and $\ol{\ka}_{\eta}$. The latter points are analytic in Huber's sense (the former are not) and one may consider the analytic locus $\mc{W}_{0}=\mc{W}_{0}^{rig}\cup \{\ol{\ka}_{\eta} \mid \eta \}$ of $\mf{W}_{0}$ (viewing $\mc{W}_{0}^{rig}$ as an adic space). $\mc{W}_{0}$ may be viewed as a compactification of $\mc{W}_{0}^{rig}$, and Coleman's idea may be interpreted as saying that one should study the behaviour of $\ms{E}^{rig}$ near the boundary of $\mc{W}_{0}$ by extending the eigencurve to an adic space living over $\mc{W}_{0}$, turning global behaviour into local behaviour `at infinity'. At a point $\ol{\ka}_{\eta}$, one can no longer measure slopes using $p$. Instead, one has to use $X$. Noting that one can use $X$ not only at $\ol{\ka}_{\eta}$ but also `near' it, the Halo conjecture says that the $X$-adic slope is constant as one approaches $\ol{\ka}_{\eta}$. This supports the idea that an extension of $\ms{E}^{rig}$ exists, and that the $E_{i}$ in the Halo conjecture are $X$-adic Coleman families (i.e. subspaces which are finite over their images in weight space, and of constant slope). In this framework, the slopes of $U_{p}$ should be the $\lambda_{i}$, with multiplicity the degree of $X_{i}$ over $\{ |X| > r \}$. The Halo conjecture asserts, remarkably, that a Coleman family centered at a point over $\ol{\ka}_{\eta}$ extends to some locus $\{ |X| > r \}$ in the component corresponding to $\eta$, where $r$ is independent of the family (and in particular its slope). 

\medskip

In \cite{lwx}, Liu, Wan and Xiao prove the Halo conjecture for eigencurves for 
definite quaternion algebras. Here the construction of (overconvergent) 
automorphic forms is of a combinatorial nature. Those authors succeeded in 
proving the Halo conjecture by calculations on some relatively explicit ad hoc 
integral models of spaces of overconvergent automorphic forms. They construct 
one space over the whole of $\mf{W}_{0}$, with a possibly non-compact 
$U_{p}$-action, and another model over $\{|X| >p^{-1} \}$ with a compact 
$U_{p}$-action. By Chenevier's $p$-adic Jacquet--Langlands correspondence 
\cite{che2}, this proves the Halo conjecture for the components of the 
Coleman--Mazur eigencurve of (generically) Steinberg or supercuspidal type at 
some prime $q \ne p$.

\subsection{Extended eigenvarieties for overconvergent cohomology}

The main goal of this paper is to construct extensions of eigenvarieties for a 
very general class of connected reductive groups $\mbf{G}$ over $\Q$. In 
particular, we give a new construction of the extended eigencurve 
$\ms{E}^{\prime}$ appearing in \cite{aip}.  Our construction also gives a 
conceptual framework for many of the results in \cite{lwx} (and establishes a 
generalisation of some of their results which was described as an `optimistic 
expectation' in \cite[Remark 3.26(2)]{lwx}). See Theorem \ref{lwxboundary} for 
an 
interpretation of some of their results using the extended eigencurve.

Our construction of these eigenvarieties appears in Section \ref{eigenvarieties}. For the purposes of the introduction, we have the following vague statement:
\begin{thmx}\label{vagueA}
Let $F$  be a number field and let $\mbf{H}$ be a connected reductive group 
over $F$ which is split at all places above $p$. Set $\mbf{G}={\rm 
Res}_{\Q}^{F}\mbf{H}$. Then the eigenvarieties for $\mbf{G}$ constructed in 
\cite{han} naturally extend to adic spaces $\ms{X}_{\mbf{G}}$ over the extended 
weight space $\mc{W} = \Spa(\Zp[[T_{0}^{\prime}]],\Zp[[T_{0}^{\prime}]])^{an}$, 
where $T_{0}^{\prime}$ is a certain quotient of the $\Zp$-points $T_{0}$ of a 
maximal torus in a suitable model of $\mbf{G}$ over $\Zp$.
\end{thmx}

The assumption that $\mbf{H}$ is split at all places above $p$ is made for 
convenience only; it makes it easy to define 
a `canonical' Iwahori subgroup. We believe that it should be relatively 
straightforward to generalise our constructions to general quasi-split $\mbf{G}$ over $\Q$ (or to the setting of \cite{loe}). The resulting theory would, however, be 
even more notationally 
cumbersome, so we have decided to stick to the simpler (but still very general) 
situation in this paper. 

\medskip

As a secondary goal, we show (Theorem \ref{gluedgalois}) that when $\mbf{G}={\rm Res}_{\Q}^{F}\GL_{n/F}$, where $F$ is a CM or totally real number field, the reduced eigenvariety that we construct carries a Galois determinant (in the language of \cite{che}) satisfying the expected compatibilities between Frobenii at unramified places and the eigenvalues of Hecke operators. This shows that, in these cases, the new systems of Hecke eigenvalues that we construct in characteristic $p$ carry arithmetic information.

\begin{thmx}\label{B}
There exists an $n$-dimensional continuous determinant $D$ of $G_{F}$ with values in $\oo^+(\ms{X}_{\mbf{G}}^{red})$ such that
$$ D(1 - X\, Frob_{v})=P_{v}(X) $$
for unramified places $v$, where $P_{v}(X)$ is the usual Hecke polynomial (\ref{Heckepoly}).
\end{thmx}

Our proof of Theorem \ref{B} is an adaptation of an argument due to the first author and David Hansen in the rigid setting, which will appear in a slightly refined form in \cite{ch}. It crucially uses Scholze's results on Galois determinants attached to torsion classes \cite{sch}, as well as filtrations on distribution modules constructed by Hansen in \cite{han2}.

To end our brief discussion of the results established in this paper, we 
explain one interpretation of the phrase `naturally extend' in Theorem 
\ref{vagueA}. Suppose for simplicity that $F = \Q$, $\mbf{G}(\R)$ is compact 
modulo centre, and $\mbf{G}(\Qp)$ may be identified with $\GL_n(\Qp)$. We let 
$T_0\sub \GL_n(\Qp)$ denote the diagonal matrices with entries in $\Zp$ (in 
this case $T_{0}=T_{0}^{\prime}$). Modules of overconvergent automorphic forms 
for $\mbf{G}$ (and some fixed tame level, which we suppress) were constructed 
in \cite{chegln} (see also \cite{loe}). If we denote by $U$ the Hecke operator 
corresponding to $$
\begin{pmatrix}
1 &    &  & \\
  &  p &  & \\
  &  & \ddots & \\
  &          & & p^{n-1}
\end{pmatrix}
$$ then this acts compactly on the spaces of overconvergent automorphic forms, and so for each continuous character $\kappa: T_0 \ra \ol{\Q}_p^\times$ there is a characteristic power series $F_\ka \in \ol{\Q}_p[[T]]$ given by the determinant of $1-TU$ on the space of overconvergent automorphic forms of weight $\ka$. The following Theorem is a consequence of our eigenvariety construction, together with Corollary \ref{global}.

\begin{thmx}
The characteristic power series $F_\ka$ glue together to $F_{\mc{W}}\in \oo(\mc{W})\{\{T \}\}$, an entire function on affine $1$-space over $\mc{W}$.

Suppose $\ol{\ka}: T_0 \ra \F_q((X))^\times$ is a continuous character, with $q$ a power of $p$. Then we give an interpretation of the specialisation $F_{\mc{W},\ol{\ka}}$ of $F_{\mc{W}}$ at $\ol{\ka}$ as the characteristic power series of $U$ acting on an $\F_q((X))$-Banach space of overconvergent automorphic forms of weight $\ol{\ka}$.

The Fredholm hypersurface $\ms{Z}$ cut out by $F_{\mc{W}}$ is locally quasi-finite, flat and partially proper over $\mc{W}$ and the eigenvariety $\ms{X}$ comes equipped with a finite map to $\ms{Z}$. 
\end{thmx}

\subsection{Outline of the construction}
The eigenvarieties that we extend are those constructed using overconvergent cohomology (sometimes also referred to as overconvergent modular symbols). Overconvergent cohomology was developed by Ash and Stevens \cite{ste,as}, and the eigenvarieties were constructed by Hansen \cite{han}. Let us recall their construction in the special case of the Coleman--Mazur eigencurve and $p\neq 2$. Let $R$ be an affinoid $\Qp$-algebra in the sense of rigid analytic geometry and let $\ka : \Zp^{\times} \ra R^{\times}$ be a continuous homomorphism. It is well known that $\ka$ is locally analytic, and in particular analytic on the cosets of $1+p^{s}\Zp$ for all sufficiently large $s$. For such $s$, we consider the Banach $R$-module $\mc{A}_{\ka}[s]$ of functions $f : p\Zp \ra R$ which are analytic on the cosets of $p^{s}\Zp$. The monoid
$$ \De = \left\{ \gamma=\begin{pmatrix} a & b \\ c & d \end{pmatrix}\in 
M_{2}(\Zp) \mid p|c,\, a\in \Zp^{\times}, ad-bc\ne 0 \right\} $$
acts on $\mc{A}_{\ka}[s]$ from the right by 
$$(f.\gamma)(x) = \ka(a+bx) f \left( \frac{c+dx}{a+bx} \right).$$
We consider the dual space $\D_{\ka}[s]=\Hom_{R,cts}(\mc{A}_{\ka}[s],R)$, with the dual left action of $\De$. A key point is that the matrix $t=\left( \begin{smallmatrix} 1 & 0 \\ 0 & p \end{smallmatrix} \right)$ acts compactly on $\D_{\ka}[s]$; it factors through the compact injection $\D_{\ka}[s] \hookrightarrow \D_{\ka}[s+1]$. Fix an integer $N\geq 5$ (for simplicity) which is coprime to $p$, and consider the congruence subgroup $\Gamma=\Gamma_{1}(N) \cap \Gamma_{0}(p)\sub \De$. We may view $\D_{\ka}[s]$ as a local system $\wt{\D}_{\ka}[s]$ on the complex modular curve $Y(\Gamma)=\Gamma \backslash \mc{H}$ and consider the singular cohomology group
$$ H^{1}(Y(\Gamma), \wt{\D}_{\ka}[s]) = H^{1}(\Gamma, \D_{\ka}[s]), $$
where the right hand side is group cohomology (in general we would consider 
cohomology in all degrees, but it turns out that $H^{i}(\Gamma, \D_{\ka}[s])=0$ 
if $i\neq 1$; cf. \S \ref{cm}). It carries an action of the Hecke operator 
$U_{p}$. Considering these spaces for varying $s$ and $R=\oo(\U)$, where $\U 
\sub \mc{W}_{0}^{rig}$ is an affinoid open subset, Hansen shows how to 
construct an eigenvariety from the Ash--Stevens cohomology groups by a clever 
adaptation of the eigenvariety 
construction of Coleman \cite{co} (in the one-dimensional case) and Buzzard 
\cite{bu} (the general case)\footnote{Note that it is not clear how to 
topologise the 
$R$-modules 
	$H^{1}(\Gamma, \D_{\ka}[s])$, nor that they can be made into potentially 
	ON-able Banach $R$-modules, so even in this 
	special case 
	we are combining Hansen's eigenvariety construction with the Fredholm 
	theory of Coleman and Buzzard, rather than using the 
	Coleman--Mazur--Buzzard eigenvariety construction.}. This eigenvariety 
	turns 
	out to 
equal the 
Coleman-Mazur eigencurve. 
To extend this construction to $\mc{W}_{0}$, the key point is to define 
generalizations of the modules $\D_{\ka}[s]$ for all open affinoid subsets $\U 
\sub \mc{W}_{0}$. Let $R=\oo(\U)$ and let $\ka : \Zp^{\times} \ra R^{\times}$ 
be the induced character. The first thing to note is that $\ka$ is continuous 
but need not be locally analytic anymore, so one cannot directly copy the 
definition of $\D_{\ka}[s]$. One could try to instead use the space 
$\mc{A}_{\ka}$ of all continuous functions $p\Zp \ra R$. This carries an action 
of $\De$ by the same formula, and we may consider its dual $\D_{\ka}$. However, 
the action of $t$ is no longer compact, so one has to do something different.

\medskip

Let $f : p\Zp \ra R$ be a continuous function and let $f(x)=\sum_{n\geq 0} c_{n} \begin{pmatrix} x/p \\ n \end{pmatrix}$ be its Mahler expansion. Recall that, when $\U \sub \mc{W}_{0}^{rig}$ (i.e.~when $R$ is a $\Qp$-algebra), a theorem of Amice (see \cite[Th\'{e}or\`{e}me I.4.7]{colmez}) says that $f$ is analytic on the cosets of $p^{s+1}\Zp$ if and only $|c_{n}|p^{n/p^{s}(p-1)} \ra 0$ as $n \ra \infty$. Here $|-|$ is any $\Qp$-Banach algebra norm such that $|p|=p^{-1}$. Dually, we may identify $\D_{\ka}$ with the ring of formal power series
$$  \sum_{n\geq 0} d_{n}T^{n} $$
where $T^{n}$ is the distribution $f \mapsto c_{n}(f)$ and $d_{n}$ is bounded as $n \ra \infty$. The analytic distribution module $\D_{\ka}[s]$ is defined by the weaker condition that $|d_{n}|p^{-n/p^{s}(p-1)}$ is bounded as $n \ra \infty$. We may define norms $||-||_{r}$, for $r\in [1/p,1)$, on $\D_{\ka}$ by $||\sum_{n\geq 0} d_{n}T^{n}||_{r} =\sup_{n} |d_{n}|r^{n}$. Let $\D_{\ka}^{r}$ denote the completion of $\D_{\ka}$ with respect to $||-||_{r}$; it may be explicit described as the ring of power series $\sum_{n\geq 0} d_{n}T^{n}$ where $|d_{n}|r^{n} \ra 0$. While the $\D_{\ka}[s]$ are not among the $\D_{\ka}^{r}$, one sees that $\varprojlim_{r \ra 1}\D_{\ka}^{r} = \varprojlim_{s \ra \infty} \D_{\ka}[s]$, so the norms allow one to recover the space of locally analytic distributions. As an aside, we remark it is possible to recover the $\D_{\ka}[s]$ on the nose from the $||-||_{r}$, but we will not need them for the construction of eigenvarieties. 

\medskip

The upshot of considering the norms $||-||_{r}$ is that they may be constructed on $\D_{\ka}$ for any open affinoid $\U \sub \mc{W}_{0}$, by the formula given above. It is, however, not clear a priori that the norms interact well with the action of $\De$. As a monoid $\De$ is generated by the Iwahori subgroup $I=\De \cap \GL_{2}(\Zp)$ and the element $t$. The element $t$ acts via multiplication by $p$ on $p\Zp$ and it is not too hard to see that it induces a norm-decreasing map $(\D_{\ka},||-||_{r}) \ra (\D_{\ka},||-||_{r^{1/p}})$ and that the inclusions $\D_{\ka}^{s} \sub \D_{\ka}^{r}$ for $r<s$ are compact. Thus $t$ induces a compact operator on $\D_{\ka}^{r}$ as desired. The action of $I$ is more complicated to analyse, but it turns out that $I$ acts by isometries for sufficiently large $r$ (depending only on $\ka$). To see this, it is useful to find a different description of $||-||_{r}$. This description, which we will outline below, is one of the key technical innovations of this paper. It is the analogue, in our setting of norms, of the observation in the rigid case that if $\ka$ is $s$-analytic then the $I$-action on $\mc{A}_{\ka}$ preserves $\mc{A}_{\ka}[s]$.

\medskip

In \cite{st}, Schneider and Teitelbaum generalised the norms defined above to 
the spaces $\D(G,L)$ of continuous distributions on a \emph{uniform} pro-$p$ 
group $G$ \cite[Definition 4.1]{ddms} valued in a finite extension $L$ of 
$\Qp$. To recall this construction briefly, a choice of a minimal set of 
topological generators of $G$ induces an isomorphism $G\cong \Zp^{\dim G}$ of 
$p$-adic manifolds and using multi-variable Mahler expansions one may identify 
$\D(G,L)$ (as an $L$-Banach space) with $\oo_{L}[[T_{1},\ldots,T_{\dim 
G}]][1/p]$, and we put ($m=\dim G$)
$$ ||\sum_{\n_{i}\geq 0} d_{n_{1},\ldots,n_{m}}T_{1}^{n_{1}}\ldots 
T_{m}^{n_{m}}||_{r} = \sup |d_{n_{1},\ldots ,n_{m}}|r^{n_{1}+\cdots+n_{m}}. $$
Schneider and Teitelbaum show that these norms are submultiplicative and independent of the choice of minimal generating set. We generalize the construction of these norms to the module $\D(G,R)$ of distributions on a uniform group $G$ valued in a certain class of normed $\Zp$-algebras $R$ that we call \emph{Banach--Tate} $\Zp$-algebras. These include the rings $R=\oo(\U)$ for $\U \sub \mc{W}_{0}$ open affinoid (for a suitable choice of norm) and generalize the constructions in the previous paragraph, which was the special case $G=p\Zp$. Moreover, the action of $g\in G$ on $\D(G,R)$ via left or right translation is an isometry for $||-||_{r}$ (for any $r$).

\medskip 

Let $B_{0}=\{ \gamma \in I \mid c=0 \}$ be the upper triangular Borel and let $\ol{N}_{1}=\{ \left( \begin{smallmatrix} 1 & 0 \\ x & 1 \end{smallmatrix} \right) \mid x\in p\Zp \} \cong p\Zp$; $I$ has an Iwahori decomposition $I \cong \ol{N}_{1}\times B_{0}$. Extend $\ka$ to a character of $B_{0}$ by setting $\ka(\gamma)=\ka(a)$. We have an $I$-equivariant injection $f \mapsto F$ of $\mc{A}_{\ka}$ into the space $\mc{C}(I, R)$ of continuous functions $F : G \ra R$ given by $F(\ol{n}b)=f(\ol{n})\ka(b)$, with $\ol{n}\in \ol{N}_{1}$ and $b\in B_{0}$. Here $I$ acts on $\mc{C}(I,R)$ via left translation. The image is the set of functions $F$ such that $F(gb)=\ka(b)F(g)$ for all $g\in I$ and $b\in B_{0}$. Dually, we obtain an $I$-equivariant surjection $\D(I,R) \ra \D_{\ka}$. If we pretend, momentarily, that $I$ is uniform, then we may consider the quotient norm on $\D_{\ka}$ induced from $||-||_{r}$ on $\D(I,R)$ and one can show that for sufficiently large $r$, this quotient norm agrees with the previously defined $||-||_{r}$ on $\D_{\ka}$. This shows that $I$ acts by isometries on $(\D_{\ka},||-||_{r})$ for sufficiently large $r$. In reality $I$ is not uniform, but one can adapt the argument by working with a suitable open uniform normal subgroup of $I$.

\medskip

This summarizes our construction of the modules $\D_{\ka}^{r}$ which we use to construct the eigenvariety. From the $\D_{\ka}^{r}$, we construct variants $\D_{\ka}^{<r}$ and function modules $\mc{A}_{\ka}^{r}\sub \mc{A}_{\ka}$ as well. When $R$ is a Banach $\Qp$-algebra, then the $\mc{A}_{\ka}[s]$ and $\D_{\ka}[s]$ appearing in \cite{han} are equal to $\mc{A}_{\ka}^{r}$ and $\D_{\ka}^{<r}$, respectively, for $r=p^{-1/p^{s}(p-1)}$. It is easiest, however, to use the modules $\D_{\ka}^{r}$ to construct the eigenvariety since they are potentially orthonormalizable. Using the $\D_{\ka}^{r}$, the construction of the eigenvariety follows \cite{han}, and amounts largely to generalizing various well-known results from rigid geometry and non-archimedean functional analysis. Our arguments also generalize from the case of $\mbf{G}=\GL_{2/\Q}$ to the general case considered in \cite{han} (as stated in Theorem \ref{vagueA}). In particular, our methods work for groups that do not have Shimura varieties (such as $\mbf{G}={\rm Res}_{\Q}^{F}\GL_{n/F}$ for $n\ge 3$), which are intractable by the methods of \cite{aip} (see also \cite{aip2}).

\begin{remark}
	In independent work, Daniel Gulotta \cite{gulottapub} has used a similar
	definition of distribution modules to extend Urban's construction \cite{urb} of equidimensional eigenvarieties for reductive groups possessing
	discrete series.
\end{remark}

\subsection{Questions and future work}

\subsubsection{Generalizations of the Halo conjecture}
It is interesting to consider how the Halo conjecture might generalize beyond the case of $\GL_2/\Q$. For general $\mbf{G}$ we raise the following questions:
\begin{question}
	Does every irreducible component of the extended eigenvariety $\ms{X}_{\mbf{G}}$ contain a point in the locus $p=0$?
\end{question}\begin{question}	
	Are there irreducible components of $\ms{X}_{\mbf{G}}$ contained in the locus $p = 0$?
\end{question}
In the case of $\mbf{G} = \GL_2/\Q$ it is a consequence of the Halo conjecture 
that every irreducible component contains a characteristic $p$ point. Similary, 
when $\mbf{G}$ is an inner form of $\GL_2/\Q$ associated to a definite 
quaternion algebra over $\Q$, it is a consequence of the results of \cite{lwx} 
that every irreducible component contains a characteristic $p$ point (see 
Theorem \ref{lwxboundary}). In general, we regard an affirmative answer to the 
first question as a very weak version of the Halo conjecture. 

If a component has a characteristic $p$ point, it becomes possible to study 
characteristic $0$ points in the component (if they exist) by passing to the 
characteristic $p$ point, or to points approximating the characteristic $p$ 
point. In the case of $\GL_2/\Q$ (or its inner forms), components which have a 
characteristic $p$ point have a Zariski dense set of points corresponding to 
(twists of) classical modular forms of weight $2$. One argument in this spirit 
appears in \cite{px}, which has been used by the authors in combination with 
the methods of \cite{lwx} to establish new cases of the parity conjecture for 
the Bloch--Kato Selmer groups associated to Hilbert modular forms. We 
essentially do this by showing that there is a classical parallel weight $2$ 
point on every irreducible component of an eigenvariety for a definite 
quaternion algebra over a totally real field in which $p$ splits completely. 
See \cite{jn3} for more details.  

\subsubsection{Dimensions of irreducible components and functoriality}
We note here that the theory of global irreducible components for the adic 
spaces we work with requires some explanation (see \cite{con} for the rigid 
case). We have done this in a sequel to this paper \cite{jn2}, where we also 
generalise some of the results of \cite{han}. In particular, we show that the 
lower bound for the dimension of irreducible components \cite[Theorem 
1.1.6]{han} and (a variant of) the interpolation of Langlands functoriality 
\cite[Theorem 5.1.6]{han} generalise to our extended eigenvarieties. 

One application of the interpolation of Langlands functoriality is that in the 
case of $\GL_{2/\Q}$ (or its inner forms), \cite{bp} and \cite{lwx} show that the 
extended eigenvarieties contain the usual rigid eigenvarieties as a proper 
subspace. Applying functoriality (cyclic base change, for example) then shows 
that this is true for a larger class of groups. See \cite{jn2} for more details.

\subsubsection{Galois representations}
In \cite{aip} the natural question is raised as to whether the Galois representations attached to characteristic $p$ points of the extended eigencurve are trianguline (in an appropriate sense). One can similarly ask this question for the characteristic $p$ Galois representations constructed in this paper. Note that in our level of generality, it is still only conjectural that the characteristic $0$ Galois representations carried by the eigenvariety are trianguline, but this is known, for example, in the case where $\mbf{G}$ is a definite unitary group defined with respect to a CM field. It would also be interesting to construct a `patched extended eigenvariety' in this setting, extending the construction of \cite{bhs}, and we hope to study this in the near future.

\subsection{An outline of the paper} Let us describe the contents of the paper. Section \ref{prelim} collects what we need about the eigenvariety machine and the notion of slope decompositions, and introduces some functional-analytic terminology that we will need throughout the paper. Since the key point of the paper is the construction of certain norms, we adopt terminology that puts emphasis on the norm, as opposed to merely the underlying topology. We give a definition of a slope decomposition (a concept introduced in \cite{as}) that differs slightly from the definitions that appear in the literature. This is necessary since the definition given in \cite{as} neither localizes nor glues well, and so is not suitable for the construction of eigenvarieties. Our definition is a formalization of an informal definition that the first author learnt from conversations with David Hansen.

\medskip

In Section \ref{distalg}, we carry out the construction of the norms on the $\D_{\ka}$, following the outline above. We first discuss the generalization of the Schneider--Teitelbaum norms to distributions on a uniform group $G$ valued in a certain class of normed $\Zp$-algebras that we call Banach--Tate $\Zp$-algebras. These include, for example, all Banach $\Qp$-algebras in the usual sense, as well as Tate rings $R=\oo(\U)$ with $\U$ an affinoid open subset of weight space (equipped with a suitable norm). We show that, in a precise sense, the completion of $\D(G,R)$ with respect to the family of norms $(||-||_{r})_{r\in [1/p,1)}$ only depends on the underlying topology of $R$. Imposing some additional conditions on the norm (which is always possible in practice), we then construct the modules $\D_{\ka}^{r}$, $\D_{\ka}^{<r}$ and $\mc{A}_{\ka}^{r}$ as outlined above.

\medskip
 Section \ref{evars} then uses the modules $\D_{\ka}^{r}$ to construct the eigenvariety, following the strategy in \cite{han}. Since the $\D_{\ka}^{r}$ are potentially orthonormalizable, the construction simplifies somewhat. We end the section by generalizing the `Tor-spectral sequence' \cite[Theorem 3.3.1]{han} to our setting, which is a key tool for analyzing the geometry of eigenvarieties, and use it give a description of the `points' of the eigenvariety valued in a local field. We use this description in Section \ref{galrep} when we construct Galois determinants.

\medskip

In Section \ref{gl2} we discuss the relationship of our work with that of 
\cite{aip} and \cite{lwx}. We show that when $\mbf{G}=\GL_{2/\Q}$, our 
construction, over the normalization of the weight space $\mc{W}_{0}$ discussed 
above, produces the same eigencurve as in \cite{aip} (this normalization is 
only different from $\mc{W}_{0}$ if $p=2$). When $\mbf{G}$ is the algebraic 
group over $\Q$ associated with the units of a definite quaternion algebra over 
$\Q$, we show that our framework gives a conceptual proof of \cite[Theorem 
3.16]{lwx}, which is a key ingredient in their proof of the Halo conjecture. In 
essence, the numerical estimate of \cite[Theorem 3.16]{lwx} falls out directly 
from our proof of compactness of the $U_{p}$-operator. Thus, it is possible to 
view our proof of compactness of suitable `$U_{p}$-like' operators (known as 
controlling operators) as a generalization of \cite[Theorem 3.16]{lwx}, as 
asked for in \cite[Remark 3.26(2)]{lwx}. Since this numerical estimate doesn't 
appear strong enough to establish the Halo conjecture in more general 
situations, we have restricted ourselves to proving the statement in the 
setting of \cite{lwx} as an illustration of our method.

\medskip

Finally, Appendix \ref{app} proves various results that we need on the class of Tate rings whose associated affinoid adic spaces appear as the local pieces of our eigenvarieties; some of these results might be of independent interest.

\subsection*{Acknowledgments} This paper was inspired by \cite{aip}, and C.J. 
wishes to thank Vincent Pilloni for some interesting discussions on \cite{aip} 
and David Hansen for innumerable and invaluable discussions on overconvergent 
cohomology and eigenvarieties. He would also like to thank Konstantin Ardakov 
for patiently answering several questions, and both authors would like to thank 
him for providing us with a proof of Lemma \ref{uniform}, and allowing us to 
include it here. We wish to thank John Bergdall and Kevin Buzzard for various 
conversations about the boundary of the eigencurve, as well as helpful comments 
on an earlier version of this paper. We also thank Judith Ludwig, who pointed out an error in the published version of this paper which is corrected here. Finally, we wish to thank the anonymous 
referees for their careful reading and helpful comments. C.J. was supported by 
NSF grant DMS-1128155 during the main work on this paper, and by the Herchel Smith Foundation during the revision. J.N. was supported by ERC Starting Grant 306326. Work on this project was also supported by the EPSRC Platform Grant EP/I019111.

\section{Preliminaries}\label{prelim}
The goal of this section is to set up some functional analytic terminology and theory. Specifically, we require the results of \cite[\S 2-3]{bu} on Fredholm determinants, Riesz theory and the construction of spectral varieties in a level of generality that is intermediate between the settings of \cite{bu} and \cite{co} (cf. also \cite[Annexe B]{aip}). For example, we need to work over coefficient rings arising from affinoid opens in the adic space $\mathcal{W}_0$ discussed in our Introduction. These rings are complete topological rings which are \emph{Tate} in the language of Huber (cf. \cite[\S 1]{hu1}). The topology on these rings is induced by a norm, and to discuss the spectral theory of compact operators it is convenient to fix such a norm. This gives rise to a class of normed rings which we call Banach--Tate rings (see Definition \ref{defBT}). 

The proofs in \cite{bu} go through with little to no change when working over Banach--Tate rings, so we will be rather brief. All norms etc. will be non-archimedean so we will ignore this adjective. All rings will be commutative unless otherwise specified.

\subsection{Fredholm determinants over Banach--Tate rings}\label{fred}
\begin{definition}\label{normedrings} Let $R$ be a ring. A function $|-|\, :\, 
R \ra \R_{\geq 0}$ is called a \emph{seminorm} if (for all $r,s\in R$)
\begin{enumerate}
\item $|0|=0$, $|1|=1$;

\item $|r+s|\leq {\rm max}(|r|,|s|);$

\item $ |rs|\leq |r||s|$.
\end{enumerate}
If in addition $|r|=0$ only if $r=0$, then we say that $|-|$ is a \emph{norm}. A ring $R$ together with a (semi)norm will be called a (semi)normed ring. A normed ring $R$ is called a Banach ring if the metric induced by the norm is complete. 

If $f: R \rightarrow S$ is a morphism of normed rings, we say that $f$ is \emph{bounded} if there is a constant $C > 0$ such that $|f(r)| \le C|r|$ for all $r\in R$. 

We say that two norms $|-|$, $|-|^{\prime}$ on a ring $R$ are \emph{equivalent} 
if they induce the same topology. We say that they are 
\emph{bounded-equivalent} if there are constants $C_{1},C_{2}>0$ such that 
$C_{1}|r|\leq |r|^{\prime}\leq C_{2}|r|$ for all $r\in R$ (note that this is 
stronger than equivalence, cf.~Lemma \ref{norm}). Let $R$ be a normed ring. We 
say that $r\in R$ is \emph{multiplicative} if $|rs|=|r||s|$ for all $s\in R$.
\end{definition}

\begin{definition}\label{defBT}
Let $R$ be a normed ring. We say that $R$ is \emph{Tate} if $R$ contains a 
multiplicative\footnote{i.e.~a unit which is multiplicative in the sense we 
	just defined, as well as being a unit for multiplication!} unit $\vp$ such 
	that $|\vp|< 
1$. We call such a 
$\vp$ a \emph{multiplicative pseudo-uniformizer}. If $R$ is also complete, we 
say that $R$ is a \emph{Banach--Tate ring}. If $R$ is a Tate normed ring and 
$\vp$ is a multiplicative pseudo-uniformizer, then we define the corresponding 
valuation $v_{\vp}$ on $R$ by $v_{\vp}(r)=-\log_{a}|r|$, where $a=|\vp^{-1}|$. 
\end{definition}

We remark that it is easy to see that a unit $\vp$ in a normed ring $R$ is multiplicative if and only if $|\vp^{-1}|=|\vp|^{-1}$. A multiplicative pseudo-uniformizer $\vp$ is a uniform unit in the sense of Kedlaya-Liu \cite[Remark 2.3.9(b)]{kl}. 

\begin{remark}\label{tateremark}
Let $R$ be a Tate normed ring, with $\vp$ a multiplicative pseudo-uniformizer. 
\begin{enumerate}
\item The underlying topological ring is a Tate ring in the language of Huber; the unit ball $R_{0}$ is a ring of definition and $\vp$ is a topologically nilpotent unit. Conversely, assume $R$ is a Tate ring and $\vp\in R$ is a topologically nilpotent unit, contained in some ring of definition $R_{0}$. If $a\in \R_{>1}$, then we may define a norm on $R$ by $|r|=\inf  \{a^{-n} \mid r\in \vp^{n}R_{0},\, n\in \Z \}$. Equipped with this norm, $R$ is a Tate normed ring with unit ball $R_{0}$ and $\vp$ is a multiplicative pseudo-uniformizer.

\medskip
\item A Banach--Tate ring $A$ is the same thing as a Banach algebra $A$ satisfying $|A^{m}|\neq 1$ in the language of Coleman \cite[\S 1]{co} (and what we call a Banach ring is what Coleman calls a Banach algebra). Here $A^{m}$ denotes the set of multiplicative units of $A$. Additionally, when $R$ is a Banach--Tate ring and $R^{+}$ is a ring of integral elements, a choice of a multiplicative pseudo-uniformizer $\vp$ may be used to identify the Gelfand spectrum $\mc{M}(R)$ of bounded multiplicative seminorms on $R$ (\cite[\S 1.2]{be}) with the maximal compact Hausdorff quotient of the adic spectrum $\Spa(R,R^{+})$ (\cite{hu1}); see \cite[Definition 2.4.6]{kl}. Concretely, $\vp$ gives us a natural way of viewing a rank $1$ point in $\Spa(R,R^{+})$ as a bounded multiplicative seminorm. 
\end{enumerate}
\end{remark}

\begin{definition}
Let $R$ be a normed ring. A \emph{normed $R$-module} is an $R$-module $M$ equipped with a function $||-||\, :\, M \ra \R_{\geq 0}$ such that (for all $m,n\in M$ and $r\in R$)
\begin{enumerate}
\item $||m||=0 \iff m=0$;

\item $||m+n|| \leq \max(||m||,||n||); $

\item $||rm|| \leq |r|.||m||.$
\end{enumerate}
We remark that if $r\in R$ is a multiplicative unit, then one sees easily that $||rm||=|r|.||m||$ for all $m\in M$. If $R$ is a Banach ring and $M$ is complete, we say that $M$ is a \emph{Banach $R$-module}.
\end{definition}

Let $R$ be a Tate normed ring. If $M$ and $N$ are normed $R$-modules, then a 
homomorphism $\phi\, :\, M \ra N$ is a continuous $R$-linear map. In this case, 
continuity of an $R$-linear map $\phi$ is equivalent to boundedness; i.e.~there 
exists $C\in \R_{>0}$ such that $||\phi(m)||\leq C||m||$ for all $m\in M$. In 
this case we set $|\phi|=\sup_{m\neq 0} |\phi(m)|.|m|^{-1}$ as usual; 
$\Hom_{R,cts}(M,N)$ becomes a normed $R$-module with respect to this norm. The 
open mapping theorem holds in this context; see e.g. \cite[Lemma 2.4(i)]{hu2}.

\medskip

Let $R$ be a Noetherian Banach--Tate ring. The results of \cite[\S 3.7.2]{bgr} hold in the context of Banach--Tate rings with the same proofs (thanks to the open mapping theorem), so $R$ being Noetherian is equivalent to all ideals being closed. Moreover, the results of \cite[\S 3.7.3]{bgr} hold for \emph{Noetherian} Banach--Tate rings with the same proofs. In particular, any finitely generated $R$-module carries a canonical complete topology, and any abstract $R$-linear map between two finitely generated $R$-modules is continuous and strict with respect to the canonical topology.

\begin{definition}
Let $R$ be a Banach--Tate ring and let $I$ be a set. We define $c_{R}(I)$ to be the set of sequences $(r_{i})_{i\in I}$ in $R$ tending to $0$ (with respect to the filter of subsets of $I$ with finite complement). It is a Banach $R$-module when equipped with the norm $||(r_{i})||=\sup_{i\in I}|r_{i}|$. 

We say that a Banach $R$-module $M$ is \emph{(potentially) orthonormalisable} (or \emph{(potentially) ON-able} for short) if there exists a set $I$ such that $M$ is $R$-linearly isometric (resp. merely $R$-linearly homeomorphic) to $c_{R}(I)$. A set in $M$ corresponding to the set $\{ e_{i}=(\delta_{ij})_{j} \mid i\in I \}\sub c_{R}(I)$ under such a map is called an \emph{(potential) ON-basis}.

Finally, we say that a Banach $R$-module $M$ has \emph{property $(Pr)$} if it is a direct summand of a potentially ON-able Banach $R$-module. 
\end{definition}

If $M \ra N$ is a continuous morphism of ON-able Banach $R$-modules, then we 
may define its matrix for a fixed ON-basis on $M$ and one on $N$ as on \cite[p. 
65]{bu}, and the properties stated there hold in this situation as well. A 
morphism $\phi : M \ra N$ between general Banach $R$-modules is said to be of 
\emph{finite rank} if the image of $\phi$ is contained in a finitely generated 
submodule of $N$. More generally, $\phi$ is said to be \emph{compact} (or 
\emph{completely continuous}) if it is a limit of finite rank operators in 
$\Hom_{R,cts}(M,N)$. If $R$ is Noetherian, \cite[Lemma 2.3, Proposition 
2.4]{bu} go through with the same proofs (using a multiplicative 
pseudo-uniformizer $\vp$ for what Buzzard calls $\rho$ in the proof of Lemma 
2.3) and we see that if $\phi : M\ra N$ is a continuous $R$-linear map between 
ON-able Banach $R$-modules with matrix $(a_{ij})$ with respect to some bases 
$(e_{i})_{i\in I}$ of $M$ and $(f_{j})_{j\in J}$ of $N$, then $\phi$ is compact 
if and only if $\lim_{j\ra \infty}\sup_{i\in I} |a_{ij}|=0$. When $M=N$ and 
$(e_{i})_{i\in I}=(f_{j})_{j\in J}$ this allows us to define the 
\emph{characteristic power series}, or \emph{Fredholm determinant}, 
$\det(1-T\phi)$ of a compact $\phi$ using the recipe on \cite[p. 67]{bu} and 
one sees that $\det(1-Tu)\in R\{\{T\}\}$, where $R\{\{T\}\}=\{ 
\sum_{n}a_{n}T^{n} \in R[[T]] \mid |a_{m}|M^{m} \ra 0\, \forall M\in \R_{\geq 
0} \}$ is the ring of entire power series in $R$.

\medskip

Moving on, we remark that \cite[Lemma 2.5, Corollary 2.6]{bu} are true in our 
setting with the same proofs. In particular, the notion of the Fredholm 
determinant extends to compact operators on potentially ON-able $M$, and may be 
computed using a potential ON-basis. It will be useful (at least 
psychologically) for us to know that these notions remain unchanged if we 
replace the norm on $R$ by an equivalent one. First, we remark that changing 
the norm on $R$ to an equivalent one doesn't change the topology on $c_{R}(I)$ 
(for $I$ arbitrary). This can be seen directly, but it is also a consequence of 
the following lemma, which we will need later.

\begin{lemma}\footnote{The published version of this paper contains an incorrect statement here. See \cite{erratum} for more details.}\label{norm}
Let $R$ be a complete Tate ring, and let $\vp,\pi\in R$ be topologically nilpotent units. Assume that we have two equivalent norms $|-|_{\vp}$ and $|-|_{\pi}$ on $R$ (inducing the intrinsic topology) such that $\vp$ is multiplicative for $|-|_{\vp}$ and $\pi$ is multiplicative for $|-|_{\pi}$. Then we may find constants $C_{1},C_{2},s>0$ such that, if $|a|_\pi < 1$,
\[
|a|_\vp \leq C_1,
\]
and if $|a|_\pi \geq 1$, then
\[
|a|_\vp \leq C_2|a|_\pi^s.
\]
\end{lemma}
\begin{proof}
For the first inequality, first note that we can find a constant $D_1 < 1$ such that if $|a|_\pi \leq D_1$ then $|a|_\vp \leq 1$ (since the norms induce the same topology). Choose an integer $m>0$ such that $|\vp^m|_\pi \leq D_1$ and set $C_1 = |\vp|_\vp^{-m}$. So, if $|a|_\pi \leq 1$, then $|\vp^m a|_\pi \leq D_1$ and hence $|\vp^m a|_\vp \leq 1$. Since $\vp$ is multiplicative for $|-|_\vp$, we deduce that $|a|_\vp \leq |\vp|_\vp^{-m} = C_1$.

\medskip

We now prove the second inequality. Assume that $|a|_\pi \geq 1$. Set
\[
n= \left\lceil \frac{\log |a|_{\pi}}{\log |\vp^{m}|_{\pi}^{-1}} \right\rceil 
;
\] 
then $n \geq 0$ by the assumption on $a$ and the definition of $m$. By definition, we have that $|\vp^{m}|_{\pi}^{n}|a|_{\pi} \leq 1$, and hence $|\vp^{mn}a|_{\pi} \leq 1$. Therefore we have $|\vp^{mn}a|_{\vp}\leq  C_1$. By multiplicativity of $\vp$ for $|-|_{\vp}$ we get $ |a|_{\vp} \leq C_1 |\vp|_{\vp}^{-mn} = C_1^{n+1}$. Then we have
\[
|a|_{\vp} \leq C_1^{n+1} \leq C_1^{\frac{\log |a|_{\pi}}{\log 
		|\vp^{m}|_{\pi}^{-1}}+2} = C_2 |a|_{\pi}^{s}
\]
where we have put $s =(\log_{C_1}|\vp^{m}|_{\pi}^{-1})^{-1}$ and $C_2 = C_1^2$; note that $s>0$. This finishes the proof. 
\end{proof}
If $R$ has two equivalent norms $|-|$ and $|-|^{\prime}$ with a common multiplicative pseudo-uniformizer $\vp$ such that $|\vp|=|\vp|^{\prime}$, then a similar argument shows that $|-|$ and $|-|^{\prime}$ are {bounded-equivalent}, as can be seen from the following lemma:
\begin{lemma}\label{norm2}
	Let $R$ be a complete Tate ring, and let $\vp\in R$ be a topologically nilpotent unit. Assume that we have two equivalent norms $|-|_1$ and $|-|_2$ on $R$ (inducing the intrinsic topology) such that $\vp$ is multiplicative for both $|-|_1$ and $|-|_2$. Then we may find constants $C_{1},C_{2} > 0$ such that
	\[
	C_1|a|_1^s \leq |a|_2 \leq C_2 |a|_1^s \]
	for all $a\in R$, where $s$ is determined by $|\varpi|_2 = |\varpi|_1^s$.
\end{lemma}
\begin{proof}
	We again fix a constant $D_1 < 1$ such that if $|a|_1 \leq D_1$ then $|a|_2 \leq 1$. Choose an integer $m>0$ such that $|\vp|^m_1 \leq D_1$. For a non-zero $a$, we set
	\[
	n= \left\lceil \frac{\log |a|_{1}}{\log |\vp^{m}|_{1}^{-1}} \right\rceil.
	\] 
Equivalently, $n$ is defined so that $|\vp^{-m(n-1)}|_1 < |a|_1 \leq |\vp^{-mn}|_1$. So $|a\varpi^{mn}|_1 \le 1$ and hence $|a|_2 \le |\varpi|_2^{-m(n+1)}$. Let $s$ be such that $|\varpi|_2 = |\varpi|_1^s$. Then $|a|_2 \le |\varpi|_1^{-m(n+1)s} \le C_2|a|_1^s$, where $C_2 = |\varpi|_1^{-2ms}$. Swapping $|-|_1$ and $|-|_2$ we get a similar inequality $|a|_1 \le C_1' |a|_2^{1/s}$. 
\end{proof}

Suppose then that $(M,|-|)$ is a Banach $(R,|-|)$-module, and $(M,|-|^{\prime})$ is a Banach $(R,|-|^{\prime})$-module, where $|-|$ and $|-|^{\prime}$ are equivalent on both $R$ and $M$. A Banach $(R,|-|)$-module isomorphism $(M,|-|) \cong (c_{R}(I),|-|)$ is then the same thing as a Banach $(R,|-|^{\prime})$-module isomorphism $(M,|-|^{\prime}) \cong (c_{R}(I),|-|^{\prime})$, since $(c_{R}(I),|-|) \cong (c_{R}(I),|-|^{\prime})$ as \emph{topological} $R$-modules via the identity map (one can argue as in Proposition \ref{normindep1}, or, perhaps better, formulate the basic notions of functional analysis introduced here in purely topological terms as in \cite[\S 2]{gulottapub}). Thus $(M,|-|)$ is potentially ON-able if and only if $(M,|-|^{\prime})$ is potentially ON-able, and $(e_{i})_{i\in I}$ is a potential ON-basis for $(M,|-|)$ if and only if it is a potential ON-basis for $(M,|-|^{\prime})$. It follows, at least when $R$ is Noetherian (which is all we need), that an operator $\phi$ is compact on a potentially ON-able $(M,|-|)$ if and only if it is compact on a potentially ON-able $(M,|-|^{\prime})$, and the Fredholm determinant is the same. 

\medskip

We remark that the results \cite[Lemma 2.7-Corollary 2.10]{bu} hold over Noetherian Banach--Tate rings $R$, again with the same proofs. We can extend the notion of Fredholm determinants of compact operators on Banach $R$-modules with property (Pr) as on \cite[p.72-73]{bu}, and the results there hold over Noetherian Banach--Tate rings. One also sees that having property (Pr) is stable when changing the norms on $(R,M)$ to equivalent ones, as is compactness of operators and the Fredholm determinants for compact operators are unchanged. We summarise the results of this section with the following Proposition:

\begin{proposition}
	Let $R$ be a Noetherian Banach--Tate ring. If $M$ is a Banach $R$-module with property $(Pr)$ and $\phi: M\rightarrow M$ is compact then there is a well-defined Fredholm determinant \[\det(1-T\phi|M) \in R\{\{T\}\}.\] If we change the norms on $(R,M)$ to equivalent ones, then $M$ still has property $(Pr)$, $\phi$ is still compact, and the Fredholm determinant is unchanged. 
\end{proposition}

\subsection{Riesz theory, slope factorizations and slope decompositions} We continue to let $R$ denote a Banach--Tate ring. If $Q\in R[T]$, we write $Q^{\ast}(T):=T^{\deg Q}Q(1/T)$. We recall the following definitions:
\begin{definition}\label{freddef}
 A \emph{Fredholm series} is a formal power series $F=1+\sum_{n \geq 1}a_{n}T^{n}\in R\{\{T\}\}$. A polynomial $Q\in R[T]$ is called \emph{multiplicative} if the leading coefficient of $Q$ is a unit (in other words, if $Q^{\ast}(0)\in R^{\times}$).
 
 Two entire series $P,Q \in  R\{\{T\}\}$ are said to be \emph{relatively prime} if the ideal $(P,Q) = R\{\{T\}\}$.
\end{definition}

The proof of \cite[Theorem 3.3]{bu} goes through without changes; we state it for completeness (see also \cite[Theoreme B.2]{aip}). Implicit in this is that \cite[Lemma 3.1]{bu} holds with the same proof; we will make use of this later.

\begin{theorem}\label{riesz}
Assume that $R$ is Noetherian. Let $M$ be a Banach $R$-module with property (Pr) and let $u : M \ra M$ be a compact operator with $F=det(1-Tu)$. Assume that we have a factorization $F=QS$ where $S$ is a Fredholm series, $Q\in R[T]$ is a multiplicative polynomial, and $Q$ and $S$ are relatively prime in $R\{\{T\}\}$. Then $\Ker Q^{\ast}(u)\sub M$ is finitely generated and projective and has a unique $u$-stable closed complement $N$ such that $Q^{\ast}(u)$ is invertible on $N$. The idempotent projectors $M \ra \Ker Q^{\ast}(u)$ and $M \ra N$ lie in the closure of $R[u]\sub \End_{R,cts}(M)$. The rank of $\Ker Q^{\ast}(u)$ is $\deg Q$, and $\det(1-Tu\mid \Ker Q^{\ast}(u))=Q$. Moreover, $u$ is invertible on $\Ker Q^{\ast}(u)$, and $\det(1-Tu\mid N)=S$.
\end{theorem}

\begin{proof}
Apart from the last sentence, this is (a minor reformulation of) \cite[Theorem 3.3]{bu}. To see that $u$ is invertible on $\Ker Q^{\ast}(u)$, note that $\det(u\mid \Ker Q^{\ast}(u))=Q^{\ast}(0)\in R^{\times}$. To see that $\det(1-Tu \mid N)=S$, write $S^{\prime}=\det(1-Tu \mid N)$ and note that $F=\det(1-Tu \mid \Ker Q^{\ast}(u))\det(1-Tu \mid N)=QS^{\prime}$. Hence $Q(S-S^{\prime})=0$, and $Q$ is not a zero divisor since $Q(0)=1$, so $S=S^{\prime}$.
\end{proof}

The following Lemma may be extracted from the proof of \cite[Lemma 5.6]{bu}; we give the short proof for completeness.

\begin{lemma}\label{link}
Assume that $R$ is Noetherian. Let $M$ and $M^{\prime}$ be two Banach $R$-modules with property (Pr) and assume that we have a continuous $R$-linear map $v : M \ra M^{\prime}$ and a compact $R$-linear map $i : M^{\prime} \ra M$. Set $u=iv$ and $u^{\prime}=vi$. Then $u$ and $u^{\prime}$ are both compact and $\det(1-Tu)=\det(1-Tu^{\prime})$; call this entire power series $F$. If $F=QS$ is a factorization as in Theorem \ref{riesz}, then $i$ restricts to an isomorphism between $\Ker Q^{\ast}(u^{\prime})$  and $\Ker Q^{\ast}(u)$.
\end{lemma}

\begin{proof}
Compactness of $u$ and $u^{\prime}$ and the equality of their Fredholm determinants follows from \cite[Proposition 2.7]{bu}. Now assume we have a factorization $F=QS$. If $x^{\prime}\in \Ker Q^{\ast}(u^{\prime})$, then $Q^{\ast}(u)(i(x^{\prime}))=i(Q^{\ast}(u^{\prime})(x^{\prime}))=0$ so $i(\Ker Q^{\ast}(u^{\prime}))\sub \Ker Q^{\ast}(u)$. Furthermore, if $i(x^{\prime})=0$ then $u^{\prime}(x^{\prime})=0$, so $Q^{\ast}(u^{\prime})(x^{\prime})=Q^{\ast}(0).x^{\prime}=0$ and hence $x^{\prime}=0$, so $i$ is injective on $\Ker Q^{\ast}(u^{\prime})$. For surjectivity onto $\Ker Q^{\ast}(u)$, let $x\in \Ker Q^{\ast}(u)$ and choose $y\in \Ker Q^{\ast}(u)$ with $u(y)=x$ (possible by Theorem \ref{riesz}). Then one checks, similarly to the computation above, that $v(y)\in \Ker Q^{\ast}(u^{\prime})$, and hence $i(v(y))=u(y)=x$, which gives us surjectivity and finishes the proof. \end{proof}

Next, we let $K$ be a field, complete with respect to a non-trivial non-archimedean absolute value. We briefly define the Newton polygon of a power series $F\in K[[T]]$, following \cite[\S 4.2]{as} (in this special case). A subset $\mc{N}\sub \R^{2}$ is said to be \emph{sup-convex} if it is convex and, whenever a point $(a,b)\in \mc{N}$, $\mc{N}$ contains the whole half-line $\{(a,b+t)\mid t\geq 0\}$ above it. Given an arbitrary subset $S\sub \R^{2}$, there is a unique smallest sup-convex set containing $S$, which we will denote by $\mc{H}_{+}(S)$. If $I\sub \Z_{\geq 0}$ and $\omega\, :\, I \ra \R$ is a function, then any set of the form $\mc{H}_{+}(\{(n,\omega(n))\mid n\in I\}$ is called a \emph{Newton polygon}. We refer to \cite[\S 4.2]{as} for the notions of vertices, edges and slopes of a Newton polygon.

\begin{definition}
Let $F=\sum_{n\geq 0} a_{n}T^{n}\in K[[T]]$. Fix a pseudo-uniformizer $\vp\in K$ and consider the corresponding valuation $v_{\vp}$. The \emph{Newton polygon of $F$} is the Newton polygon $\mc{H}_{+}(\mc{S}(F))$, where
$$ \mc{S}(F)=\{(n,v_{\vp}(a_{n})) \mid n\in I_{F} \}\sub \R^{2}$$
with $I_{F}=\{ n\in \Z_{\geq 0} \mid a_{n}\neq 0 \}$.
\end{definition}

Let $h\in \R$. We say that a power series $F\in K[[T]]$ has slope $\leq h$ (or $>h$) if all slopes of its Newton polygon are $\leq h$ (or $>h$). Now consider a Banach--Tate ring $R$ with a multiplicative pseudo-uniformizer $\vp$. We say that $F\in R[[T]]$ has slope $\leq h$ (or $>h$) if, for any $x$ in the Gelfand spectrum $\mc{M}(R)$ with residue field $K_{x}$, the specialization $F_{x}\in K_{x}[[T]]$ has slope $\leq h$ (or $>h$).

\begin{definition}
Let $R$ be a Banach--Tate ring with a fixed multiplicative pseudo-uniformizer $\vp$. Let $F\in R\{\{T\}\}$ be a Fredholm series and let $h\in \R$. A \emph{slope $\leq h$-factorization} of $F$ is a factorization $F=QS$ in $R\{\{T\}\}$ where $Q$ is a multiplicative polynomial of slope $\leq h$ and $S$ is a Fredholm series of slope $>h$.
\end{definition}
\begin{remark}\label{slopetop}
If $R$ is a complete Tate ring with a fixed topologically nilpotent unit $\vp$, then the notions of slope factorizations and slope $\leq h$ or $> h$ are independent of the choice of a norm on $R$ with $\vp$ multiplicative. Moreover, one can define all these notions directly without choosing a norm on $R$.
\end{remark}

Recall that an element $a\in R$ is called \emph{quasinilpotent} if its spectral 
seminorm\footnote{The definition of the spectral seminorm is recalled in 
Appendix \ref{app}.} $|a|_{sp}$ is $0$. This is equivalent to $|a|_{x}=0$ for 
all $x\in 
\mc{M}(R)$ by \cite[Corollary 1.3.2]{be}. The set of quasinilpotent elements 
form an ideal of $R$, which is the kernel of the Gelfand transform $R \ra 
\prod_{x\in \mc{M}(R)}K_{x}$. We note that it is easy to see that a quasinilpotent element is topologically nilpotent. For the kinds of ring $R$ which 
appear in practice in this paper, the quasinilpotent elements are just the 
nilpotent elements (this follows from Theorem \ref{su3}), and the proof of the 
following lemma is simpler. However, we will avoid imposing 
additional technical assumptions at this stage.

\begin{lemma}\label{coprime}
Let $R$ be a Banach--Tate ring with a fixed multiplicative pseudo-uniformizer $\vp$ and let $h\in \Q_{\geq 0}$. Let $S$ be a Fredholm series of slope $>h$ and $Q$ a multiplicative polynomial of slope $\leq h$. Then $Q$ and $S$ are relatively prime.
\end{lemma}

\begin{proof}
We will use Coleman's resultant $\Res$, for which we refer to \cite[\S A3]{co} (the reader may also benefit from the discussion on \cite[p. 74]{bu}). By \cite[Lemma A3.7]{co} it suffices to prove that $\Res(Q,S)$ is a unit in $R\{\{T\}\}$. Pick $x\in \mc{M}(R)$ and specialize to $K_{x}$. Then $\Res(Q,S)_{x}=\Res(Q_{x},S_{x})$ and since $Q_{x}$ has slope $\leq h$ and $S_{x}$ has slope $>h$ we see that $\Res(Q,S)_{x}\in K_{x}\{\{T\}\}^{\times}=K_{x}^{\times}$. By \cite[Corollary 1.2.4]{be} we see that $\Res(Q,S)=a_{0}+T.F(T)$ where $a_{0}\in R^{\times}$ and $F(T)\in R\{\{T\}\}$ has quasinilpotent coefficients. Multiplying by $a_{0}^{-1}$ we see that it suffices to prove that if $F\in R\{\{T\}\}$ has quasinilpotent coefficients, then $1-T.F(T)\in R\{\{T\}\}^{\times}$.

\medskip

To prove this, we use an argument suggested to us by a referee, which is more efficient than our original argument. First note that the formal inverse of $1-T.F(T)$ is 
$G(T)=\sum_{n\geq 0}T^{n}.F(T)^{n}$, so we need to show that this is entire. Setting $H(T)=T.F(T)$, it suffices to show that if $H(T)$ is any entire power series with $H(0)=0$ and with quasinilpotent coefficients, then $\sum_{n\geq 0}H(T)^n $ converges and is entire. In fact, it suffices to prove that $\sum_{n\geq 0}H(T)^n \in R\langle T \rangle$, since if we have proved this we can apply this to $H(\vp^{-N}T)$ for all $N$ (which still has quasinilpotent coefficients) to deduce that $\sum_{n\geq 0}H(T)^n$ is entire.

\medskip

So, to show this, it suffices to show that $H(T)$ is topologically nilpotent in $R\langle T \rangle$, where we equip $R\langle T \rangle$ with the Gauss norm coming from the norm on $R$. Since the coefficients of $H$ tend to $0$, we may write $H(T)=\sum_{n=1}^{N}h_n T^n + \sum_{n>N}h_n T^n$, with $|h_n|<1$ for $n>N$. Then the tail $\sum_{n>N}h_n T^n$ is topologically nilpotent (it has Gauss norm $<1$) and the terms $h_n T^n$ are topologically nilpotent for all $n$ since $h_n$ is quasinilpotent (and hence topologically nilpotent). So $H(T)$ is a finite sum of topologically nilpotent elements, hence topologically nilpotent. This finishes the proof.
\end{proof}

Continue to let $R$ be a Banach--Tate ring with a fixed multiplicative pseudo-uniformizer $\vp$. The following is a minor variation of \cite[Definition 4.6.1]{as}.

\begin{definition}
Let $M$ be an (abstract) $R$-module, let $u\, :\, M \ra M$ be an $R$-linear map and let $h\in \Q$. An element $m\in M$ is said to have slope $\leq h$ with respect to $u$ if there is a multiplicative polynomial $Q\in R[T]$ such that
\begin{enumerate}
\item $Q^{\ast}(u).m=0$;

\item The slope of $Q$ is $\leq h$.
\end{enumerate}
We let $M_{\leq h}\sub M$ denote the subset of elements of slope $\leq h$.
\end{definition}

\begin{lemma}(\cite[Proposition 4.6.2]{as})
$M_{\leq h}$ is an $R$-submodule of $M$, which is stable under $u$.
\end{lemma}

\begin{proof}
It is clear from the definition that $M_{\leq h}$ is closed under multiplication, and stable under $u$. It therefore suffices to prove that it is closed under addition, for which it suffices to prove that if $Q_{1}$ and $Q_{2}$ are two multiplicative polynomials of slope $\leq h$, then so is $Q_{1}Q_{2}$. To see this it suffices to specialize to the case when $R$ is a field and the norm is an absolute value. The assertion is then well known (for example, the argument in the proof of  \cite[Proposition 4.6.2]{as} carries over without change).
\end{proof}

\begin{definition}(\cite[Definition 4.6.3]{as})
Let $M$ be an $R$-module with an $R$-linear map $u\, :\, M \ra M$ and let $h\in \Q$. A slope $\leq h$-decomposition of $M$ is an $R[u]$-module decomposition $M=M_{h}\oplus M^{h}$ such that
\begin{enumerate}
\item $M_{h}$ is a finitely generated $R$-submodule of $M_{\leq h}$;

 \item For every multiplicative polynomial $Q\in R[T]$ of slope $\leq h$, the map $Q^{\ast}(u)\, : \, M^{h} \ra M^{h}$ is an isomorphism of $R$-modules.
\end{enumerate}
\end{definition}

\begin{proposition} We keep the above notation. If $M$ has a slope $\leq h$-decomposition $M_{h}\oplus M^{h}$, then it is unique, and $M_{h}=M_{\leq h}$ (in particular the latter is finitely generated over $R$). We will from now on write $M_{>h}$ for the unique complement. Moreover, slope decompositions satisfy the following functorial properties:

\begin{enumerate} 
\item Let $f\, :\, M \ra N$ be a morphism of $R[u]$-modules with slope $\leq h$-decompositions. Then $f(M_{\leq h})\sub N_{\leq h}$ and $f(M_{>h})\sub N_{>h}$. Moreover, both ${\rm Ker}(f)$ and ${\rm Im}(f)$ have slope $\leq h$-decompositions.

\item Let $C^{\bu}$ be a complex of $R[u]$-modules and suppose that each $C^{i}$ has a slope $\leq h$-decomposition. Then every $H^{i}(C^{\bu})$ has a slope $\leq h$-decomposition, explicitly given by $H^{i}(C^{\bu})=H^{i}(C^{\bu}_{\leq h}) \oplus H^{i}(C^{\bu}_{>h})$.
\end{enumerate}
\end{proposition}

\begin{proof}
The proof is identical to that of \cite[Lemma 4.6.4]{as}: one equates slope $\leq h$-decompositions with $\mc{S}$-decompositions (as defined and studied in \cite[\S 4.1]{as})) for the set $\mc{S}\sub R[u]$ of all $Q^{\ast}(u)$ where $Q$ is a multiplicative polynomial of slope $\leq h$. The properties then stated follow from general facts about $\mc{S}$-decompositions, recorded in \cite[Proposition 4.1.2]{as}.
\end{proof}

\begin{definition}
Let $R$ be a Banach--Tate ring with a fixed multiplicative pseudo-uniformizer $\vp$ and let $M$ be a Banach $R$-module. Assume that $M$ has a slope-$\leq h$ decomposition $M=M_{\leq h}\oplus M_{>h}$. If $f : R\ra S$ is a bounded morphism of Banach--Tate rings such that $f(\vp)$ is a multiplicative pseudo-uniformizer in $S$, we say that the slope-$\leq h$ decomposition is \emph{functorial for $R\ra S$} if the decomposition $M\ctens_{R}S = (M_{\leq h} \otimes_{R}S) \oplus (M_{>h} \ctens_{R}S)$ is a slope-$\leq h$ decomposition of $M \ctens_{R}S$ (using $f(\vp)$ to define slopes for $S$). We say that the slope-$\leq h$ decomposition is \emph{functorial} if it is functorial for all such bounded homomorphisms of Banach--Tate rings out of $R$.
\end{definition}

\begin{theorem}
Let $R$ be a Noetherian Banach--Tate ring with a fixed multiplicative pseudo-uniformizer $\vp$, and let $M$ be a Banach $R$-module with property (Pr). Let $u$ be a compact $R$-linear operator on $M$, with Fredholm determinant $F(T)=\det(1-Tu)$. If $M$ has a slope $\leq h$-decomposition which is functorial with respect to $R \ra K_{x}$ for all $x\in \mc{M}(R)$, then $F$ has a slope $\leq h$-factorization. Conversely, if $F$ has a slope-$\leq h$ factorization, then $M$ has a functorial slope-$\leq h$ decomposition.
\end{theorem}

\begin{proof}
Assume that $M$ has a slope-$\leq h$ decomposition $M=M_{\leq h} \oplus M_{>h}$ 
with is functorial with respect to $R \ra K_{x}$ for all $x\in \mc{M}(R)$. Then 
both of these spaces satisfy property (Pr) and are $u$-stable, and hence we have
$$ F = \det(1-Tu \mid M_{\leq h})\det(1-Tu \mid M_{>h}). $$
We claim that this is a slope-$\leq h$ factorization.  Put $Q=\det(1-Tu \mid M_{\leq h})$, $S=\det(1-Tu \mid M_{>h})$. Pick $x\in \mc{M}(R)$ with residue field $K_{x}$ and specialize. We have $Q_{x}=\det(1-Tu \mid M_{\leq h}\otimes_{R}K_{x})$ and $S_{x}=\det(1-Tu \mid M_{>h}\ctens_{R}K_{x})$. By assumption $M\ctens_{R}K_{x} = (M_{\leq h} \otimes_{R}K_{x}) \oplus (M_{>h}\ctens_{R}K_{x})$ is a slope-$\leq h$ decomposition, so $Q_{x}$ has slopes $\leq h$ and $S_{x}$ has slopes $>h$ and so $F=QS$ is a slope-$\leq h$ factorization. 

\medskip

Conversely, assume that $F$ has a slope-$\leq h$ factorization $F=QS$. By Lemma \ref{coprime} $Q$ and $S$ are relatively prime, so we may apply Theorem \ref{riesz} to get a $u$-stable decomposition $M = \Ker Q^{\ast}(u) \oplus N$. It is easy to see that this decomposition is functorial, as is a slope-$\leq h$ factorization, so it suffices to prove that this decomposition is a slope-$\leq h$ decomposition. First, since $Q$ has slope $\leq h$ we see that $\Ker Q^{\ast}(u)\sub M_{\leq h}$ (and we know it's finitely generated). It remains to show that for every multiplicative polynomial $P$ of slope $\leq h$, $P^{\ast}(u)$ is invertible on $N$. By Lemma \ref{coprime} $P$ and $S$ are relatively prime. Since $S=\det(1-Tu\mid N)$ (by Theorem \ref{riesz}), it follows from \cite[Lemma 3.1]{bu} that $P^{\ast}$ is invertible on $N$, as desired.
\end{proof}

\begin{corollary}
With notation and assumptions as in the theorem, a slope-$\leq h$ decomposition of $M$ is functorial if and only if it is functorial for the natural map $R \ra K_{x}$ for all $x\in \mc{M}(R)$. 
\end{corollary}

\subsection{Fredholm hypersurfaces}

 In this section we discuss the notion of Fredholm hypersurfaces and relate this to slope factorizations and decompositions. We will use Huber's adic spaces as our framework for non-archimedean geometry, and we will use standard notions and notation from this theory freely, refering to the basic references \cite{hu2, hu3}. 

\medskip

Any Tate ring $R$ with a Noetherian ring of definition has an associated 
affinoid adic space $\Spa(R,R^{+})$, for any ring of integral elements $R^{+}$, 
by \cite[Theorem 2.5]{hu2}. Fix an $R^{+}$ and consider $X=\Spa(R,R^{+})$. We 
will frequently be interested in affine $1$-space over $X$. As an adic space 
over $(\Z,\Z)$, we have $\A^{1}=\Spa(\Z[T],\Z)$; it represents the functor $X 
\mapsto \oo(X)$ on the category of adic spaces (we note that the functor $X 
\mapsto \oo^{+}(X)$ is represented by the `closed unit disc' 
$\Spa(\Z[T],\Z[T])$). The fibre product $\A^{1}_{X}:=X 
\times_{\Spa(\Z,\Z)}\A^{1}$ exists, but it is no longer affinoid. Indeed, 
if we pick a topologically nilpotent unit $\vp\in R$, it can be checked that 
the fibre product is 
given by
$$ \A_{X}^{1} = \bigcup_{m\geq 0} \Spa(R \langle \vp^{m}T \rangle, R^{+}\langle \vp^{m}T\rangle ) $$
with respect to the transition maps coming from the natural inclusions. The ring of global functions on $\A^{1}_{X}$ is the ring of entire power series $R\{\{T\}\}$. Pick a topologically nilpotent unit $\vp\in R$. If $h\in \Q$ then, writing $h=m/n$ with $m\in \Z$ and $n\in \Z_{\geq 1}$, we define an affinoid subset $\mb{B}_{X,h}\sub \mb{A}^{1}_{X}$ by
$$ \mb{B}_{X,h}=\{ |T^{n}|\leq |\vp^{-m}| \} \sub \mb{A}^{1}_{X}. $$
We have $\mb{A}^{1}_{X}=\bigcup_{h\in \Q} \mb{B}_{X,h}$. 

Let $R$ be a complete Tate ring with a Noetherian ring of definition and let $F\in R\{\{T\}\}$ be a Fredholm series. Put $X=\Spa(R,R^{\circ})$. The closed subvariety $Z(F):=\{F=0\}\sub \A^{1}_{X}$ is called the \emph{Fredholm hypersurface} of $F$, or sometimes the \emph{spectral variety} of $F$. It carries a projection map $Z(F)\ra X$, which is flat, locally quasi-finite and partially proper by \cite[Theoreme B.1]{aip}.

\begin{definition}
Let $R$ be a complete Tate ring with a Noetherian ring of definition, and pick a topologically nilpotent unit $\vp\in R$. Let $F$ be a Fredholm series with Fredholm hypersurface $Z=Z(F)\sub \A^{1}_{X}$, where $X=\Spa(R,R^{\circ})$. Let $h\in \mb{Q}_{\geq 0}$ and let $U\sub X$ be an open affinoid in $X$; put $Z_{U,h}=Z\cap \mb{B}_{U,h}\sub \A^{1}_{X}$ (this is an open affinoid subset of $Z$). We say that the pair $(U,h)$ is a \emph{slope datum} for $(X,F)$ if $Z_{U,h} \ra U$ is finite of constant degree (if the pair $(X,F)$ is clear form the context, we occasionally just say that $(U,h)$ is a slope datum).
\end{definition}

\begin{theorem}\label{slopeadapted}
Let $R$ be a complete Tate ring with a Noetherian ring of definition, and pick a topologically nilpotent unit $\vp\in R$. Let $F$ be a Fredholm series over $R$ with spectral variety $Z=Z(F)\sub \mb{A}^{1}_{X}$, where $X=\Spa(R,R^{\circ})$. Let $U\sub X$ be an open affinoid and let $h\in \Q_{\geq 0}$. Then:

\begin{enumerate}
\item $(U,h)$ is a slope datum for $(X,F)$ if and only if $F$ has a slope $\leq h$-factorization in $\oo_{X}(U)\{\{T\}\}$.

\item The collection of all $Z_{U,h}$ for all slope data $(U,h)$ is an open cover of $Z$.
\end{enumerate}
\end{theorem}

\begin{proof}
This follows almost directly from \cite[Theoreme B.1, Corollaire B.1]{aip}. The second assertion is shown by tracing through the proof of \cite[Lemma B.1, Theoreme B.1]{aip}; adapting the proof of \cite[Theoreme B.1]{aip} slightly one sees that one may take the sets to be of the form $Z_{U,h}$ (the degree is locally constant on $U$, so constancy of the degree can be arranged). For the first assertion, the statement that if $(U,h)$ is a slope datum then $F$ has a slope-$\leq h$ factorization in $\oo_{X}(U)\{\{T\}\}$ is \cite[Corollaire B.1]{aip}. Conversely, if $F=QS$ is a slope-$\leq h$ factorization in $\oo_{X}(U)\{\{T\}\}$, then $S$ is a unit in $\oo(\mb{B}_{U,h})$. Therefore $\oo(Z_{U,h})=\oo(\mb{B}_{U,h})/(F)= \oo(\mb{B}_{U,h})/(Q)$ is finite of constant degree equal to $\deg Q$ over $U$, and hence $Z_{U,h} \ra U $ is finite of constant degree. 
\end{proof}

More generally, let $X$ be an analytic adic space locally of the form $\Spa(R,R^{\circ})$ for $R$ a complete Tate ring with a Noetherian ring of definition, and let $F$ be a Fredholm series over $X$ with Fredholm hypersurface $Z$. If $U\sub X$ is an open affinoid and $h\in \Q_{\geq 0}$, we say that $(U,h)$ is a slope datum for $(X,F)$ if $\oo(U)$ is Tate and there is a topologically nilpotent unit $\vp\in\oo(U)$ such that $Z_{U,h}$, defined using this choice of $\vp$, is finite flat of constant degree over $U$.  

\medskip

When constructing eigenvarieties, it will be useful to consider a slightly more general notion. Let $X$ be an analytic adic space as above and let $F$ be a Fredholm series over $X$, with associated hypersurface $Z$. Write $\pi : Z \ra X$ for the projection. We let  $\ms{C}ov(Z)$ denote the set of all open affinoid $V\sub Z$ such that $\pi(V)\sub X$ is open affinoid, $\oo(\pi(V))$ is Tate, and the map $\pi|_{V} : V \ra \pi(V)$ is finite of constant degree.  Then we have the following theorem.

\begin{theorem}\label{cover}
Keep the notation and assumptions of the paragraph above. Then $\ms{C}ov(Z)$ is an open cover of $Z$. If $V\in \ms{C}ov(Z)$, then there exists a factorization $F=QS$ in $\oo(\pi(V))\{\{T\}\}$, where $Q$ is a multiplicative polynomial of degree $\deg \pi|_{V}$, $S$ is a Fredholm series, $Q$ and $S$ are relatively prime, and we have $\oo(V)=\oo(\pi(V))[T]/(Q)$ and $\oo^+(V)=(\oo(\pi(V))[T]/(Q))^\circ$. Conversely, if such a factorization of $F$ exists in $\oo(U)\{\{T\}\}$, where $U\sub X$ is open affinoid and $\oo(U)$ is Tate, then $V=\Spa(\oo(U)[T]/(Q), (\oo(U)[T]/(Q))^{\circ})$ is naturally an element of $\ms{C}ov(Z)$. 
\end{theorem}  

\begin{proof}
The assertions about rings of integral elements follow immediately from the rest by Lemma \ref{finitepowerbounded}. The first two parts are \cite[Theoreme B.1, Corollaire B.1]{aip}. Note that if $\pi: V \rightarrow U$ is a finite flat morphism, with $U \subset X$ and $V \subset Z$ open affinoid, then $\pi$ is open by \cite[Lemma 1.7.9]{hu3}.  For the last part, it is clear that $V \ra U$ is finite and surjective of constant degree $\deg Q$, so it remains to see that $V$ is naturally an open subset of $Z$. For this we may work locally over $U$. Set $B=\oo(U)$. For each $n$ we have compatible morphisms 
$$ B[T]/(Q) \ra B \langle \vp^{n}T \rangle / (Q) \leftarrow B \langle \vp^{n}T \rangle /(F). $$
The second map is the projection onto the first factor in the decomposition $B \langle \vp^{n}T \rangle /(F) \cong B \langle \vp^{n}T \rangle /(Q) \times B \langle \vp^{n}T \rangle /(S) $ which results from the fact that $Q$ and $S$ are relatively prime. Thus $\{Q=0\}\sub Z$ is open and closed in $Z$. Moreover, when $n$ is sufficiently large, we claim that the first map is an isomorphism. To see this, consider the quotient map $p : B[T] \ra B[T]/(Q)$ and equip the target with a submultiplicative norm that induces the canonical topology. For large enough $n$ we will have $|p(\vp^{n}T)|\leq 1$ and hence $p(B_{0}[\vp^{n}T])\sub (B[T]/(Q))_{0}$ (here we are using $-_{0}$ to denote unit balls), so $p$ is continuous for the topology on $B[T]$ coming from the inclusion $B[T]\sub B\langle \vp^{n}T \rangle$, and we may complete to obtain a morphism $B\langle \vp^{n}T \rangle \ra B[T]/(Q)$ with kernel $QB\langle \vp^{n}T \rangle$. This gives an inverse to the map $B[T]/(Q) \ra B \langle \vp^{n}T \rangle / (Q)$, proving the claim. Thus we may identify $V$ with $\{Q=0\}\sub Z$, which shows that $V$ is naturally an open subset of $Z$.
\end{proof}

\section{Relative distribution algebras}\label{distalg}

\subsection{Relative distribution algebras and norms}

A $p$-adic analytic group will in this paper always mean a $\Qp$-analytic group. Let $R$ be a Banach--Tate ring. We denote the unit ball of $R$ by $R_{0}$. If there exists a norm-decreasing homomorphism $\Zp\ra R$, where we equip $\Zp$ with the usual norm $|x|_{p}=p^{-\mathrm{ord}_{p}(x)}$, we call such $R$ (together with the map $\Zp \ra R$) a \emph{Banach--Tate $\Zp$-algebra}. The goal of this section is to extend some of the constructions of \cite[\S 4]{st} to the case of continuous functions and distributions valued in such $R$. In particular, we construct $R$-valued analogues of (locally) analytic distribution algebras for compact $p$-adic analytic groups. We begin with a lemma on the existence of Banach--Tate $\Zp$-algebra norms.

\begin{lemma}\label{BTZp}
Let $R$ be a Noetherian Banach--Tate ring with norm $|-|$ and a multiplicative pseudo-uniformizer $\vp$. Assume that there exists a continuous homomorphism $\Zp \ra R$ (necessarily unique). Then there exists a Banach-Tate $\Zp$-algebra norm $|-|^{\prime}$ on $R$ which is bounded-equivalent to $|-|^{s}$ for some $s>0$, and such that $\vp$ is a multiplicative pseudo-uniformizer for $|-|^{\prime}$.
\end{lemma}

\begin{proof}
Note first that a norm $|-|^{\prime}$ on $R$ is a Banach--Tate $\Zp$-algebra norm if and only if $|p|^{\prime} \leq p^{-1}$, so we need to check this. By continuity of $\Zp \ra R$ we have $p\in R^{\circ\circ}$. Choose $m\in \Z_{\geq 1}$ such that $|p|_{sp}< |\vp|^{2/m}$ and consider the finite free $R$-algebra $S=R[\vp^{1/m}]=R[X]/(X^{m}-\vp)$. We equip $S$ with its canonical topology as a finite $R$-module; then the induced subspace topology on $R\sub S$ agrees with the original topology on $S$. Thus we have $p,\vp^{1/m}\in S^{\circ\circ}$. Now equip $S$ with a submultiplicative $R$-Banach module norm $|-|_{S}$ that induces the canonical topology. Note that $\vp$ is a multiplicative pseudo-uniformizer for $|-|_{S}$ with $|\vp|_{S}=|\vp|$, and that $(R,|-|) \ra (S,|-|_{S})$ is norm-decreasing. We then have $|p\vp^{-1/m}|_{S,sp}<|\vp|^{1/m}<1$ by construction, so $p\vp^{-1/m}$ is topologically nilpotent in $S$. We can then choose a ring of definition $S_{2}$ of $S$ containing $\vp^{1/m}$ and $p\vp^{-1/m}$ and consider the norm
$$ |s|_{2}=\inf \{ |\vp|^{k/m} \mid s \in \vp^{k/m}S_{2} \}. $$
Since $p\in \vp^{1/m}S_{2}$ we have $|p|_{2}<1$, and we may hence find $s>0$ such that $|p|_{2}^{s}\leq p^{-1}$. Restricting the norm $|-|^{\prime}:=|-|_{2}^{s}$ to $R\sub S$ then gives the desired norm.
\end{proof}

\begin{definition}
Let $X$ be a compact topological space and let $A$ be a topological $\Zp$-algebra.
\begin{enumerate}
\item We let $\mc{C}(X,A)$ denote the $A$-module of all continuous $A$-valued functions on $X$ , and let $\mc{C}_{sm}(X,A)$ denote the subspace of all locally constant functions.

\item We put $\D(X,A)=\Hom_{A,cts}(\mc{C}(X,A),A)$.

\end{enumerate}
\end{definition}

When $A$ is a normed ring, we may topologize $\mc{C}(X,A)$ and $\mc{C}_{sm}(X,A)$ using the supremum norm, and we may give $\D(X,A)$ the corresponding dual/operator norm. If $A$ is complete, this makes $\mc{C}(X,A)$ into a complete $A$-module. When the topology on $X$ is profinite, $\mc{C}_{sm}(X,A)$ is dense in $\mc{C}(X,A)$ and, if $R$ is a Banach--Tate $\Zp$-algebra, the natural map $\mc{C}(X,\Zp) \ctens_{\Zp}R \ra \mc{C}(X,R)$ is a topological isomorphism. Similarly $\mc{C}(X,\Zp) \ctens_{\Zp}R_{0} \cong \mc{C}(X,R_{0})$ where $R_{0}$ is the unit ball of $R$.

\medskip

Continue to let $X$ be a profinite set and $R$ a Banach--Tate $\Zp$-algebra with unit ball $R_{0}$ and a multiplicative pseudo-uniformizer $\vp$. Note that $\D(X,R_{0})=\Hom_{R_{0}}(\mc{C}(X,R_{0}),R_{0})$ (i.e. continuity with respect to the $\vp$-adic topology is automatic) and that this is the unit ball in $\D(X,R)$. We may equip $\D(X,R_{0})$ with the weak topology coming from the family of maps $\D(X,R_{0}) \ra R_{0}$ given by $\mu \mapsto \mu(f)$ for $f\in \mc{C}(X,R_{0})$ and the $\vp$-adic topology on $R_{0}$. We will refer to this topology as the \emph{weak-star topology} on $\D(X,R_{0})$. When $X=G$ is a profinite group, $\D(G,R)$ carries a convolution product
$$ (\mu \ast \nu)(f) = \mu(g \mapsto \nu(h\mapsto f(gh))).$$
One checks directly that $\delta_{g}\ast\delta_{h}=\delta_{gh}$ for all $g,h\in G$, where $\delta_{g}$ denotes the Dirac distribution at $g$. This product preserves $\D(G,R_{0})$. We sum up some basic properties of the weak-star topology.

\begin{lemma}\label{weak-star}
If $X$ is finite (hence discrete) the weak-star topology on $\D(X,R_{0})$ coincides with the $\vp$-adic topology. In general, if $X=\varprojlim_{n}X_{n}$ is an inverse limit of finite sets $X_{n}$, we have a natural isomorphism $\D(X,R_{0}) \cong \varprojlim_{n}\D(X_{n},R_{0})$ which identifies the weak-star topology on the source with the inverse limit topology where $\D(X_{n},R_{0})$ is equipped with the $\vp$-adic topology. When $X=G$ is a profinite group this is a ring homomorphism, and multiplication on $\D(G,R_{0})$ is jointly continuous with respect to the weak-star topology.
\end{lemma}

\begin{proof}
The first assertion is straightforward. For the second, note that the formation of $\D(X,R_{0})$ is covariantly functorial in $X$, so the maps $X \ra X_{n}$ induce a natural map $\D(X,R_{0}) \ra \varprojlim_{n}\D(X_{n},R_{0})$ which is continuous by the first assertion when we equip the source and the target with the topologies in the statement of the lemma. Moreover it is easily checked to be a ring homomorphism when $X=G$ is a profinite group. Unraveling, we see that this morphism is the natural morphism
$$ \Hom_{R_{0}}(\mc{C}(X,R_{0}),R_{0}) \ra \Hom_{R_{0}}(\mc{C}_{sm}(X,R_{0}),R_{0}) $$
induced by the inclusion $\mc{C}_{sm}(X,R_{0}) \sub \mc{C}(X,R_{0})$. Since this subspace is dense for the $\vp$-adic topology, we see that the map is an isomorphism. To check that it also a homeomorphism, note first that by the same density one may define the weak-star topology using only locally constant functions. It is then straightforward to check that all basic opens from locally constant functions come by pullback from basic opens on the $\D(X_{n},R_{0})$, which implies that the map is a homeomorphism. Finally, multiplication is jointly continuous for the $\vp$-adic topology on $\D(G_{n},R_{0})$ for all $n$, and hence jointly continuous on the inverse limit. This implies the last assertion.
\end{proof}

Let $X=\varprojlim_{n}X_{n}$  be a countable inverse limit of finite sets, viewed as a profinite set. 
We define $R_{0}[[X]] := \varprojlim_{n} R_{0}[X_{n}]$ and $R[[X]] := 
(\varprojlim_{n} R_{0}[X_{n}])[1/\vp]$; these are $R_{0}$- resp. $R$-modules 
and independent of the choice of the $X_{n}$'s (here, if $B$ is a ring and $S$ 
is a finite set, $B[S]$ denotes the free $B$-module generated by the set $S$). 
If the $X_{n}$ are groups (so $X$ is profinite group) then they carry natural 
algebra structures. We may topologize $R_{0}[[X]]$ in two ways; either giving 
it the natural inverse limit topology or the $\vp$-adic topology. We give 
$R[[X]]$ the topology induced from the $\vp$-adic topology on $R_{0}[[X]]$, 
which is compatible with viewing $R[[X]]$ as an $R$-Banach module with unit 
ball $R_{0}[[X]]$. 

\begin{proposition}
Let $X=\varprojlim_{n}X_{n}$ be a profinite set. There is a natural $R$-Banach module isomorphism $R[[X]] \ra \D(X,R) $ sending $[x]$ to $\delta_{x}$. It restricts to an $R_{0}$-module isomorphism $R_{0}[[X]] \ra \D(X,R_{0})$ which identifies the inverse limit topology on the source with the weak-star topology on the target. If $X=G$ is a profinite group, then these maps are ring homomorphisms.
\end{proposition}

\begin{proof}
Define compatible maps $R_{0}[X_{n}] \ra \Hom_{R_{0}}(\mc{C}(X_{n},R_{0}), 
R_{0})$ by $[x] \mapsto \delta_{x}$; one checks directly that this is an 
isomorphism of topological $R_{0}$-modules, and that it is a ring homomorphism 
when $X$ is a profinite group. Taking inverse limits we get 
$$ R_{0}[[X]] \overset{\sim}{\longrightarrow} 
\Hom_{R_{0}}(\mc{C}_{sm}(X,R_{0}),R_{0})=\mc{D}(X,R_{0}). $$

Lemma \ref{weak-star} shows that this identifies the inverse limit topology on 
the source with the weak-star topology on the target. Inverting $\varpi$ we get 
the desired isomorphism $R[[X]] \ra \D(X,R)$ which is clearly a Banach module 
isomophism since it identifies the respective unit balls $R_0[[X]]$ and 
$\D(X,R_0)$. 
\end{proof}

Recall the notion of a uniform pro-$p$ group from \cite[Definition 4.1]{ddms}. When $G$ is a uniform pro-$p$ group, $\Zp[[G]]$  may be identified with a ring of non-commutative formal power series
$$ \left\{ \sum_{\alpha}d_{\alpha}\mbf{b}^{\alpha} \mid d_{\alpha}\in \Zp \right\} $$
where $d$ is the dimension of $G$, $\alpha=(\alpha_{1},\ldots,\alpha_{d})\in 
\mb{Z}_{\geq 0}^{d}$ is a multi-index, $g_{1},\ldots,g_{d}$ is a minimal set of 
topological generators of $G$, $b_{i}=[g_{i}]-1$ and 
$\mbf{b}^{\alpha}:=b_{1}^{\alpha_{1}}\cdots b_{d}^{\alpha_{d}}$. Our next goal 
is to show that the analogous description holds for $R_{0}[[G]]$, with the same 
commutation relations between the $\mbf{b}^{\alpha}$. 

\begin{proposition}\label{amice}
Let $G=\Zp^{d}$. For $\alpha\in \mb{Z}_{\geq 0}^{d}$, let $E_{\alpha}\, :\, \Zp^{d} \ra \Zp $ denote the function
$$ E_{\alpha}(x_{1},\ldots,x_{d})= \left( \begin{matrix} x_{1} \\ \alpha_{1} 
\end{matrix} \right)\cdots  \left( \begin{matrix} x_{d} \\ \alpha_{d} 
\end{matrix} \right). $$
Then the Amice transform 
$$ \mu \mapsto \sum_{\alpha}\mu(E_{\alpha})T_{1}^{\alpha_{1}}\cdots 
T^{\alpha_{d}}_{d} $$
defines an algebra isomorphism $\mc{D}(\Zp^{d},R_{0}) 
\overset{\sim}{\longrightarrow} R_{0}[[T_{1},\ldots,T_{d}]] $ which identifies 
the weak-star topology on the source with the product topology 
$R_{0}[[T_{1},\ldots,T_{d}]] = \prod_{\alpha} R_{0}.T^{\alpha_{1}}_{1}\cdots 
T_{d}^{\alpha_{d}}$ on the target.
\end{proposition}

\begin{proof}
Once again this is a simple extension of a well known result, so we content ourselves with a sketch. The key observation is that $\mc{C}(\Zp^{d},R_{0})\cong \mc{C}(\Zp^{d},\Zp) \ctens_{\Zp}R_{0}$. Then it is clear that
$$ \mc{D}(\Zp^{d},R_{0}) \overset{\sim}{\longrightarrow} \prod_{\alpha} R_{0} $$
via $\mu \mapsto (\mu(E_{\alpha}))_{\alpha}$ and it is straightforward to check that this identifies the weak topology on the source with the product topology on the target (it is the statement that the $E_{\al}$ suffice to define the weak-star topology). To finish, we remark that the computation that the algebra structures match up is identical to the well known one in the case $R_{0}=\Zp$.
\end{proof}

Note that the topology on $R_{0}[[T_{1},\ldots,T_{d}]]$ described in the 
Proposition is equal to the $(\vp,T_{1},\ldots,T_{d})$-adic topology. Let us 
return to the case of a general uniform pro-$p$ group $G$. The ring $\Zp[[G]]$ 
is described by formal power series as above. Following \cite{st}, let us 
define elements $c_{\beta\gamma,\alpha}\in \Zp$ by 
\begin{equation}\label{cbga}\mbf{b}^{\beta}\mbf{b}^{\gamma} = \sum_{\alpha}c_{\beta\gamma,\alpha}\mbf{b}^{\alpha}.\end{equation}
We remark that, for fixed $\alpha$, $c_{\beta\gamma,\alpha} \ra 0$ as 
$|\beta|+|\gamma| \ra +\infty$ (here and elsewhere, for a multi-index $\alpha$ 
we define $|\alpha|=\alpha_{1}+\cdots +\alpha_{d}$). This follows from 
\cite[Lemma 4.1(ii)]{st}.

\begin{proposition}\label{D-explicit}
Let $G$ be a uniform pro-$p$ group and use the notation above. Then $R_{0}[[G]]$ may be identified with the ring of formal power series
$$ \left\{ \sum_{\alpha}d_{\alpha}\mbf{b}^{\alpha} \mid d_{\alpha}\in R_0 \right\} $$
with multiplication given by
$$ \left( \sum_{\beta}d_{\beta}\mbf{b}^{\beta} \right) \left( \sum_{\gamma}e_{\gamma}\mbf{b}^{\gamma} \right) = \sum_{\alpha} \left( \sum_{\beta,\gamma} d_{\beta}e_{\gamma}c_{\beta\gamma,\alpha} \right) \mbf{b}^{\alpha}. $$
\end{proposition}

\begin{proof}
The choice $g_{1},\ldots,g_{d}$ of a minimal (ordered) set of topological 
generators identifies $G$, as $p$-adic analytic manifold, with $\Zp^{d}$. Thus 
we get a topological isomorphism $R_{0}[[G]] \cong R_{0}[[\Zp^{d}]]$ of 
$R_{0}$-modules (for both the weak-star and the $\vp$-adic topology). 
Proposition \ref{amice} then implies the description in terms of power series. 
To see that the multiplication works out as described, note that the natural 
map $\Zp[[G]] \ra R_{0}[[G]]$ is an algebra homomorphism, hence the above 
formula is true for products of monomials. We can then deduce the formula in 
general by noting that $R_{0}$ is central in $R_{0}[[G]]$ and that the subring 
generated by $R_{0}$ and the image of $\Zp[[G]]$ (for which the formula holds) 
is dense in $R_{0}[[G]]$ with respect to the weak-star topology, and that 
multiplication is jointly continuous for the weak-star topology. 
\end{proof}

Inverting $\vp$ we get an explicit description of $R[[G]]$ when $G$ is uniform. 
Using this we may now define a family of norms on $R[[G]]$ following \cite[\S 
4]{st}. We continue to fix a minimal ordered set of topological generators 
$g_{1},\ldots,g_{d}$.

\begin{definition}\label{defrnorm}
Let $r\in [1/p,1)$. We define the $r$-norm $||-||_{r}$ on $R[[G]]$ by the formula
$$ || \sum_{\alpha}d_{\alpha}\mbf{b}^{\alpha} ||_{r} = \sup_{\alpha} |d_{\alpha}|r^{|\alpha |}. $$
\end{definition}

Recall that we have fixed a choice of norm $|-|$ on $R$ such that $|z|\leq |z|_{p}$ for all $z\in \Zp$. This will be convenient for some calculations. Note that the definition of $||-||_{r}$ a priori depends on the choice of generators. We remark that for all $r\in [1/p,1)$, $||-||_{r}$ induces the weak-star topology on $R_{0}[[G]]$ by Proposition \ref{amice} and a straightforward calculation. It follows that any homomorphism $G \ra H$ of uniform groups induces a continuous homomorphism $R_{0}[[G]] \ra R_{0}[[H]]$ when the source and target are equipped with (any) $r$-norms, since this is true for the weak-star topology using the characterisation in Lemma \ref{weak-star}.

\begin{proposition}\label{normindep}
$||-||_{r}$ is independent of the choice of minimal ordered set of topological generators for $G$, and is submultiplicative. Finally, if we replace the norm $|-|$ on $R$ by a bounded-equivalent $\Z_p$-algebra norm $|-|^\prime$, then the resulting norm $||-||^\prime_{r}$ is bounded-equivalent to $||-||_r$.
\end{proposition}

\begin{proof}
For the proof of independence, we follow the discussion after the proof of 
\cite[Theorem 4.10]{st}. Let $g_{1}^{\prime},\ldots,g_{d}^{\prime}$ be a 
different choice and set $b_{i}^{\prime}=[g_{i}^{\prime}]-1\in \Zp[[G]]$ etc., 
and let $||-||^{\prime}_{r}$ denote the $r$-norm with respect to this choice. 
We may write
$$ \mbf{b}^{\prime \beta} = \sum_{\alpha}c_{\beta,\alpha}\mbf{b}^{\alpha} $$
in $\Zp[[G]]$, and one has $|c_{\beta,\alpha}|_{p}r^{|\alpha|} \leq r^{|\beta|}$ (see \emph{loc.~cit.}). Transporting this to $R[[G]]$ we have the same identity, and the inequality $|c_{\beta,\alpha}|r^{|\alpha|}\leq r^{|\beta|}$ (since $|c_{\beta,\alpha}|\leq |c_{\beta,\alpha}|_{p}$). Expanding out a general element $\mu\in R[[G]]$ we then have
$$ \mu = \sum_{\beta}d^{\prime}_{\beta}\mbf{b}^{\prime\beta} = \sum_{\alpha} \left( \sum_{\beta} d^{\prime}_{\beta}c_{\beta,\alpha} \right) \mbf{b}^{\alpha}. $$
We then have
$$ ||\mu||_{r}\leq \sup_{\beta,\alpha} |d_{\beta}^{\prime}|.|c_{\beta,\alpha}|r^{|\alpha |} \leq \sup_{\beta} |d_{\beta}^{\prime}|r^{|\beta |}= ||\mu||_{r}^{\prime}.$$
By symmetry, we must have equality.

\medskip
To prove submultiplicativity, we follow the proof of \cite[Proposition 4.2]{st}. Recall the $c_{\beta\gamma,\alpha}$ (\ref{cbga}). By \cite[Lemma 4.1(ii)]{st} we have $|c_{\beta\gamma,\alpha}|r^{|\alpha |} \leq |c_{\beta\gamma,\alpha}|_{p}r^{|\alpha|} \leq r^{|\beta|+|\gamma|}$ for all $\alpha,\beta,\gamma$. Let $\mu=\sum_{\beta}d_{\beta}\mbf{b}^{\beta}$ and $\nu = \sum_{\gamma}e_{\gamma}\mbf{b}^{\gamma}$ be elements of $R[[G]]$, then we have
$$ \mu \ast \nu = \sum_{\alpha}\left( \sum_{\beta,\gamma}d_{\beta}e_{\gamma}c_{\beta\gamma,\alpha} \right) \mbf{b}^{\alpha} $$ 
and we can calculate
$$ ||\mu\nu||_{r} = \sup_{\alpha} | \sum_{\beta,\gamma} d_{\beta}e_{\gamma}c_{\beta\gamma,\alpha}|r^{|\alpha|} \leq \sup_{\beta,\gamma} |d_{\beta}||e_{\gamma}|r^{|\beta|+|\gamma|}=||\mu||_{r}||\nu||_{r}$$
where we use submultiplicativity and $|c_{\beta\gamma,\alpha}|r^{|\alpha |}\leq r^{|\beta|+|\gamma|}$ to obtain the middle inequality. This finishes the proof of submultiplicativity.

Finally, suppose we have $C_1|x|\le |x|' \le C_2|x|$ for all $x \in R$. Then $C_1|d_\alpha|r^{|\alpha|}\le |d_\alpha|'r^{|\alpha|}\le C_2|d_\alpha|r^{|\alpha|}$ for all $\alpha$ which implies that \[ C_1|| \sum_{\alpha}d_{\alpha}\mbf{b}^{\alpha} ||_{r} \le || \sum_{\alpha}d_{\alpha}\mbf{b}^{\alpha} ||'_{r} \le C_2|| \sum_{\alpha}d_{\alpha}\mbf{b}^{\alpha} ||_{r} \] as desired.
\end{proof}
 Before moving on to general compact $p$-adic analytic $G$, we record a few properties of the $r$-norms.

\begin{lemma} \label{aut}Let $r\in [1/p,1)$. If $g\in G$, then $||[g]||_{r}=1$ and $||[g]\mu||_{r}=||\mu[g]||_{r}=||\mu||_{r}$ for all $\mu\in R[[G]]$. Moreover, if $\phi$ is an automorphism of $G$ (of $p$-adic analytic groups), then $\phi$ induces an automorphism of $R[[G]]$ satisfying $||\phi(\mu)||_{r}=||\mu||_{r}$ for all $\mu\in R[[G]]$.
\end{lemma}

\begin{proof}
The first statement follows from the fact that the expansion of $[g]$ has coefficients in $\Zp$ and the constant term is $1$. The second is an easy consequence of the first and submultiplicativity (since $[g]^{-1}=[g^{-1}]$ also has norm $1$).

\medskip
For the final statement, observe that if $g_{1},\ldots,g_{d}$ is a set of 
topological generators then so are $\phi(g_{1}),\ldots,\phi(g_{d})$. Since the 
$r$-norms are independent of the choice of generators, we conclude that if 
$\mu=\sum_{\alpha}d_{\alpha}\mbf{b}^{\alpha}$, then
$$ ||\phi(\mu)||_{r} = || \sum_{\alpha} d_{\alpha} (\phi(\mbf{b}))^{\alpha} ||_{r}=\sup_{\alpha}|d_{\alpha}|r^{|\alpha|}= ||\mu||_{r}. $$
This finishes the proof.
\end{proof}

We will use the last property mostly in the case when $H$ is a compact $p$-adic analytic group, $N\sub H$ is a uniform open normal subgroup, and $\phi$ is the automorphism of $N$ given by conjugation by some $h\in H$. 

\medskip
Now let $G$ be an arbitrary compact $p$-analytic group. Pick a uniform open 
normal subgroup $N$ and a set $h_{1},\ldots,h_{t}$ of coset representatives of 
$G/N$. Any $\mu\in R[[G]]$ may be written uniquely as 
$\mu=\sum_{i}[h_{i}]\mu_{i}$ with $\mu_{i}\in R[[N]]$, and we define a norm 
$||-||_{N,r}$ on $R[[G]]$ by
$$ ||\mu||_{N,r} =\sup_{i} ||\mu_{i}||_{r}. $$ We could alternatively take a right coset decomposition $\mu=\sum_{i}\nu_{i}[h_{i}]$ with $\nu_i \in R[[N]]$, and define $||-||^{right}_{N,r}$ on $R[[G]]$ by
$$ ||\mu||^{right}_{N,r} =\sup_{i} ||\nu_{i}||_{r}. $$

\begin{proposition}\label{aut2}
We have $||-||_{N,r} = ||-||^{right}_{N,r}$. The definition is also independent of the choice of coset representatives. Moreover, $||-||_{N,r}$ is submultiplicative and satisfies $||[g]||_{N,r}=1$ and $||[g]\mu||_{N,r}=||\mu[g]||_{N,r}=||\mu||_{N,r}$ for all $g\in G$ and $\mu\in R[[G]]$.
\end{proposition}

\begin{proof}
For left/right independence, note that $\mu=\sum_{i} ([h_{i}]\mu_{i}[h_{i}^{-1}])[h_{i}]$ and hence
$$ ||\mu||_{N,r}^{right}=\sup_{i} ||[h_{i}]\mu_{i}[h_{i}^{-1}]||_{r} =\sup_{i} ||\mu_{i}||_{r} = ||\mu||_{N,r} $$
by Lemma \ref{aut}. For independence of the coset representatives, suppose we have $h_{i}^{\prime}$ a different set with $h_{i}=h_{i}^{\prime}n_{i}$, then $\mu=\sum_{i} [h_{i}^{\prime}]([n_{i}]\mu_{i}) $ and hence
$$ ||\mu||_{N,r}^{\prime}=\sup_{i}||[n_{i}]\mu_{i}||_{r} = \sup_{i} 
||\mu_{i}||_{r} = ||\mu||_{N,r} $$
by Lemma \ref{aut} again. Next we prove submultiplicativity. Define $k(i,j)$ by $h_{i}h_{j}=h_{k(i,j)}n_{ij}$. For $\mu = \sum_{i}[h_{i}]\mu_{i}$ and $\nu=\sum_{j}[h_{j}]\nu_{j}$ in $R[[G]]$, we have
$$ \mu\nu = \sum_{k} [h_{k}] \left( \sum_{i,j\,:\,k(i,j)=k} [n_{ij}]([h_{j}^{-1}]\mu_{i}[h_{j}])\nu_{j} \right). $$
Using Lemma \ref{aut} and submultiplicativity for $||-||_{r}$ on $N$ one then sees easily that $||\mu\nu||_{N,r}\leq ||\mu||_{N,r}||\nu||_{N,r}$.

\medskip
Next, let $g\in G$. Writing $g=h_{i}n$ for some $i$ and $n\in N$, we see that $||[g]||_{N,r}=||[n]||_{r}=1$ by Lemma \ref{aut}. Finally, the last property then follows by the same argument as in the proof of Lemma \ref{aut}.
\end{proof}

As the notation suggests, $||-||_{N,r}$ \emph{does} depend on the choice of $N$. For a study of how the completions change when one changes the subgroup in certain situations, see \cite[\S 10.6-10.8]{aw}.

\subsection{Completions}\label{completions} In this section we study the case when $G$ is a uniform pro-$p$ group in more detail. Let $R$ be a Banach--Tate $\Zp$-algebra with multiplicative pseudo-uniformizer $\vp$ as usual.

\begin{definition}
For $r\in [1/p,1)$, define $\mc{D}^{r}(G,R)$ to be the completion of $\mc{D}(G,R)$ with respect to the norm $||-||_{r}$, and we let $\ms{D}(G,R)$ denote the completion of $\mc{D}(G,R)$ with respect to the entire family of norms $(||-||_{r})_{r\in [1/p,1)}$.
\end{definition}
\begin{remark}\label{boundedok}
Note that if we change the norm on $R$ to a bounded-equivalent one, then the completion $\mc{D}^{r}(G,R)$ is unchanged, by Proposition \ref{normindep}.
\end{remark}

The motivation for this definition is that if $R$ is a Banach $\Qp$-algebra, $\ms{D}(G,R)$ is naturally isomorphic to the space of locally analytic $R$-valued distributions on $G$, by Proposition \ref{distrocomparison}. 

Note that there are natural norm-decreasing injective maps $\mc{D}^{s}(G,R) \ra \mc{D}^{r}(G,R)$ whenever $r\leq s$ (which we will think of as inclusions), and that they fit together into an inverse system with limit $\ms{D}(G,R)$. The explicit description of $\mc{D}(G,R)$ from Proposition \ref{D-explicit} gives us an explicit formal power series description
$$ \mc{D}^{r}(G,R)=\left\{ \sum_{\al} d_{\al}\mbf{b}^{\al} \mid d_{\al}\in R,\,|d_{\alpha}|r^{|\al|} \ra 0 \right\}$$
and the norm is still given by $|| \sum d_{\al}\mbf{b}^{\al}||_{r}=\sup |d_{\al}|r^{|\al|}$. We see that, unlike $\mc{D}(G,R)$, the $\mc{D}^{r}(G,K)$ are naturally potentially ON-able, with a potential ON-basis given by the elements $\vp^{-n(r,\vp,\al)}\mbf{b}^{\al}$, where
\begin{equation}\label{ONbasis} n(r,\vp,\al)= \left\lfloor \frac{|\al|\log_{p}r}{\log_{p}|\vp|} \right\rfloor.\end{equation}

\begin{lemma}\label{cpt2}
If $r<s$ then the inclusion $\iota\, :\, \mc{D}^{s}(G,R) \hookrightarrow \mc{D}^{r}(G,R)$ is a compact map of $R$-Banach modules.
\end{lemma}

\begin{proof}
For simplicity set $\mc{D}^{r}:=\mc{D}^{r}(G,R)$ etc. for the duration of this 
proof. Choosing a minimal set of topological generators $g_{1},\ldots,g_{d}$ of 
$G$ and set $b_{i}=[g_{i}]-1$ as usual. For $n\geq 1$ define $T_{n}\, :\, 
\mc{D}^{s} \ra \mc{D}^{r}$ by 
$$ T_{n}\left( \sum_{\alpha} d_{\alpha}\mbf{b}^{\alpha} \right) =\sum_{|\alpha| < n} d_{\alpha}\mbf{b}^{\alpha}. $$
By definition we see that $T_{n}$ is of finite rank. Then if $\sum_{\alpha} d_{\alpha}\mbf{b}^{\alpha}$ is in the unit ball of $\mc{D}^{s}$, i.e. $|d_{\al}|s^{|\al|}\leq 1$ for all $\al$, we have
$$ ||(\iota -T_{n})( \sum_{\alpha} d_{\alpha}\mbf{b}^{\alpha})||_{r} = ||\sum_{|\alpha|\geq n} d_{\alpha}\mbf{b}^{\alpha}||_{r}\leq (r/s)^{n}. $$
Thus $||\iota-T_{n}||\leq (r/s)^{n}$ where $||-||$ is the operator norm on $\Hom_{R,cts}(\mc{D}^{s},\mc{D}^{r})$, and hence $T_{n} \ra \iota $ as $n \ra \infty$, so $\iota$ is compact. 
\end{proof}

\begin{lemma}\label{cpt1}
Let $G$ and $H$ be two uniform pro-$p$ groups and assume that $f\, :\, G \ra H$ is a homomorphism of $p$-adic analytic groups such that $f(G)\sub H^{p^{n}}$ for some $n\geq 0$. Then the induced map $f_{\ast}\, :\, \mc{D}(G,R) \ra \mc{D}(H,R)$ is norm-decreasing when we equip $\mc{D}(G,R)$ with $||-||_{r}$ and $\mc{D}(H,R)$ with $||-||_{r^{1/p^{n}}}$. As a consequence, we get an induced map 
$$ f_{\ast}\, :\, \mc{D}^{r}(G,R) \ra \mc{D}^{r}(H,R) $$
which factors through the natural map  $\mc{D}^{r^{1/p^{n}}}(H,R) \ra \mc{D}^{r}(H,R)$. In particular, when $n\geq 1$, $f_{\ast}$ is compact.
\end{lemma}

\begin{proof}
We start with the first assertion. Note that the general case follows from two 
special cases: $n=0$, $f$ arbitrary, and $n=1$, $G=H^{p}$ with $f$ the 
inclusion. Indeed the general case can be written as a composition of these 
cases.

\medskip

So, suppose first that $n=0$. Scaling by powers of $\vp$, it suffices to prove 
this for $\D(G,R_{0})$. The map $f_{\ast}$ is continuous with respect to the 
norm $||-||_r$ (see the discussion after Definition \ref{defrnorm}). Let 
$g_{1},\ldots,g_{d}$ be a minimal set of topological generators for $G$ and let 
$b_{i}=[g_{i}]-1$ as usual. Using that $||[h]-1||_{r}\leq r$ for all $h\in H$, 
we see that 
$$ ||f_{\ast}( \sum d_{\alpha}\mbf{b}^{\alpha} )||_{r} = || \sum d_{\alpha}f_{\ast}(\mbf{b})^{\alpha}||_{r} \leq \sup |d_{\alpha}|r^{|\alpha|}= ||\sum d_{\alpha}\mbf{b}^{\alpha}||_{r}$$
as desired, using continuity of $f_{\ast}$. This completes the proof of the 
first special case. 

Next, we consider the second case $G=H^{p}\sub H$. Let $s=r^{1/p}$; since 
$r\geq p^{-1}$ we have $s\geq p^{-1/p}>p^{-1/(p-1)}$. Let $h_{1},\ldots,h_{d}$ 
be 
a a minimal set of topological generators for $H$. Then 
$h_{1}^{p},\ldots,h_{1}^{d}$ form a minimal set of topological generators for 
$H^{p}$. Set $b_{i}=[h_{i}^{p}]-1$ and $b_{i}^{\prime}=[h_{i}]-1$ (apologies 
for the mild abuse of notation). Then, inside $\mc{D}(H,R)$, we have
$$ ||b_{i}||_{s}=||(1+b_{i}^{\prime})^{p}-1||_{s}= ||\sum_{k=1}^{p} \begin{pmatrix} p \\ k \end{pmatrix} (b_{i}^{\prime})^{k} ||_{s} = s^{p}=r $$
using $s>p^{-1/(p-1)}$ and $|p|\leq |p|_{p}=p^{-1}$. Thus, if $\sum d_{\alpha}\mbf{b}^{\alpha}\in \mc{D}(H^{p},R)$, then inside $\mc{D}(H,R)$ we have
$$ ||\sum d_{\alpha}\mbf{b}^{\alpha}||_{s} \leq \sup |d_{\alpha}|r^{\alpha} $$
which is equal to $||\sum d_{\alpha}\mbf{b}^{\alpha}||_{r}$ computed inside $\mc{D}(H^{p},R)$. This finishes the proof of first assertion. The remaining assertions are then easily verified (using Lemma \ref{cpt2} for last one).
\end{proof}

Before discussing what happens when one changes the norm on $R$, we record the following important base change lemma.

\begin{lemma}\label{funct}
Let $R$ and $S$ be Banach--Tate $\Z_p$-algebras, and let $f\, :\, R \ra S$ be a bounded ring homomorphism. Suppose that there is a multiplicative pseudo-uniformizer $\vp \in R$ such that $f(\vp)$ is also multiplicative. Let $r\in [1/p,1)$. Then the natural map $\D^{r}(G,R)\ctens_{R}S \ra \D^{r}(G,S)$ is an isomorphism of Banach $S$-modules.
\end{lemma}

\begin{proof}
	Note that (since $f$ is bounded) the fact that $\vp$ and $f(\vp)$ are both multiplicative implies that $|\vp| = |f(\vp)|$. We recall the potential ON-bases $(\vp^{-n(r,\vp,\al)}\mbf{b}^{\al})_{\al}$ and $(f(\vp)^{-n(r,f(\vp),\al)}\mbf{b}^{\al})_{\al}$ of $\mc{D}^{r}(G,R)$ and $\mc{D}^{r}(G,S)$, respectively. It is straightforward to check that the tensor product norm on \[\left(\bigoplus_\al R(\vp^{-n(r,\vp,\al)}\mbf{b}^{\al})\right) \otimes_R S = \bigoplus_\al S(f(\vp^{-n(r,\vp,\al)})\mbf{b}^{\al})\] induced by $||-||_r$ on $\D^{r}(G,R)$ and the norm on $S$ is bounded-equivalent to the norm induced by $||-||_r$ on $\mc{D}^{r}(G,S)$, and so we obtain isomorphic completions, which gives the desired statement.\end{proof}
	
Our next goal is to prove that $\ms{D}(G,R)$ is independent, as a topological $R$-module, of the choice of norm on $R$. Recall that we only consider norms for which there exists a multiplicative pseudo-uniformizer, and such that the natural map $\Zp \ra R$ is norm-decreasing.

\begin{proposition}\label{normindep1}
$\ms{D}(G,R)$ is independent of the choice of norm on $R$.
\end{proposition}

\begin{proof}
	Choose a norm on $R$ and a minimal set of topological generators for $G$. Recall that $\ms{D}(G,R) = \varprojlim_r \mc{D}^r (G,R)$ and that $\mc{D}^r (G,R)$ can be described as
	\[
	\mc{D}^{r}(G,R)=\left\{ \sum_{\al} d_{\al}\mbf{b}^{\al} \mid d_{\al}\in R,\,|d_{\alpha}|r^{|\al|} \ra 0 \right\}.
	\]
	As a consequence, $\ms{D}(G,R)$ consists of the formal sums $\sum_{\al} d_{\al}\mbf{b}^{\al}$ for which $|d_{\alpha}|r^{|\al|} \ra 0$ for all $r\in [1/p,1)$. The topology is given by the family of norms $(||-||_r)_{r \in [1/p,1)}$, with $||\sum_{\al} d_{\al}\mbf{b}^{\al}||_r = \sup_\al |d_\al|r^{|\al|}$. In particular, it is metrizable and translation-invariant, so the topology is determined by the set of sequences that tend to $0$.
	
	\medskip
	
	To prove independence, we start by proving that the condition $|d_{\alpha}|r^{|\al|} \ra 0$ for all $r\in [1/p,1)$ is independent of the choice of norm. Let $|-|_\pi$ and $|-|_\vp$ be two equivalent norms on $R$ with multiplicative pseudo-uniformizers $\pi$ and $\vp$, respectively. Choose $t\in [1/p,1)$; by symmetry we need to show that if $|d_{\alpha}|_\pi r^{|\al|} \ra 0$ for all $r\in [1/p,1)$, then $|d_{\alpha}|_\vp t^{|\al|} \ra 0$. By Lemma \ref{norm}, if $|d_\al|_\pi < 1$ then $|d_\al|_\vp \leq C_1$ for some constant $C_1$. As $t < 1$, we see that it remains to show $|d_{\alpha}|_\vp t^{|\al|} \ra 0$ for the subsequence of $\al$'s for which $|d_\al|_\pi \geq 1$. But for those $\al$, Lemma \ref{norm} gives us constants $C_2,s>0$ such that
	\[
	|d_\al|_\vp t^{|\al|} \leq C_2 |d_\al|_\pi^s t^{|\al|} = C_2 \left( |d_\al|_\pi (t^{1/s})^{|\al|} \right) ^s.
	\]
	Since $t^{1/s} \in (0,1)$, $|d_\al|_\pi (t^{1/s})^{|\al|} \ra 0$ by assumption, so $|d_\al|_\vp t^{|\al|} \ra 0$ as desired.
	
	\medskip
	
	This shows that the $R$-module $\ms{D}(G,R)$ is independent of the choice of norm, so we are left with checking independence of the topology, for which it suffices to check that convergence to $0$ is independent. Let $\mu_n = \sum_{\al} d_\al^{(n)} \mbf{b}^{\al}$ be a sequence in $\ms{D}(G,R)$. Let $t \in [1/p,1)$. By symmetry, we need to show that if $\sup_\al |d_\al^{(n)}|_\pi r^{|\al|} \ra 0$ for all $r\in [1/p,1)$, then $\sup_\al |d_\al^{(n)}|_\vp t^{|\al|} \ra 0$. Choose an arbitrary $\epsilon >0$; we need to show that $\sup_\al |d_\al^{(n)}|_\vp t^{|\al|} \leq \epsilon$ for large enough $n$.
	
	\medskip
	
	Let $C_1,C_2,s> 0$ be the constants from Lemma \ref{norm}. There is a finite subset $S$ of $\al$'s such that if $\al \notin S$, then $C_1 t^{|\al|} \leq \epsilon$. Write 
	\[
	\sup_\al |d_\al^{(n)}|_\vp t^{|\al|} = \max \left\{ \sup_{\al \in S} |d_\al^{(n)}|_\vp t^{|\al|}, \sup_{\al \notin S} |d_\al^{(n)}|_\vp t^{|\al|} \right\}.
	\] 
	Note that for a fixed $\al$, $|d_\al^{(n)}|_\pi \ra 0$, so by equivalence $|d_\al^{(n)}|_\vp \ra 0$. In particular, the term $\sup_{\al \in S} |d_\al^{(n)}|_\vp t^{|\al|} \ra 0$ since $S$ is finite, so we can find $N_1$ such that $n\geq N_1$ implies that $\sup_{\al \in S} |d_\al^{(n)}|_\vp t^{|\al|} \leq \epsilon$. It then remains to treat the term $\sup_{\al \notin S} |d_\al^{(n)}|_\vp t^{|\al|}$. Let $\al \notin S$. If $|d_\al^{(n)}|_\pi < 1$, then
	\[
	|d_\al^{(n)}|_\vp t^{|\al|} \leq C_1 t^{|\al|} \leq \epsilon
	\]
	independent of $n$. On the other hand, if $|d_\al^{(n)}|_\pi \geq 1$, then as above
	\[
	|d_\al^{(n)}|_\vp t^{|\al|} \leq C_2 \left( |d_\al^{(n)}|_\pi (t^{1/s})^{|\al|} \right) ^s
	\]
	and by assumption $\sup_\al |d_\al^{(n)}|_\pi (t^{1/s})^{|\al|} \ra 0$ as $n \ra \infty$, so we can find $N_2$ such $n \geq N_2$ implies $\sup_{\al \notin S} |d_\al^{(n)}|_\vp t^{|\al|} \leq \epsilon$. Thus $\sup_{\al} |d_\al^{(n)}|_\vp t^{|\al|} \leq \epsilon$ for $n \geq N= \max(N_1,N_2)$, which finishes the proof.
\end{proof}

We now introduce a variant of the $\D^{r}(G,R)$, which, when $R$ is a Banach 
$\Qp$-algebra, recovers the analytic distribution algebra (with fixed radius of 
analyticity). Although we do not use this variant in our construction of 
eigenvarieties, it is used in Section \ref{galrep} to construct Galois 
representations.

\medskip

Let $r>s\geq 1/p$. Let $\D^{r,\circ}(G,R)=\{ \sum d_{\al}\mbf{b}^{\al} \mid 
|d_{\al}|r^{|\al|}\leq 1 \}$ denote the unit ball of $\D^{r}(G,R)$. We define 
$\D^{<r,\circ}(G,R)$ to be the closure of $\D^{r,\circ}(G,R)$ in $\D^{s}(G,R)$. 
$\D^{<r,\circ}(G,R)$ is an $R_{0}$-module, a priori depending on $s$, which 
carries two natural topologies (the $\vp$-adic topology and the subspace 
topology coming from $\D^{s}(G,R)$).

\begin{proposition}\label{<}
We have 
$$\D^{<r,\circ}(G,R) = \left\{\sum_{\al} d_{\al}\mbf{b}^{\al} \in \D^{s}(G,R) \mid |d_{\al}|r^{|\al|} \leq 1 \right\}  $$
as a subset of $\D^{s}(G,R)$. Thus $\D^{<r,\circ}(G,R)$ is independent of $s$. The subspace topology corresponds to the weak topology with respect to the family of maps $(\mu \mapsto d_{\al}(\mu))_{\al}$ on the right hand side (and is therefore also independent of $s$). The $\vp$-adic topology is induced by the norm
$$ ||\sum d_{\al}\mbf{b}^{\al}||_{r} = \sup_{\al} |d_{\al}|^{\prime}r^{|\al|} $$
and is separated and complete. 
\end{proposition}

\begin{proof}
Let $W=\{ \sum d_{\al}\mbf{b}^{\al} \in \D^{s}(G,R) \mid |d_{\al}|r^{|\al|} \leq 1 \}$. Since the maps $\mu \mapsto d_{\al}(\mu)$ are continuous we see that $W$ is closed. On other hand, any finite truncation of an element in $W$ is in $\D^{r,\circ}(G,R)$, so $\D^{r,\circ}(G,R)$ is dense in $W$. It follows that $W=\D^{<r,\circ}(G,R)$. The subspace topology is given by the norm $||-||_{s}$, and one checks easily that this agrees with the weak topology in the statement of the proposition. The final statement is similarly easy to check; we leave it to the reader.  
\end{proof}

We then set $\D^{<r}(G,R)=\D^{<r,\circ}(G,R)[1/\vp]$; this is naturally a 
Banach $R$-module which embeds into $\D^{s}(G,R)$ for all $s\in [1/p,r)$. The 
$(\D^{<r}(G,R))_{r>1/p}$ form an inverse system and the natural map 
$\D^{r}(G,R) \ra \D^{s}(G,R)$ factors over $\D^{<r}(G,R)$, so we therefore have
$$ \ms{D}(G,R) = \varprojlim_{r} \D^{<r}(G,R) $$
as well. Recall the potential ON-basis $(\vp^{-n(r,\vp,\al)}\mbf{b}^{\al})_{\al}$ of $\D^{r}(G,R)$; we have
$$ \D^{<r}(G,R)=\left\{ \sum_{\al}d_{\al}\mbf{b}^{\al} \mid |d_{\al}|r^{|\al|}\, \text{bounded} \right\} = \left( \prod_{\al} R_{0}.\vp^{-n(r,\vp,\al)}\mbf{b}^{\al} \right)\left[ \frac{1}{\vp}\right]. $$
We remark that if $(\vp^{-n(r,\vp,\al)}\mbf{b}^{\al})_{\al}$ is an ON-basis, we have $\D^{<r,\circ}(G,R)=\prod_{\al}R_{0}.\vp^{-n(r,\vp,\al)}\mbf{b}^{\al}$ and the weak topology on the left hand side is equal to the product topology on the right hand side. Next, we define some function modules in a similar fashion. We put 
$$ \mc{C}^{<r,\circ}(G,R) = \{ f \in \mc{C}(G,R) \mid |\mu(f)| \leq 1\, \forall \mu \in \D(G,R) \cap \D^{r,\circ}(G,R) \} $$
and set $\mc{C}^{<r}(G,R)=\mc{C}^{<r,\circ}(G,R)[1/\vp]\sub \mc{C}(G,R)$. We note that $f=\sum_{\al}c_{\al}E_{\al}\in \mc{C}^{<r}(G,R)$ if and only if $\vp^{-n(r,\vp,\al)}\mbf{b}^{\al}(f)=c_{\al}\vp^{-n(r,\vp,\al)}$ is bounded as $\al \ra \infty$; from this one sees that
$$ \mc{C}^{<r}(G,R) =\left( \prod_{\al} R_{0}.\vp^{n(r,\vp,\al)}E_{\al} \right) \left[\frac{1}{\vp} \right] $$
and that it is the dual space of $\D^{r}(G,R)$, and that the dual norm is given by $||\sum_{\al}c_{\al}E_{\al}||_{r}=\sup_{\al}|c_{\al}|r^{-|\al|}$. When $r>s\geq 1/p$, we let $\mc{C}^{r}(G,R)$ denote the closure of $\mc{C}^{<s}(G,R)$ inside $\mc{C}^{<r}(G,R)$. Tracing through the definitions we get an explicit description
$$ \mc{C}^{r}(G,R) = \left\{ \sum_{\al}c_{\al}E_{\al} \mid |c_{\al}|r^{-|\al|}\ra 0 \right\}= \wh{\bigoplus_{\al}}R.\vp^{n(r,\vp,\al)}E_{\al} $$
and see that the dual space of $\mc{C}^{r}(G,R)$ is $\D^{<r}(G,R)$.
\begin{remark}
The reader should compare our description of $\mc{C}^{r}(G,R)$ with the construction of \cite[Section 5.4]{lwx}. Here, the authors give a definition of a `modified' space of continuous functions on $\Z_p$ with values in $\Z_p[[T,pT^{-1}]]$ which is morally (apart from the slight difference in coefficients and the fact that we have only defined $\mc{C}^{r}(G,R)$ for $r>1/p$) the space $\mc{C}^{1/p}(\Z_p,R)$, where $R$ is the Tate ring obtained as the rational localisation \[\Z_p[[T]]\langle \frac{p}{T}\rangle\left[\frac{1}{T}\right]\] of $\Z_p[[T]]$. We give $R$ the norm with unit ball $\Z_p[[T]]\langle \frac{p}{T}\rangle$ and $|T|=1/p$, defined as in Remark \ref{tateremark}. The definition which appears in \cite{lwx} is therefore a useful motivation for the general constructions of this article, although we will not use the modules $\mc{C}^{r}(G,R)$ in this article. We will see in Theorem \ref{lwxboundary} that one may use the module $\D^{1/p}(\Zp,R)$ (which is defined) to prove the main results of \cite{lwx}.
\end{remark}

\medskip

To finish this section we discuss the relationship between our constructions and the spaces of locally analytic functions resp. distributions when $R$ is a Banach $\Qp$-algebra. We may and do assume that $\Qp$ is isometrically embedded into $R$ and that $|p|=1/p$. Recall that the atlas on $G$ induced from our choice of a topological basis identifies $G^{p^{n}}$ with $(p^{n}\Zp)^{\dim G}$. Amice's theorem \cite[Th\'{e}or\`{e}me I.4.7]{colmez} tells us that the space of $n$-analytic functions on $G$ with respect to this atlas is explicitly given as
$$ \mc{C}^{n-an}(G,\Qp) = \wh{\bigoplus_{\al}} \Qp.k_{\al}E_{\al}$$
where $ k_{\al}:= \lfloor p^{-n}\al_{1} \rfloor ! \cdots  \lfloor 
p^{-n}\al_{\dim G} \rfloor !$. The space of $n$-analytic $R$-valued functions 
on $G$ is then 
$$ \mc{C}^{n-an}(G,R) = \wh{\bigoplus_{\al}} R.k_{\al}E_{\al}\sub \mc{C}(G,R). $$
It is well known that $|k_{\al}| \sim r_{n}^{|\al|}$ with $r_{n}=p^{-1/p^{n}(p-1)}$, and it follows that $\mc{C}^{n-an}(G,R)=\mc{C}^{r_{n}}(G,R)$ as $\Qp$-Banach spaces. Since $r_{n} \ra 1$ from below as $n \ra \infty$ it follows that $\ms{C}(G,R)$ is the space of locally analytic $R$-valued functions on $G$, with its usual locally convex topology. Dually, $\ms{D}(G,R)$ is then the space of locally analytic $R$-valued distributions with its usual locally convex topology. We sum up this discussion in a proposition.

\begin{proposition}\label{distrocomparison}
When $R$ is Banach $\Qp$-algebra, $\ms{C}(G,R)$ is canonically the space of locally analytic $R$-valued functions on $G$ and $\ms{D}(G,R)$ is dually the space of locally analytic $R$-valued distributions on $G$.
\end{proposition}

\subsection{Ash-Stevens distribution modules for Banach--Tate $\Zp$-algebras}\label{modules} While the definitions in this subsection will be of  a local nature, let us nevertheless start by introducing the global setup that we will need to define eigenvarieties. Let $F$ be a number field. We put $\mbf{G}={\rm Res}_{\Q}^{F}\mbf{H}$, where $\mbf{H}$ is a connected reductive group over $F$ split at all places $v|p$. $\mbf{G}$ is then a connected reductive group over $\Q$. When $v|p$ we write $\mbf{H}_{\oo_{F_{v}}}$ for a split model of $\mbf{H}_{F_{v}}$ over $\oo_{F_{v}}$ and choose a Borel subgroup $\mbf{B}_{v}$ (with unipotent radical $\mbf{N}_{v}$) and a maximal torus $\mbf{T}_{v}\sub \mbf{B}_{v}$ of $\mbf{H}_{\oo_{F_{v}}}$.  Set $\mbf{G}_{\Zp}=\prod_{v\mid p}{\rm Res}_{\Zp}^{\oo_{F_{v}}}\mbf{H}_{\oo_{F_{v}}}$, $\mbf{B}=\prod_{v\mid p}{\rm Res}_{\Zp}^{\oo_{F_{v}}}\mbf{B}_{v}$ and $\mbf{T}=\prod_{v\mid p}{\rm Res}_{\Zp}^{\oo_{F_{v}}}\mbf{T}_{v}$ and also put $\mbf{N}=\prod_{v\mid p} {\rm Res}_{\Zp}^{\oo_{F_{v}}}\mbf{N}_{v}$. We use overlines to denote opposite groups; e.g. $\ol{\mbf{B}}=\prod_{v\mid p}{\rm Res}_{\Zp}^{\oo_{F_{v}}}\ol{\mbf{B}}_{v}$ where the $\ol{\mbf{B}}_{v}$ are the opposite Borels of $\mbf{B}_{v}$.

\medskip
We will also need notation for various subgroups of $G:=\mbf{G}(\Qp)$. Set $G_{0}=\mbf{G}_{\Zp}(\Zp)$, $B_{0}=\mbf{B}(\Zp)$, $T_{0}=\mbf{T}(\Zp)$, $N_{0}=\mbf{N}(\Zp)$. We let $I$ denote the Iwahori subgroup of $G$ defined as the preimage of $\mbf{B}(\mb{F}_{p})$ under the reduction map $G_{0} \ra \mbf{G}_{\Zp}(\mb{F}_{p})$, and we let $K_{s}={\rm Ker}(G_{0} \ra \mbf{G}_{\Zp}(\Z/p^{s}))$ be the $s$-th principal congruence subgroup of $G_{0}$, for $s\geq 1$. Again for $s\geq 1$, set
$$ T_{s}= {\rm Ker}(T_{0} \ra \mbf{T}(\Z /p^{s}\Z ));$$
$$ N_{s}= {\rm Ker}(N_{0} \ra \mbf{N}(\Z /p^{s}\Z ));$$
$$ \ol{N}_{s} = {\rm Ker}(\ol{\mbf{N}}(\Zp) \ra \ol{\mbf{N}}(\Z/p^{s}\Z)). $$
Set $B_{s}=T_{s}N_{s}$ for $s\geq 0$. We have Iwahori decompositions $I=\ol{N}_{1}T_{0}N_{0}$ and $K_{s}=\ol{N}_{s}T_{s}N_{s}$. Next, choose a splitting $T:=\mbf{T}(\Qp) \ra T_{0}$ of the inclusion $T_{0}\sub T$ and put $\Sigma={\rm Ker}(T \ra T_{0})$. We set
$$ \Sigma^{+}= \left\{ t\in \Sigma \mid t\ol{N}_{1}t^{-1}\sub \ol{N}_{1} \right\}; $$
$$ \Sigma^{cpt}= \left\{ t\in \Sigma \mid t\ol{N}_{1}t^{-1}\sub \ol{N}_{2} \right\}. $$
We may then define $\Delta_{p}=I\Sigma^{+}I$; this is a monoid and 
$(\Delta_{p},I)$ is a Hecke pair (which means that $I$ and $\delta I 
\delta^{-1}$ are commensurable for all $\delta\in \Delta_{p}$). The 
corresponding Hecke algebra (defined over $\Z_p$) will be denoted by 
$\mb{T}(\Delta_{p},I)$.

\medskip
With these preparations let us now move on to the definition of analytic and locally analytic distribution modules for general Banach--Tate $\Zp$-algebras $R$. When $R$ is a Banach $\Qp$-algebra, these were defined in \cite{as} using (locally) analytic functions on $I$. Recasting this in terms of the norms of the previous section, we are able to extend the definition to all Banach--Tate $\Zp$-algebras. 

\medskip
Let $\kappa\, :\, T_{0} \ra R^{\times}$ be a continuous character. We will put some restrictions on the choice of norm on $R$, according to the following lemma. We set $\epsilon=1$ if $p\neq 2$ and $\epsilon=2$ if $p=2$, and put $q=p^{\epsilon}$.

\begin{lemma}\label{norm2}
Keep the above notation and assume that $R$ is Noetherian. Write $|-|$ for the given norm on $R$ and let $\vp$ be a multiplicative pseudo-unformizer. Then there exists a Banach--Tate $\Zp$-algebra norm $|-|^{\prime}$ on $R$, bounded-equivalent to $|-|^{s}$ for some $s>0$, such that $|\ka(t)|^{\prime} \leq 1$ for all $t\in T_{0}$, $|\ka(t)-1|^{\prime}<1$ for all $t \in T_{\epsilon}$, and $\vp$ is a multiplicative pseudo-uniformizer for $|-|^{\prime}$.
\end{lemma}

\begin{proof}
The proof is similar to that of Lemma \ref{BTZp}. Let $t_{1},\ldots,t_{a}$ be a 
set of topological generators of $T_{0}$ and let $t_{a+1},\ldots,t_{b}$ be a 
set of topological generators for $T_{\epsilon}$; it suffices to find 
$|-|^{\prime}$ such that $|\ka(t_{i})|^{\prime}\leq 1$ for $i=1,\ldots,a$, 
$|\ka(t_{i})-1|<1$ for $i=a+1,\ldots,b$, and $|p|^{\prime} \leq p^{-1}$, which 
is bounded-equivalent to $|-|^{s}$ for some $s>0$.

\medskip

Since $T_{0}$ is compact, $\ka(T_{0})$ is bounded, and hence $t_{i}$ is 
powerbounded for all $i$. Moreover, $T_{\epsilon}$ is pro-$p$ and 
$R^{\circ}/R^{\circ\circ}$ is a reduced discrete ring of characteristic $p$, so 
the continuous homomorphism $T_{\epsilon} \ra 
(R^{\circ}/R^{\circ\circ})^{\times}$ induced from $\ka$ is trivial. We conclude 
that $\ka(T_{\epsilon})\sub 1+R^{\circ\circ}$ and hence $\ka(t_{i})-1\in 
R^{\circ\circ}$ for $i=a+1,\ldots,b$. We may now choose $m$ such that 
$|\ka(t_{i})-1|_{sp} < |\vp|^{2/m}$ for $i=a+1,\ldots,b$ and 
$|p|_{sp}<|\vp|^{2/m}$. Arguing with $S=R[\vp^{1/m}]$ as in Lemma \ref{BTZp} we 
may construct a norm $|-|_{2}$ on $R$ for which $\vp$ is a multiplicative 
pseudo-uniformizer with $|\vp|_{2}=|\vp|$, $|\ka(t_{i})|_{2}\leq 1$ for 
$i=1,\ldots,a$, $|\ka(t_{i})-1|_{2}<1$ for $i=a+1,\ldots ,b$, and $|p|_{2}<1$. 
Setting $|-|^{\prime}:=|-|_{2}^{s}$ for $s$ sufficiently large then gives the 
desired norm.
\end{proof}

\begin{definition}
Let $R$ be a Banach--Tate $\Zp$-algebra and $\ka : T_{0} \ra R^{\times}$ a continuous character. We will say that the norm of $R$ is \emph{adapted to $\ka$} if $\ka(T_{0})\sub R_{0}$ and $|\ka(t)-1|<1$ for all $t\in T_{\epsilon}$. Note that then $|\ka(t)|=1$ for all $t\in T_{0}$ and there exists an $r<1$ such that $|\ka(t)-1|\leq r$ for all $t\in T_{\epsilon}$.
\end{definition}

For the rest of the subsection we consider a Banach--Tate $\Zp$-algebra $R$ and character $\ka : T_{0} \ra R$ such that the norm on $R$ is adapted to $\ka$. We extend $\ka$ to a character of $B_{0}$ by making it trivial on $N_{0}$. We define $\mc{A}_{\kappa}\sub \mc{C}(I,R)$ to be the subset of functions such that $f(gb)=\kappa(b)f(g)$ for all $g\in I$ and $b\in B_{0}$. $\mc{A}_{\kappa}$ is naturally a Banach $R$-module and carries a continuous right action of $I$ by left translation. Restricting a function from $I$ to $\ol{N}_{1}$ gives a topological isomorphism $\mc{A}_{\kappa} \cong \mc{C}(\ol{N}_{1},R)$. By definition, $\Sigma^{+}$ acts on the left on $\ol{N}_{1}$ by conjugation, and via the previous isomorphism this induces a right action of $\Sigma^{+}$ on $\mc{A}_{\kappa}$. These actions fit together into a right action of $\Delta_{p}$ on $\mc{A}_{\kappa}$. We let $\mc{D}_{\kappa}$ denote the dual $\Hom_{R,cts}(\mc{A}_{\kappa},R)$, equipped with the dual left $\Delta_{p}$-action. Since $\mc{A}_{\kappa}$ is the $B_{0}$-invariants of $\mc{C}(I,R)$ with respect to the action $(f.b)(g)=\kappa(b)^{-1}f(gb)$ ($b\in B_{0}$, $g\in I$), $\mc{D}_{\kappa}$ is the Hausdorff $B_{0}$-coinvariants of $\mc{D}(I,R)$ with respect to the dual (right) action. We record a more precise statement for future use:

\begin{proposition}\label{explicit}
The natural surjection $\mc{D}(I,R)\ra \mc{D}_{\kappa}$ is equivariant for the natural left $I$-actions on both sides. Identifying $\mc{D}_{\kappa}$ with $\mc{D}(\ol{N}_{1},R)$, the map is given by $\delta_{\bar{n}b} \mapsto \kappa(b) \delta_{\bar{n}}$.
\end{proposition}

\begin{proof}
The inverse to the restriction map $\mc{A}_{\kappa} \ra \mc{C}(\ol{N}_{1},R)$ is given by $f \mapsto (\bar{n}b \mapsto \kappa(b) f(\bar{n}) )$ (here and above $\bar{n}\in \ol{N}_{1}$ and $b\in B_{0}$). From this one sees directly that the dual map sends $\delta_{\bar{n}b}$ to $\kappa(b) \delta_{\bar{n}}$. That this characterises the maps follows from $R$-linearity and continuity for the weak-star topology.
\end{proof}

To apply the results from the previous subsection we will need to know that some groups are uniform. The following result is presumably well known to experts but we have been unable to find a suitable reference. The proof we give is due to Konstantin Ardakov, and we thank him for allowing us to include it here (any errors are due to the authors).  

\begin{proposition}\label{uniform}
$K_{s}$ and $\ol{N}_{s}$ are uniform for $s\geq \epsilon$.
\end{proposition}

\begin{proof}
By construction the groups are products $K_{s}=\prod_{v|p}K_{s,v}$, and similarly for $\ol{N}_{s,v}$, in a natural way, so it suffices to prove that each $K_{s,v}$ and $\ol{N}_{s,v}$ is uniform. Fix $v|p$ and let $s\geq \epsilon$. First assume that $\mbf{H}_v=\GL_{n/\oo_{F_{v}}}$. Then we have the usual matrix logarithm $\log : K_{s,v} \ra p^{s}M_{n}(\oo_{F_{v}})$ and exponential $\exp : p^{s}M_{n}(\oo_{F_{v}}) \ra K_{s,v}$. They converge and are inverse to each other (by our assumption on $s$) and the Lie algebra $p^{s}M_{n}(\oo_{F_{v}})$ is easily seen to be powerful by assumption, so by the definition of the correspondence between powerful Lie algebras and uniform pro-$p$ groups \cite[Theorem 9.10]{ddms} via the Campbell--Hausdorff series we see that $K_{s,v}$ is the uniform group corresponding to $p^{s}M_{n}(\oo_{F_{v}})$. To get the result for $\ol{N}_{s,v}$ we may by conjugation assume that $\ol{\mbf{N}}_{v}$ is the group of lower triangular unipotent matrices; the corresponding Lie algebra is that of lower triangular nilpotent matrices and we then argue similarly.

\medskip

Now let $\mbf{H}_v$ be arbitrary and choose a closed immersion $\mbf{H}_{v} 
\hookrightarrow \GL_{n/\oo_{F_{v}}}$ for some $n$. We thereby identify 
$\mbf{H}_{v}$ with a closed subgroup of $\GL_{n/\oo_{F_{v}}}$.  Writing 
$\mbf{B}_{v}^{\prime}$ for the upper triangular Borel of $\GL_{n/\oo_{F_{v}}}$, 
$\mbf{N}_{v}^{\prime}$ for its unipotent radical, $\mbf{T}_{v}^{\prime}$ and 
$\ol{\mbf{N}}_{v}^{\prime}$ for its opposite we may assume, after conjugating 
if necessary, that $\mbf{T}_{v} \sub \mbf{T}_{v}^{\prime}$ etc. We write 
$K_{s,v}^{\prime}$ etc. for the corresponding principal congruence subgroups. 
With this setup, we now give the rest of the proof for the $K_{s,v}$ only; the 
proof for $\ol{N}_{s,v}$ proceeds in the same way. Note that $K_{s,v} = 
\mbf{H}_{v}(\oo_{F_{v}}) \cap K_{s,v}^{\prime}$. By \cite[Theorem 4.5]{ddms} we 
see that $K_{s,v}^{\prime}$ is torsion-free and that it suffices to prove that 
$K_{s,v}$ is powerful, i.e. that $[K_{s,v},K_{s,v}] \sub K_{s,v}^{q}$, where 
$[K_{s,v},K_{s,v}]$ is the derived subgroup of $K_{s,v}$ and $K_{s,v}^{q}$ is 
the subgroup generated by the $q$-th powers of elements in $K_{s,v}$ (any 
compact $p$-adic analytic group is topologically finitely generated). We remark 
that it is easy to see that $[K_{s,v}^{\prime},K_{s,v}^{\prime}]\sub 
K_{s+\epsilon,v}^{\prime}=(K_{s,v}^{\prime})^{q}$. Using this and $K_{s,v} = 
\mbf{H}_{v}(\oo_{F_{v}}) \cap K_{s,v}^{\prime}$, we see that $[K_{s,v},K_{s,v}] 
\sub K_{s+\epsilon,v}^{\prime} \cap \mbf{H}_{v}(\oo_{F_{v}})=K_{s+\epsilon,v}$ 
and $K_{s,v}^{q} \sub K_{s+\epsilon,v}$.

\medskip

It remains to prove that $K_{s+\epsilon,v} \sub K_{s,v}^{q}$. We have $T_{s+\epsilon,v}=T_{s,v}^{q}$ using the logarithm and exponential ($\mbf{T}_{v}$ is a split torus). We have an isomorphism of $\oo_{F_{v}}$-schemes $\prod_{\al}x_{\al} : \prod_{\al} \mb{G}_{a} \overset{\sim}{\longrightarrow} \mbf{N}_{v}$ where $\al$ ranges through the roots of $\mbf{H}_{v}$ whose root subgroups are contained in $\mbf{N}_{v}$, and $x_{\al}$ is a corresponding root homomorphism. Under this isomorphism $N_{s,v}$ corresponds to $\prod_{\al}p^{s}\oo_{F_{v}}$, and by standard properties of $x_{\al}$ one has $x_{\al}(q.(p^{s}a))=x_{\al}(p^{s}a)^{q}$ for any $a\in \oo_{F_{v}}$. It follows that $N_{s+\epsilon,v} = N_{s,v}^{q}$. Similarly $\ol{N}_{s+\epsilon,v} = \ol{N}_{s,v}^{q}$. By the Iwahori decomposition we then see that $K_{s+\epsilon,v} \sub K_{s,v}^{q}$ as desired.
\end{proof}
\begin{remark}\label{unipcong}
	Note that the argument in the final paragraph of the above proof also implies that, for arbitrary $p$ and $s \ge 1$, we have $\ol{N}_{s+1,v} = \ol{N}_{s,v}^{p}$. 
\end{remark}

\begin{definition} Let $r\in [1/p,1)$.
\begin{enumerate}
\item We define a norm $||-||_{r}^{sub}$ on $\mc{D}_{\kappa}$ by transporting 
the norm $||-||_{\ol{N}_{\epsilon},r}$ (defined before Proposition \ref{aut2}) 
on $\mc{D}(\ol{N}_{1},R)$ to $\mc{D}_{\kappa}$ via the isomorphism 
$\mc{D}(\ol{N}_{1},R)\cong \mc{D}_{\kappa}$ obtained by restriction of 
functions from $I$ to $\ol{N}_{1}$.

\item We define a norm $||-||_{r}^{quot}$ on $\mc{D}_{\kappa}$ as the quotient 
norm induced from $||-||_{K_{\epsilon},r}$ (defined before Proposition 
\ref{aut2}) via the natural surjection $\mc{D}(I,R) \ra \mc{D}_{\kappa}$.
\end{enumerate}
\end{definition}

Note that $||-||_{r}^{quot}$ is $I$-invariant; this follows from $I$-equivariance of the surjection and Proposition \ref{aut2}. The following is the key result of this subsection:

\begin{proposition}
Suppose that $|\kappa(t)-1|\leq r$ for all $t \in T_{\epsilon}$. Then $||-||_{r}^{quot}=||-||_{r}^{sub}$ on $\mc{D}_{\kappa}$.
\end{proposition}

\begin{proof}
Let us first assume that $p\neq 2$. First we claim that $||-||_{r}^{quot}$ is equal to the quotient norm on $\mc{D}_{\kappa}$ coming from the surjection $\mc{D}(K_{1},R) \ra \mc{D}_{\kappa}$ and the norm $||-||_{r}$ on the source. To see this, pick a set $(b_{i})_{i}$ of coset representatives of $I/K_{1}$ lying in $B_{0}$ and define a map $\pi_\kappa:\mc{D}(I,R)=\bigoplus_{i}\delta_{b_{i}}\mc{D}(K_{1},R) \ra \mc{D}(K_{1},R)$ by
$$ (\delta_{b_{i}}\mu_{i})_{i} \mapsto \sum_{i} \kappa(b_{i})\mu_{i}. $$
We remark that composing this map with the natural surjection 
$\mc{D}(K_{1},R)\ra \mc{D}_{\kappa}$ gives the natural surjection $\mc{D}(I,R) 
\ra \mc{D}_{\kappa}$ (this follows from the explicit formula in Proposition 
\ref{explicit}). To prove the claim, it then suffices to prove that $||-||_{r}$ 
is the quotient norm of $||-||_{K_{1},r}$ via $\pi_\kappa$. Write 
$||-||_{r}^{\prime}$ for this quotient norm. For simplicity assume that $1$ is 
one of the coset representatives $b_i$. Then our map is a section of the 
inclusion $\mc{D}(K_{1},R) \sub \mc{D}(I,R)$. It is then clear that 
$||-||_{r}^{\prime} \leq (||-||_{K_{1},r})|_{\mc{D}(K_{1},r)}=||-||_{r}$. 
Conversely, if $\mu=\sum_{i}\kappa({b_{i}})\mu_{i} \in \mc{D}(K_{1},R)$ is the 
image of $\sum_{i}\delta_{b_{i}}\mu_{i} \in \mc{D}(I,R)$, then 
$$ ||\mu||_{r}\leq \sup_{i}||\kappa(b_{i})\mu_{i}||_{r} \leq \sup_{i} ||\mu_{i}||_{r} = || \sum_{i}\delta_{b_{i}}\mu_{i}||_{K_{1},r}.$$
Taking the infimum over such presentations we obtain $||-||_{r} \leq ||-||_{r}^{\prime}$, and hence equality.

\medskip
Next, let $\bar{n}_{1},\ldots ,\bar{n}_{k}$ (resp. $n_{1},\ldots,n_{k}$) be a 
minimal set of topological generators of $\ol{N}_{1}$ (resp. $N_{1}$), and let 
$t_{1},\ldots,t_{l}$ be a set of topological generators of $T_{1}$. Put 
$\mbf{n}^{\alpha}=\prod_{i}(\delta_{n_{i}}-1)^{\alpha_{i}}$ and similarly for 
$T_{1}$ and $\ol{N}_{1}$. By Proposition \ref{explicit}, the map 
$\mc{D}(K_{1},R) \ra \mc{D}_{\kappa}$ is then given by
$$ \sum_{\alpha,\beta,\gamma} d_{\alpha,\beta,\gamma}\bar{\mbf{n}}^{\alpha}\mbf{t}^{\beta}\mbf{n}^{\gamma} \mapsto \sum_{\alpha} \left( \sum_{\beta} d_{\alpha,\beta,\underline{0}} \prod_{i}(\kappa(t_{i})-1)^{\beta_{i}}\right) \bar{\mbf{n}}^{\alpha}.$$
We then make a computation similar to the one in the proof of the claim above. First, it's clear that $||-||_{r}^{quot}\leq ||-||_{r}^{sub}$ since restricting $||-||_{r}$ on $\mc{D}(K_{1},R)$ to $\mc{D}(\ol{N}_{1},R)$ gives (the intrinsically defined) $||-||_{r}$, and the composition $\mc{D}(\ol{N}_{1},R) \ra \mc{D}(K_{1},R) \ra D_{\kappa}\cong \mc{D}(\ol{N}_{1},R)$ is the identity. Second, if 
$$ \sum_{\alpha} e_{\alpha}\bar{\mbf{n}}^{\alpha} = \sum_{\alpha} \left( \sum_{\beta} d_{\alpha,\beta,\underline{0}} \prod_{i}(\kappa(t_{i})-1)^{\beta_{i}}\right) \bar{\mbf{n}}^{\alpha}, $$
then
$$ || \sum_{\alpha} e_{\alpha}\bar{\mbf{n}}^{\alpha} ||_{r} \leq \sup_{\alpha, \beta} \left( |d_{\alpha,\beta,\underline{0}}|\prod_{i}|\kappa(t_{i})-1|^{\beta_{i}}r^{|\alpha|} \right) \leq \sup_{\alpha,\beta}|d_{\alpha,\beta,\underline{0}}|r^{|\alpha|+|\beta|}\leq || \sum_{\alpha,\beta,\gamma} d_{\alpha,\beta,\gamma}\bar{\mbf{n}}^{\alpha}\mbf{t}^{\beta}\mbf{n}^{\gamma}||_{r}$$
where we have used the assumption $|\kappa(t_{i})-1|\leq r$ for all $i$ to obtain the second inequality. Hence, taking the infimum over such presentations, we see that $||-||_{r}^{sub}\leq ||-||_{r}^{quot}$ and equality follows.

\medskip

The case $p=2$ is similar. We identify $\D_{\ka}$ with $\D(\ol{N}_{1},R)$ and consider the subspace $\D(\ol{N}_{2},R)$. It carries the norm $||-||_{r}$ and also receives a quotient norm from the norm $||-||_{r}$ on $\D(K_{2},R)$ via the surjection $\varphi : \D(K_{2},R) \ra \D(\ol{N}_{2},R)$. These two norms are equal by the same type of argument as in the second part above. We then equip $\D(I,R)$ with the norm $||-||_{K_{2},r}$ and $\D(\ol{N}_{1},R)$ with the norm $||-||_{r}^{sub}$. Pick coset representatives $(\ol{n}_{i})_{i}$ of $\ol{N}_{1}/\ol{N}_{2}$ and $(b_{j})_{j}$ of $B_{0}/B_{2}$, both containing $1$. We may then write the map $\D(I,R) \ra \D(\ol{N}_{1},R)$ as
$$ \bigoplus_{i,j}\delta_{\ol{n}_{i}b_{j}}\D(K_{2},R) \ra \bigoplus_{i} \delta_{\ol{n}_{i}}\D(\ol{N}_{2},R); $$
$$ (\delta_{\ol{n}_{i}b_{j}}\mu_{ij})_{i,j} \mapsto \left( \delta_{\ol{n}_{i}}\sum_{j} \ka(b_{j})\varphi(\mu_{ij}) \right)_{i}. $$
By a computation similar to that in the first part of the proof in the case $p\neq 2$ (using additionally the equality of the two norms on $\D(\ol{N}_{2},R)$ asserted above) the norm $||-||_{r}^{quot}$ agrees with the norm $||-||_{r}^{sub}$, as desired.
\end{proof}

\begin{remark}
It might happen that $\ol{N}_{1}$ is uniform when $p=2$ (e.g. when $\mbf{G}={\rm Res}_{\Q}^{F}\GL_{2/F}$). In this case, the norm $||-||_{r}^{sub}$ is bounded-equivalent to $||-||_{r^{1/2}}$ on $\D(\ol{N}_{1},R)$.
\end{remark}

\begin{definition}
Write $r_{\kappa}$ for the minimal $r\in [1/p,1)$ such that $|\kappa(t)-1|\leq r$ for all $t \in T_\epsilon$. When $r\geq r_{\kappa}$, we write $||-||_{r}$ for the norm $||-||_{r}^{sub}=||-||_{r}^{quot}$ on $\mc{D}_{\kappa}$ (we will never consider these norms when $r<r_{\kappa}$).

Let $r\geq r_{\kappa}$. We define $\mc{D}^{r}_{\kappa}$ to be the completion of $\mc{D}_{\kappa}$ with respect to $||-||_{r}$, and let $\ms{D}_{\kappa}=\varprojlim_{r} \mc{D}_{\kappa}^{r}$.
\end{definition}

$\mc{D}_{\kappa}^{r}$ is a Banach $R$-module with respect to its induced norm, and carries a left $I$-action (since $I$ acts on $\mc{D}_{\kappa}$ by $||-||_{r}$-isometries, the action extends to the completion). When $R$ is a Banach $\Qp$-algebra, it follows from Proposition \ref{distrocomparison} that $\ms{D}_{\ka}$ is the locally analytic distribution module used in \cite{as,han}. We will also need to extend the action of $\Sigma^{+}$ to these modules, and prove that elements of $\Sigma^{cpt}$ give compact operators. For this, it is convenient to use the definition of $||-||_{r}$ as $||-||_{r}^{sub}$. We have a natural identification $\mc{D}_{\ka}^{r} \cong \mc{D}^{r}(\ol{N}_{1},R)$ when $p\neq 2$; when $p=2$ we have $\D_{\ka}^{r}\cong \bigoplus_{\ol{n}_{i}\in \ol{N}_{1}/\ol{N}_{2}} \delta_{\ol{n}_{i}}\D^{r}(\ol{N}_{2},R)$.

\begin{corollary}\label{cor:cpt}
If $t\in \Sigma^{+}$, then the action of $t$ on $\D_{\ka}$ is norm-decreasing with respect to $||-||_{r}$ for any $r\geq r_{\ka}$, and hence the action of $t$ extends to $\D_{\ka}^{r}$. If $t\in \Sigma^{cpt}$, then $t$ acts compactly on $\mc{D}_{\ka}^{r}$. Moreover, in this case the action of $t$ is given by the composition of a norm-decreasing map \[\mc{D}_{\ka}^{r} \rightarrow \mc{D}_{\ka}^{r^{1/p}}\] with the compact (norm-decreasing) inclusion \[\mc{D}_{\ka}^{r^{1/p}}\hookrightarrow \mc{D}_{\ka}^{r}.\]
\end{corollary}

\begin{proof}
We use the identification $\mc{D}_{\ka}\cong \mc{D}(\ol{N}_{1},R)$, with regards to which $t$ acts by the map induced from the homomorphism $\ol{N}_{1} \ra \ol{N}_{1}$ given by $\ol{n} \mapsto t\ol{n}t^{-1}$. First assume $p\neq 2$. The first assertion follows directly from Lemma \ref{cpt1}. The second assertion follows from the third, so it remains to prove the third assertion. If $t\in \Sigma^{cpt}$ we have $t\ol{N}_{1}t^{-1}\sub \ol{N}_{2}=(\ol{N}_{1})^{p}$ by the definition of $\Sigma^{cpt}$ and Remark \ref{unipcong}, so the third assertion follows from Lemmas \ref{cpt2} and \ref{cpt1}. 

Now assume $p=2$ and let $(\ol{n}_{i})_{i}$ be a set of coset representatives of $\ol{N}_{1}/\ol{N}_{2}$. We have $t\ol{N}_{2}t^{-1}\sub \ol{N}_{2}$ and $\ol{n} \mapsto t\ol{n}t^{-1}$ induces a map $\ol{N}_{1}/\ol{N}_{2} \ra \ol{N}_{1}/\ol{N}_{2}$; moreover if $t\in \Sigma^{cpt}$ then $t\ol{N}_{2}t^{-1} \sub \ol{N}_{3}$, by Remark \ref{unipcong}. Writing $\D(\ol{N}_{1},R) = \bigoplus_{i} \delta_{\ol{n}_{i}}\D(\ol{N}_{2},R)$ we see that $t$ acts by a sum of maps of the form $\D(\ol{N}_{2},R) \ra \D(\ol{N}_{2},R)$ induced by $\ol{n} \mapsto t\ol{n}t^{-1}$, and the proof now proceeds as in the case when $p\neq 2$.
\end{proof}

Continue to assume $r\geq r_{\ka}$. In the rest of this subsection, for 
simplicity of notation we will assume $p\neq 2$ in the discussion, but 
everything we do works for $p=2$ as well after minor adjustments, writing 
$\D_{\ka}^{r}\cong \bigoplus_{i} \delta_{\ol{n}_{i}}\D^{r}(\ol{N}_{2},R)$. Let 
$\ol{n}_{1},\ldots,\ol{n}_{t}$ be a topological basis of $\ol{N}_{1}$ and set 
$\mbf{n}_{i}=[\ol{n}_{i}]-1\in \D_{\ka}$ as usual. Then the description 
$$\D^{r}(\ol{N}_{1},R) = \wh{\bigoplus_{\al}}R.\vp^{-n(r,\vp,\al)}\mbf{n}^{\al}$$
gives us an explicit description of $\D_{\ka}^{r}$ via the identification $\D_{\ka}^{r}\cong \D^{r}(\ol{N}_{1},R)$. In particular we remark that the $\D_{\ka}^{r}$ are potentially ON-able with a countable potential ON-basis that we can actually write down. This is in contrast with the compact distribution modules considered in \cite{as}, which are potentially ON-able but one cannot write down an explicit basis (and the dimension is uncountable), as well as the distribution modules considered in \cite{han}, which are not known to be potentially ON-able in general. Our next goal is to introduce variants of the modules $\D_{\ka}^{r}$, as well as modules $\mc{A}_{\ka}^{r}\sub \mc{A}_{\ka}$, which are analogous to those considered in \cite{han}.

\medskip

Let $r>r_{\ka}$ and pick $s\in [r_{\ka},r)$. The unit ball $\D_{\ka}^{r,\circ}$ of $\D_{\ka}^{r}$ is $\De$-stable since $\De$ acts by norm-decreasing operators. The natural map $\D^{r}_{\ka} \ra \D^{s}_{\ka}$ is injective, and may naturally be thought of as an inclusion. Doing so, we define $\D_{\ka}^{<r,\circ}$ to be the closure of $\D_{\ka}^{r,\circ}$ inside $\D_{\ka}^{s}$. It is a $\vp$-torsion free $R_{0}$-module and we set $\D^{<r}_{\ka}=\D_{\ka}^{<r,\circ}[1/\vp]$; this is an $R$-module which naturally embeds into $\D_{\ka}^{s}$. Since $\D^{r,\circ}_{\ka}\sub \D^{s}_{\ka}$ is $\De$-stable we see that $\D_{\ka}^{<r,\circ}$ and $\D_{\ka}^{<r}$ are as well. We may then define
$$ \mc{A}_{\ka}^{<r,\circ}= \left\{ f \in \mc{A}_{\ka} \mid |\mu(f)| \leq 1 \, \forall \mu\in \D_{\ka}\cap \D_{\ka}^{r,\circ} \right\} $$
and $\mc{A}_{\ka}^{<r}=\mc{A}_{\ka}^{<r,\circ}[1/\vp]\sub \mc{A}_{\ka}$. Then 
$\mc{A}_{\ka}^{<r}$ is the dual space of $\D^{r}_{\ka}$. We equip it with the 
norm dual to $||-||_{r}$, and define $\mc{A}_{\ka}^{r}\sub \mc{A}_{\ka}^{<r}$ 
to be the closure of $\mc{A}_{\ka}^{<s}\sub \mc{A}_{\ka}^{<r}$ with respect to 
this norm. These spaces are $\De$-stable since $\D_{\ka}\cap 
\D_{\ka}^{r,\circ}$ is. Note that we have natural identifications 
$\D^{<r}_{\ka}\cong \D^{<r}(\ol{N}_{1},R)$, $\mc{A}_{\ka}^{<r}\cong 
\mc{C}^{<r}(\ol{N}_{1},R)$ and $\mc{A}_{\ka}^{r}\cong 
\mc{C}^{r}(\ol{N}_{1},R)$, so the discussion in \S \ref{completions} applies to 
give explicit descriptions of these spaces, and show that they are independent 
of the choice of $s$.

\section{Overconvergent cohomology and eigenvarieties}\label{evars}

In this section we establish the basic results on overconvergent cohomology needed to construct and analyse eigenvarieties. We retain the notation from \S \ref{modules}, but we will change our point of view slightly, from a functional-analytic point of view to a geometric one. Instead of working with Banach--Tate $\Zp$-algebras, we will work with complete Tate $\Zp$-algebras, which we will always assume to have a Noetherian ring of definition. A \emph{weight} will therefore be a continuous homomorphism $\ka : T_{0} \ra R^{\times}$, where $R$ is a complete Tate $\Zp$-algebra with a Noetherian ring of definition. We follow the strategy of Hansen \cite[\S 3-4]{han} to construct our eigenvarieties. A similar construction was also carried out by Xiang \cite{xia}.

\subsection{Eigenvarieties}\label{eigenvarieties}

We retain the global setup from the beginning of \S \ref{modules}. To construct our eigenvarieties, we will need some more notation as well as some concepts from \cite{as} and \cite{han}. First, let us fix a compact open subgroup $K_{\ell}=K_{\ell}\sub \mbf{G}(\mb{Q}_{\ell})$ for each prime $\ell\neq p$, which is hyperspecial for all but finitely many $\ell$, and set $K^{p}=\prod_{\ell\neq p}K_{\ell}$ (the tame level) and $K=K^{p}I$. We assume that $K$ is neat (which is the case when $K^{p}$ is sufficiently small)\footnote{In fact it suffices to assume that $K$ contains a neat open normal subgroup with index prime to $p$.}. Let $\mbf{Z}$ denote the center of $\mbf{G}$ and put $Z(K)=\mbf{Z}(\Q)\cap K$. All weights in this section will be assumed to be trivial on $\ol{Z(K)}\sub T_{0}$.

We also fix a monoid $\Delta_{\ell}\sub \mbf{G}(\mb{Q}_{\ell})$ containing 
$K_{\ell}$, which is equal to $\mbf{G}(\mb{Q}_{\ell})$ when $K_{\ell}$ is 
hyperspecial, such that $(\De_{\ell},K_{\ell})$ is a Hecke pair and the 
$\ell$-Hecke algebra $\mb{T}(\Delta_{\ell}, K_{\ell})$ (defined over $\Z_p$) is 
commutative.  Set $\Delta^{p}=\prod^{\prime}\Delta_{\ell}$ (restricted product 
with respect to the $K_{\ell}$) and $\Delta = \Delta^{p}\Delta_{p}$ (recall 
that $\Delta_{p}=I\Sigma^{+}I$). Next, we fix a choice $C_{\bu}(K,-)$ of an 
\emph{augmented Borel-Serre complex} as in \cite[\S 2.1]{han} and for any left 
$\Delta$-module $M$ we define $C^{\bu}(K,M)$ as in \cite[\S 2.1]{han}. 
$C^{\bu}(K,M)$ carries an action of the Hecke algebra $\mb{T}(\Delta,K)$.
In general, if $C^{\bu}$ is a cochain complex we let $C^{\ast}=\bigoplus_{i\in \Z}C^{i}$ and, similarly, we use $H^{\ast}$ to denote the direct sum of all cohomology groups when cohomology makes sense. 

\medskip

Fix once and for all an element $t\in \Sigma^{cpt}$. Let $\ka : T_{0} \ra R^{\times}$ be a weight, and choose a Banach--Tate $\Zp$-algebra norm on $R$ which is adapted to $\ka$. We let $\widetilde{U}_{\ka,r}=\widetilde{U}_{t,\ka,r}$ denote the corresponding Hecke operator on $C^{\bu}(K, \mc{D}^{r}_{\ka})$ (here $r\geq r_{\ka}$). This operator is compact and we let
$$ F_{\ka}^{r}(T) = \det\left(1 - T \wt{U}_{\ka,r} \mid C^{\ast}(K, \mc{D}_{\ka}^{r})\right) $$
denote its Fredholm determinant, which exists since $C^{\ast}(K, \mc{D}_{\ka}^{r})$ is potentially ON-able (by basic propoerties of Borel-Serre complexes). 

Before proceeding, let us recall the definition of weight space. 

\begin{definition}
	Suppose $(A,A^+)$ is a complete sheafy affinoid $(\Zp,\Zp)$-algebra. The functor  
	$$ (A,A^{+}) \mapsto  \Hom_{cts}( T_{0}/\ol{Z(K)}, A^{\times})$$  
	from complete sheafy affinoid $(\Zp,\Zp)$-algebras $(A,A^{+})$ to sets is representable by the affinoid ring $(\Zp[[T_{0}/\ol{Z(K)}]], \Zp[[T_{0}/\ol{Z(K)}]])$, and we let $\mf{W}$ denote the corresponding adic space. We remark that any continuous homomorphism $T_{0}/\ol{Z(K)} \ra A^{\times}$ automatically lands in $(A^{+})^{\times}$. To see this, note that $T_{0}$ is non-canonically isomorphic to $F \times \Zp^{r}$ as a $p$-adic Lie group, where $F$ is a finite group and $r\in \Z_{\geq 0}$. The image of $F$ lands in the roots of unity $\mu_{\infty}(A)$ in $A$, and the image of $\Zp^{r}$ lands in $1+A^{\circ\circ}$. Since $A^{+}$ is open and integrally closed in $A$ (which is complete), $\mu_{\infty}(A)$ and  $1+A^{\circ\circ}$ are both subsets of  $(A^{+})^{\times}$.
	
	We let $\mc{W}$ denote the \emph{analytic} locus of $\mf{W}$; this is an open subset. For any weight $\ka\, :\, T_{0} \ra R^{\times}$ and ring of integral elements $R^+ \sub R^\circ$, we obtain a map $\mc{U}=\Spa(R,R^+) \ra \mc{W}$. If this map is an open  immersion (so in particular we have $R^+ = R^\circ$, by Corollary \ref{su2}), we will conflate the weight $\ka$ and the open subset $\mc{U}\sub \mc{W}$ and refer to $\mc{U}$ as an \emph{open} weight. In this case, we will also replace $\ka$ by $\mc{U}$ in our notation, writing $\mc{D}_{\U}$ et cetera. 
\end{definition}

\begin{proposition}\label{err3.1.1}
	The Fredholm series $F_{\ka}^{r,|-|}$ is independent of the choice of $r\geq r_{\ka}$ and on the choice of norm $|-|$ on $R$.
\end{proposition}
\begin{proof}
	We assume $p>2$ to simplify the notation, the proof for $p=2$ is the same (except that one needs to write down the potential ON-basis below differently). Let $|-|_\pi$ and $|-|_\vp$ be two adapted Banach--Tate $\Zp$-algebra norms with multiplicative pseudouniformizers $\pi$ and $\vp$, respectively, and let $r,s \in [1/p,1)$ be such that $|\ka(t)-1|_\pi \leq r$ and $|\ka(t)-1|_\vp \leq s$ for all $t\in T_\epsilon$. Consider $\mc{D}_\ka^{r}$ and $\mc{D}_\ka^s$, where the former is formed using $|-|_\pi$ and the latter is formed using $|-|_\vp$. Identify $\mc{D}_\ka^{r}$ and $\mc{D}_\ka^s$ with $\mc{D}^r(\ol{N}_1,R)$ and $\mc{D}^r(\ol{N}_1,R)$, respectively, and choose a minimal set of topological generators $n_1, \dots, n_k$ of $\ol{N}_1$. Put $\mbf{n}^\al = \prod_i (\delta_{n_i} - 1)^{\al_i}$. Then $\mc{D}^r(\ol{N}_1,R)$ has a potential ON-basis $(\pi^{-n(r,\pi,\al)}\mbf{n}^\al)$ and $\mc{D}(\ol{N}_1,R)^s$ has a potential ON-basis $(\vp^{-n(r,\vp,\al)}\mbf{n}^\al)$, cf.~(\ref{ONbasis}). In particular, these bases differ by scalars and produce potential ON-bases of $C^{\ast}(K, \mc{D}_{\ka}^{r})$ and $C^{\ast}(K, \mc{D}_{\ka}^{s})$ which also differ by scalars. It then follows that the infinite matrices describing the actions $\wt{U}_{\ka,r}$ and $\wt{U}_{\ka,s}$ differ by the conjugation of an infinite \emph{diagonal} matrix. This then shows that $F_{\ka}^{r,|-|_\pi} = F_{\ka}^{s,|-|_\vp}$.
\end{proof}

In light of this we will from now on drop $r$ from the notation and simply write $F_{\ka}$. 

\begin{proposition}\label{3.1.2} Let $\ka\, :\, T_{0} \ra R^{\times}$ be a weight and choose an adapted Banach--Tate $\Zp$-algebra norm on $R$. Let $\vp$ be a multiplicative pseudo-uniformizer in $R$.
\begin{enumerate}
\item Assume that $F_{\ka}$ has a factorization $F_{\ka}=QS$ where $Q$ is a multiplicative polynomial, $S$ is a Fredholm series, and $Q$ and $S$ are relatively prime. Let $s>r\geq r_{\ka}$. The inclusion $C^{\bu}(K, \D_{\ka}^{s}) \sub C^{\bu}(K, \D_{\ka}^{r})$ induces an equality $\Ker^{\bu} Q^{\ast}(\wt{U}_{\ka,s}) = \Ker^{\bu} Q^{\ast}(\wt{U}_{\ka,r})$ (here and elsewhere we write $\Ker^{\bu} Q^{\ast}(\wt{U}_{\ka,r})$ for the complex with $i$-th term being the kernel of $Q^{\ast}(\wt{U}_{\ka,r})$ acting on $C^{i}(K, \D_{\ka}^{r})$).

\item Let $ R^{\prime}$ be a complete Tate ring with a Noetherian ring of definition, which we assume to be equipped with a Banach--Tate $\Zp$-algebra norm $|-|^{\prime}$ which induces the topology. Assume that we have a bounded homomorphism $\phi : R \ra R^{\prime}$ such that $|-|^{\prime}$ is adapted to $\ka^{\prime}=\ka \circ \phi$ and $\phi(\vp)$ is multiplicative for $|-|^{\prime}$. Then $F_{\ka^{\prime}}=\phi(F_{\ka})$.

If we assume moreover that $F_{\ka}$ has a factorization as in the previous part, we have a canonical isomorphism $(\Ker^{\bu}Q^{\ast}(\wt{U}_{\ka,r})) \otimes_{R}R^{\prime} \cong \Ker^{\bu}Q^{\ast}(\wt{U}_{\ka^{\prime},r})$.
\end{enumerate}
\end{proposition}

\begin{proof}
For assertion (1), we note that the general case follows from the case $s\leq r^{1/p}$. By Corollary \ref{cor:cpt}, $\wt{U}_{\ka,r}$ factors as
$$ C^{\ast}(K, \D_{\ka}^{r}) \overset{\Phi}{\ra} C^{\ast}(K, \D_{\ka}^{s}) \overset{\iota}{\ra} C^{\ast}(K, \D_{\ka}^{r})$$
where $\iota$ is induced by the natural compact inclusion $\D_{\ka}^{s}\hookrightarrow \D_{\ka}^{r}$. We have $\wt{U}_{\ka,s}=\Phi \circ \iota$. The result now follows from Lemma \ref{link}. For part (2), the first assertion follows from Lemma \ref{funct} and \cite[Lemma 2.13]{bu}, and the second assertion follows from Lemma \ref{funct} upon writing $C^{\bu}(K, \D_{\ka}^{r})= \Ker^{\bu} Q^{\ast}(\wt{U}_{\ka,r}) \oplus N^{\bu}$ as in Theorem \ref{riesz} since $Q^{\ast}(\wt{U}_{\ka,r})$ is invertible on $N^{\bu}$.
\end{proof}

\begin{remark}
	The simple proof of Proposition \ref{err3.1.1} is an observation of Daniel Gulotta; see the beginning of \cite[\S4.4]{gulottapub}. It replaces a more complicated argument in the published version of this paper, which had the additional disadvantage of making use of the incorrect version of Lemma \ref{norm}.
\end{remark}

From the preceding two propositions, we can immediately deduce the following corollary.
\begin{corollary}\label{global}
Let $\mc{U}_{1}\sub \mc{U}_{2}$ be open weights and let $\phi\, :\, \oo_{\mc{W}}(\mc{U}_{2}) \ra \oo_{\mc{W}}(\mc{U}_{1})$ be the induced map. Then $\phi(F_{\mc{U}_{2}})=F_{\mc{U}_{1}}$. Therefore, the Fredholm determinants $(F_{\mc{U}})_{\mc{U}}$, where $\mc{U}$ ranges over all open weights, glue together to a Fredholm series $F_{\mc{W}}\in \oo(\mc{W})\{\{T \}\} = \oo(\A^{1}_{\mc{W}})$.
\end{corollary}

We write $\wt{U}_{\ka}$ for the Hecke operator on $C^{\bu}(K, \ms{D}_{\ka})$ coming from our fixed $t\in \Sigma^{cpt}$.

\begin{proposition}\label{horizontal}
Let $\ka : T_{0} \ra R^{\times}$ be a weight. Assume that $F_{\ka}$ has a factorization $F_{\ka}=QS$ where $Q$ is a multiplicative polynomial, $S$ is a Fredholm series, and $Q$ and $S$ are relatively prime. Then $\Ker^{\bu} Q^{\ast}(\wt{U}_{\ka})$ is a complex of projective $R$-modules. If $\phi : R\ra R^{\prime}$ is a continuous homomorphism, where $R^{\prime}$ is a complete Tate ring with a Noetherian ring of definition, then we have a canonical isomorphism $(\Ker^{\bu}Q^{\ast}(\wt{U}_{\ka})) \otimes_{R}R^{\prime} \cong \Ker^{\bu}Q^{\ast}(\wt{U}_{\ka^{\prime}})$ where $\ka^{\prime}=\ka \circ \phi$.
\end{proposition}

\begin{proof}
We have $C^{\bu}(K, \ms{D}_{\ka})= \varprojlim_{r\geq r_{\ka}} C^{\bu}(K, \mc{D}^{r}_{\ka})$. Since $\wt{U}_\ka=\varprojlim_{r\geq r_{\ka}}\wt{U}_{\ka,r}$ the result follows from Proposition \ref{3.1.2} upon noting that we may choose a topologically nilpotent unit $\vp\in R$ and norms on $R$ and $R'$ such that $\phi$ and $\vp$ satisfies the assumptions of that Proposition (by Proposition \ref{normindep1}, the topologies on $\ms{D}_{\ka}$ and $\ms{D}_{\ka'}$ are independent of these choices).
\end{proof}

Next we study what happens when the factorization of $F_{\ka}$ changes. Keep the notation of Proposition \ref{horizontal}. We make $\Ker^{\bu} Q^{\ast}(\wt{U}_{\ka})$ into a complex of $R[T]/(Q(T))$-modules by letting $T$ act as $\wt{U}_\ka^{-1}$.

\begin{proposition}\label{vertical}
Let $\ka : T_{0} \ra R^{\times}$ be a weight. Assume that $F_{\ka}$ has two factorizations $F_{\ka}=Q_{1}S_{1}=Q_{2}S_{2}$ where the $Q_{i}$ are multiplicative polynomials, the $S_{i}$ are Fredholm series, and for each $i$ the $Q_{i}$ and $S_{i}$ are relatively prime. Assume further that $Q_{1}|Q_{2}$. Then we have a canonical isomorphism 
$$\Ker^{\bu}Q_{2}^{\ast}(\wt{U}_{\ka}) \otimes_{R[T]/(Q_{2})} R[T]/(Q_{1}) \cong \Ker^{\bu}(Q_{1}^{\ast}(\wt{U}_{\ka})).$$
\end{proposition}

\begin{proof}
Let $P$ be such that $Q_{2}=PQ_{1}$; $P$ is then a multiplicative polynomial and one checks that it is relatively prime to $Q_{1}$, so we may find polynomials $A,B\in R[T]$ such that $PA+Q_{1}B=1$. We then have $R[T]/(Q_{2}) \cong R[T]/(Q_{1}) \times R[T]/(P)$ and $PA\in R[T]/(Q_{2})$ corresponds to $(1,0)\in R[T]/(Q_{1}) \times R[T]/(P)$. The proposition now amounts to showing that $PA.\Ker^{\bu}Q_{2}^{\ast}(\wt{U}_{\ka})=\Ker^{\bu}Q_{1}^{\ast}(\wt{U}_{\ka})$. If $x\in \Ker^{\bu}Q_{2}^{\ast}(\wt{U}_{\ka})$, then 
$$ Q^{\ast}_{1}(\wt{U}_{\ka}).PAx= \wt{U}_{\ka}^{\deg Q_{1}}AQ_2x=0,$$
which gives us one inclusion. For the other, assume that $y\in \Ker^{\bu} Q_{1}^{\ast}(\wt{U}_{\ka})$. Note that $\deg PA =\deg Q_{1}B$. Then 
$$ y = \wt{U}_{\ka}^{-\deg PA}(P^{\ast}(\wt{U}_{\ka})A^{\ast}(\wt{U}_{\ka})+B^{\ast}(\wt{U}_{\ka})Q_{1}^{\ast}(\wt{U}_{\ka}))y = P^{\ast}(\wt{U}_{\ka})A^{\ast}(\wt{U}_{\ka})\wt{U}_{\ka}^{-\deg PA}y, $$
which gives us the other inclusion.
\end{proof}

We now return to the Fredholm series $F_{\mc{W}}$ from Corollary \ref{global}. Let $\ms{Z}\sub \A^{1}_{\mc{W}}$ denote its Fredholm hypersurface. Consider $\ms{C}ov(\ms{Z})$, the set of all open affinoid $V\sub Z$ such that $\pi(V)\sub X$ is open affinoid, $\oo(\pi(V))$ is Tate, and the map $\pi|_{V} : V \ra \pi(V)$ is finite of constant degree, where $\pi : \ms{Z} \ra \mc{W}$ is the projection map. For $V\in \ms{C}ov(\ms{Z})$, let us write $F_{\mc{W}}=Q_{V}S_{V}$ for the associated factorization of $F_{\mc{W}}$ from Theorem \ref{cover}.

\begin{corollary}
The assignment $V \mapsto \Ker^{\bu}Q^{\ast}_{V}(\wt{U}_{\pi(V)})$, with $V\in \ms{C}ov(\ms{Z})$, defines a bounded complex of coherent sheaves $\ms{K}^{\bu}$ on $\ms{Z}$.
\end{corollary}

\begin{proof}
$\ms{C}ov(\ms{Z})$ is an open cover of $\ms{Z}$ so we need to prove that whenever $V_{1}\sub V_{2}$ are elements of $\ms{C}ov(\ms{Z})$, we have $(\Ker^{\bu}Q^{\ast}_{V_{2}}(\wt{U}_{\pi(V_{2})}))\otimes_{\oo(V_{2})}\oo(V_{1}) \cong \Ker^{\bu}Q_{V_{1}}^{\ast}(\wt{U}_{\pi(V_{1})})$ canonically. Define $V_{3}=\pi|_{V_{2}}^{-1}(\pi(V_{1}))$. Then we have $V_{1}\sub V_{3} \sub V_{2}$, so it suffices to treat the inclusions $V_{1}\sub V_{3}$ and $V_{3}\sub V_{2}$. In the first case we have $\pi(V_{1})=\pi(V_{3})$ and the result follows from Proposition \ref{vertical}, since $V_{1}\sub V_{3}$ forces $Q_{V_{1}}|Q_{V_{3}}$. For the second we have $V_{3}= V_{2} \times_{\pi(V_{2})} \pi(V_{1})$ and the result follows from Proposition \ref{horizontal}. 
\end{proof}

This allows us to finish the construction of the eigenvariety. We define $\ms{H}^{\ast}=H^{\ast}(\ms{K}^{\bu})$; this is a coherent sheaf on $\ms{Z}$. Since the projectors $C^{\bu}(K, \ms{D}_{\pi(V)}) \ra \Ker^{\bu}Q_{V}^{\ast}(\wt{U}_{\pi(V)})$ commute with the action of $\T(\De^{p},K^{p})$ (by construction, using the assertion about the projectors in Theorem \ref{riesz}), we get an induced action $\T(\De^{p},K^{p}) \ra \ms{E}nd_{\ms{Z}}(\ms{H}^{\ast})$. Let $\ms{T}\sub \ms{E}nd_{\ms{Z}}(\ms{H}^{\ast})$ denote the sub-presheaf generated over $\oo_{\ms{Z}}$ by the image of $\T(\De^{p},K^{p})$. It is a sheaf by flatness of rational localization, hence a coherent sheaf of $\oo_{\ms{Z}}$-algebras, and we define the eigenvariety $\ms{X}=\ms{X}_{\mbf{G},K^{p}}$ to be the relative $\underline{\Spa}(\ms{T},\ms{T}^{\circ}) \ra \ms{Z}$ (note that the sheaf of integral elements is determined by Lemma \ref{finitepowerbounded}). The morphism $q : \ms{X} \ra \ms{Z}$ is finite by construction, and we have
$$ \oo(q^{-1}(V))= {\rm Im}\left(\T(\De^{p},K^{p})\otimes_{\Zp}\oo(V) \ra \End_{\oo(V)}(H^{\ast}(\Ker^{\bu}Q^{\ast}_{V})) \right)$$
for all $V\in \ms{C}ov(\ms{Z})$. In particular, if $(\mc{U},h)$ is a slope datum for $(\mc{W},F_{\mc{W}})$, we write $\ms{T}_{\U,h}=\oo(q^{-1}(\ms{Z}_{\U,h}))$ and have
$$ \ms{T}_{\U,h}={\rm Im}\left(\T(\De^{p},K^{p})\otimes_{\Zp}\oo(\ms{Z}_{\U,h}) \ra \End_{\oo(\ms{Z}_{\U,h})}(H^{\ast}(K, \ms{D}_{\U})_{\leq h}) \right). $$

\begin{remark}
Our eigenvariety $\ms{X}$ contains the eigenvariety constructed by Hansen in \cite[\S 4]{han} as the open subset $\{p\neq 0\}$. Indeed, our construction specializes to his over Banach $\Qp$-algebras, with the minor difference that we use the complexes $C^{\bu}(K, \D_{\U}^{r})$ to construct the auxiliary Fredholm hypersurface $\ms{Z}$, whereas the complexes $C_{\bu}(K, \ms{A}_{\U}^{r})$ (in our notation) are used in \cite{han}, giving a different auxiliary Fredholm hypersuface. However, working over the union of the two Fredholm hypersurfaces, one sees that the coherent sheaf $\ms{H}^{\ast}$ on $\A^{1}_{\mc{W}^{rig}}$, with its $\T(\De^{p},K^{p})$-action, is equal to the sheaf $\ms{M}^{\ast}$ on $\A^{1}_{\mc{W}^{rig}}$ (in the notation of \cite[\S 4.3]{han}), with its $\T(\De^{p},K^{p})$-action (here we have used $\mc{W}^{rig}$ to denote the locus $\{p\neq 0\}\sub \mc{W}$).
\end{remark}

\begin{remark}
Like the other constructions, our construction of overconvergent cohomology and eigenvarieties has numerous variations, which are sometimes useful to keep in mind. For example, one may use compactly supported cohomology, homology or Borel-Moore homology instead (cf. \cite[\S 3.3]{han}), and/or one could use the modules $\ms{A}_{\ka}$ instead of the $\ms{D}_{\ka}$. One can also add (or remove) Hecke operators, or work over some restricted family of weights, rather than the universal one.
\end{remark}

\subsection{The Tor-spectral sequence} 

We now give the analogue of the Tor spectral sequence in \cite[Theorem 3.3.1]{han}, which is a key tool in analyzing the eigenvarieties. We phrase it in terms of slope decompositions and Banach--Tate rings, though we could have formulated it more generally for elements in $\ms{C}ov(\ms{Z})$.

\begin{theorem}
Let $h\in \Q_{\geq 0}$ and let $\ka\, :\, T_{0} \ra R^{\times}$ be a weight. We fix an adapted Banach--Tate $\Zp$-algebra norm on $R$, and suppose that $C^{\bu}(K, \D_{\ka}^{r})$ has a slope-$\leq h$ decomposition for some $r\geq r_{\ka}$. Let $R \ra S$ be a bounded homomorphism of Banach--Tate $\Zp$-algebras with adapted norms and write $\ka_{S}$ for the induced weight $T_{0} \ra S^{\times}$. Then there is a convergent Hecke-equivariant (cohomological) second quadrant spectral sequence
$$ E_{2}^{ij} = \Tor_{-i}^{R}(H^{j}(K, \ms{D}_{\ka})_{\leq h}, S) \implies H^{i+j}(K, \ms{D}_{\ka_{S}})_{\leq h}. $$
\end{theorem}

\begin{proof}
We follow the proof of \cite[Theorem 3.3.1]{han}. Define a chain complex $\ms{C}_{\bu}$ by $\ms{C}_{i}=C^{-i}(K, \ms{D}_{\ka})_{\leq h}$ and the obvious differentials (i.e. we are just reindexing and viewing $C^{\bu}(K,\ms{D}_{\ka})_{\leq h}$ as chain complex). This is a bounded chain complex of finite projective $R$-modules. Thus
$$ \mbf{Tor}_{i+j}^{R}(\ms{C}_{\bu}, S) = H_{i+j}( \ms{C}_{\bu} \otimes_{R}S) $$
where $\mbf{Tor}$ denotes hyper-$\Tor$. The hyper-$\Tor$ spectral sequence then gives us a homological spectral sequence
$$ E_{ij}^{2}=\Tor_{i}^{R}(H_{j}(\ms{C}_{\bu}), S) \implies H_{i+j}( \ms{C}_{\bu} \otimes_{R}S) $$
which is concentrated in the fourth quadrant. Reindexing we may turn this into a cohomological spectral sequence (cf. \cite[Dual definition 5.2.3]{wei})
$$ E_{2}^{ij}=\Tor_{-i}^{R}(H_{-j}(\ms{C}_{\bu}), S) \implies H_{-i-j}(\ms{C}_{\bu} \otimes_{R}S) $$
which is concentrated in the second quadrant. Since $H_{-j}(\ms{C}_{\bu})=H^{j}(K, \ms{D}_{\ka})_{\leq h}$ (by definition) and $H_{-i-j}(\ms{C}_{\bu}\otimes_{R}S) \cong H^{i+j}(K, \ms{D}_{\ka_{S}})_{\leq h}$ (canonically, by our previous results) this gives the desired spectral sequence. Finally, Hecke-equivariance follows from the functoriality of the hyper-Tor spectral sequence.
\end{proof}

\medskip

As an application, we prove the following analogue of \cite[Theorem 4.3.3]{han}. We will use it in the next section when we construct Galois representations.

\begin{proposition}\label{points1}
Let $(\mc{U},h)$ be a slope datum and let $\mf{m}\sub \oo_{\mc{W}}(\mc{U})$ be a maximal ideal corresponding to a weight $\ka : T_{0} \ra L^{\times}$ where $L=\oo_{\mc{W}}(\mc{U})/\mf{m}$; this is a local field by Proposition \ref{points}. Fix an absolute value on $L$ with $|p|\le p^{-1}$ (i.e.~an adapted Banach--Tate $\Zp$-algebra norm). Let $\T \sub \T(\De,K)$ be a $\Zp$-subalgebra and put
$$ \T_{\ka, h}= {\rm Im}(\T\otimes_{\Zp}L \ra \End_{L}(H^{\ast}(K, \ms{D}_{\ka})_{\leq h})); $$
$$  \T_{\mc{U}, h}= {\rm Im}(\T\otimes_{\Zp}\oo_{\mc{W}}(\mc{U}) \ra \End_{\oo_{W}(\mc{U})}(H^{\ast}(K, \ms{D}_{\mc{U}})_{\leq h})). $$
Then there is a natural isomorphism $(\T_{\mc{U},h}\otimes_{\oo_{\mc{W}}(\mc{U})}L)^{red} \cong \T_{\ka, h}^{red}$.
\end{proposition}

\begin{proof}
By the $\Tor$-spectral sequence we see that, if $T\in \T$ acts as $0$ on $H^{\ast}(K, \ms{D}_{\U})_{\leq h}$, then it acts nilpotently on $H^{\ast}(K, \ms{D}_{\ka})_{\leq h}$. It follows that we have a surjection $\T_{\U,h}\otimes_{\oo_{\mc{W}}(\mc{U})}L \twoheadrightarrow \T_{\ka, h}^{red}$ of finite-dimensional commutative $L$-algebras. To finish the proof it suffices to show that if $\mf{q}$ is a maximal ideal of $\T_{\U,h} \otimes_{\oo(\U)}L$ then the localization $(\T_{\ka,h})_{\mf{q}}$ is non-zero. Let $j$ be maximal such that $(H^{j}(K, \ms{D}_{\U})_{\leq h})_{\mf{q}}\neq 0$ and localize the entire $\Tor$-spectral sequence with respect to $\mf{q}$. Then the entry $(E_{2}^{0,j})_{\mf{q}}$ is stable  (i.e. $(E_{2}^{0,j})_{\mf{q}}=(E_{\infty}^{0,j})_{\mf{q}}$) and it follows that $(H^{j}(K, \ms{D}_{\ka})_{\leq h})_{\mf{q}}\neq 0$. Thus we must have $(\T_{\ka,h})_{\mf{q}}\neq 0$ as desired.
\end{proof} 

\begin{corollary}\label{points2}
We retain the notation of Proposition \ref{points1}. If we let $U_{t}$ be the double coset operator $[KtK]\in \T(\De,K)$ for our fixed $t\in \Sigma^{cpt}$, consider the commutative subalgebra $\T(\De^{p},K^{p})[U_{t}]\sub \T(\De,K)$. Then we have a natural isomorphism $(\ms{T}_{\U,h}\otimes_{\oo(\U)}L)^{red}\cong \T(\De^{p},K^{p})[U_{t}]_{\ka,h}^{red}$.
\end{corollary}

\begin{proof}
By Proposition \ref{points1} we have a natural isomorphism 
$$(\T(\De^{p},K^{p})[U_{t}]_{\U,h}\otimes_{\oo(\U)}L)^{red}\cong \T(\De^{p},K^{p})[U_{t}]_{\ka,h}^{red}$$
so it suffices to show that $\ms{T}_{\U,h} \cong \T(\De^{p},K^{p})[U_{t}]_{\U,h}$, which is clear from the definitions (note that $\End_{\oo(\ms{Z}_{\U,h})}(H^{\ast}(K,\ms{D}_{\U})_{\leq h}) \sub \End_{\oo(\U)}(H^{\ast}(K,\ms{D}_{\U})_{\leq h})$).
\end{proof}

\section{Galois representations}\label{galrep}

We continue to assume that all weights are trivial on $\ol{Z(K)}$. Let $\mbf{G}={\rm Res}_{\Q}^{F}\GL_{n/F}$, with $F$ a totally real or CM field. In this section we construct a Galois determinant (in the language of \cite{che}) valued in the global sections of the \emph{reduced} extended eigenvariety, satisfying the expected compatibility between Hecke eigenvalues and the characteristic polynomial of Frobenius at all unramified primes. The construction is an adaptation of a construction due to the first author and David Hansen, to appear in \cite{ch} (in a slightly refined form), which produces a Galois determinant over the reduced rigid eigenvariety as constructed in \cite{han}. The key step is to produce the desired Galois determinant for all `points' of the extended eigenvariety. In the rigid analytic setting one can then glue these individual determinants together by an argument due to Bellaiche and Chenevier; we prove a version of this gluing technique in our setting. We will not assume that $\mbf{G}={\rm Res}_{\Q}^{F}\GL_{n/F}$ until \S \ref{5.3}.

\subsection{Filtrations} Let $\ka : T_{0} \ra L^{\times}$ be a weight, where $L$ is a local field equipped with an adapted absolute value; we let $\vp$ be a uniformizer of $L$. Let $r>r_{\ka}$ and choose an auxiliary $s\in (r_{\ka},r)$. In this subsection, we construct a filtration on the unit ball $\D_{\ka}^{<r,\circ}$ of $\D_{\ka}^{<r}$ with finite graded pieces, generalizing the filtrations constructed by Hansen in \cite{han2}. We define
$$ \Fil^{j}\D_{\ka}^{<r,\circ} := \D_{\ka}^{<r,\circ}\cap \vp^{j}\D_{\ka}^{s,\circ}. $$
When $\ka$ and $r$ are clear from the context we will simply write $\Fil^{j}$ for  $\Fil^{j}\D_{\ka}^{<r,\circ}$ (we will always omit the choice of $s$). By the definitions the $\Fil^{j}$ are open and closed in the subspace topology on $\D^{<r,\circ}_{\ka}$ coming from $\D^{s}_{\ka}$ (we recall that $\D_{\ka}^{<r,\circ}$ is compact with respect to this subspace topology since $\oo_{L}$ is compact). Therefore the $\D_{\ka}^{<r,\circ}/\Fil^{j}$ are finite discrete $\oo_{L}$-torsion modules and we have $\D_{\ka}^{<r,\circ} = \varprojlim_{j} \D_{\ka}^{<r,\circ}/\Fil^{j}$ topologically. Note that $\Fil^{j}$ is $\De$-stable (being the intersection of two $\De$-stable subsets in $\D^{s}_{\ka}$), so the  $\oo_{L}$-torsion modules $\D_{\ka}^{<r,\circ}/\Fil^{j}$ inherit a $\De$-action and the equality in the previous sentence is $\De$-equivariant. We also record the following lemma. 

\begin{lemma}
We have $H^{i}(K, \D_{\ka}^{<r,\circ})=\varprojlim_{j} H^{i}(K, \D_{\ka}^{<r,\circ}/\Fil^{j})$ for all $i$.
\end{lemma}

\begin{proof}
On Borel-Serre complexes we have $C^{\bu}(K, \D_{\ka}^{<r,\circ})=\varprojlim_{j} C^{\bu}(K, \D_{\ka}^{<r,\circ}/\Fil^{j})$ and these are bounded complexes of compact abelian Hausdorff groups. This category is abelian and has exact inverse limits (e.g. by \cite[Proposition IV.2.7]{neu}), which gives us the result.
\end{proof}

We remark that $C^{\bu}(K, \D_{\ka}^{<r})$ has a slope-$\leq h$ decomposition for any $h\in \Q_{\geq 0}$ (since we are working over the field $L$) and $C^{\bu}(K, \D_{\ka}^{<r})_{\leq h} = C^{\bu}(K, \D_{\ka}^{r})_{\leq h}$ (since it is sandwiched between $C^{\bu}(K, \D_{\ka}^{r})_{\leq h}$ and $C^{\bu}(K, \D_{\ka}^{s})_{\leq h}$, and these are equal). Thus, we may use the $\D_{\ka}^{<r}$ instead of the $\D_{\ka}^{r}$ to define overconvergent cohomology. We will do this to define Galois representations because of the fact that $\D_{\ka}^{<r,\circ}$ is profinite with respect to topology coming from the $\Fil^{j}$.

\subsection{A morphism of Hecke algebras} We continue with the notation of the previous subsection. In this subsection we will fix a $\Zp$-subalgebra $\T \sub \T(\De,K)$ and, if $A$ is any $\Zp$-algebra, we will write $\T_{A}$ for $\T \otimes_{\Zp}A$. Then $\T_{\oo_{L}}$ acts on $H^{\ast}(K, \D_{\ka}^{<r,\circ}/\Fil^{j})$ for all $r$ and $j$ and we set
$$ \T^{r}_{\ka,j} = {\rm Im}\left( \T_{\oo_{L}} \ra \End_{\oo_{L}}(H^{\ast}(K, \D_{\ka}^{<r,\circ}/\Fil^{j}) \right); $$
$$ \T^{r}_{\ka, \Fil} = {\rm Im}\left( \T_{\oo_{L}} \ra \prod_{j}\T_{\ka, j}^{r} \right). $$
We equip each $\T_{\ka,j}^{r}$ with the discrete topology (they are finite rings) and give $\prod_{j}\T_{\ka,j}^{r}$ the product topology; we then give $\T_{\ka,\Fil}^{r}$ the subspace topology. We let $\wh{\T}_{\ka,\Fil}^{r}$ denote the completion of $\T_{\ka, \Fil}^{r}$ in $\prod_{j}\T_{\ka,j}^{r}$; this is a compact Hausdorff ring.

\medskip

Fix $h\in \Q_{\geq 0}$. We have natural maps $H^{\ast}(K, \D_{\ka}^{<r, \circ}) \ra H^{\ast}(K, \D_{\ka}^{<r}) \ra H^{\ast}(K, \D_{\ka}^{<r})_{\leq h}$ and we define $H^{\ast}(K, \D_{\ka}^{<r,\circ})_{\leq h}$ to be the image of the composition. 
 
\begin{lemma}\label{finite}
$H^{\ast}(K, \D_{\ka}^{<r,\circ})_{\leq h}$ is an open and bounded $\oo_{L}$-submodule of $H^{\ast}(K, \D_{\ka}^{<r})_{\leq h}$ (and hence a finite free $\oo_{L}$-module).
\end{lemma}

\begin{proof}
It is an $\oo_{L}$-submodule which spans $H^{\ast}(K, \D_{\ka}^{<r})_{\leq h}$ essentially by construction, so it suffices to show that it is finitely generated. The morphisms $H^{\ast}(K, \D_{\ka}^{<r, \circ}) \ra H^{\ast}(K, \D_{\ka}^{<r}) \ra H^{\ast}(K, \D_{\ka}^{<r})_{\leq h}$ are induced by the morphisms
$$ C^{\bu}(K, \D_{\ka}^{<r,\circ}) \ra C^{\bu}(K, \D_{\ka}^{<r}) \ra C^{\bu}(K, \D_{\ka}^{<r})_{\leq h} $$
of complexes. Let $C^{\bu}(K, \D_{\ka}^{<r,\circ})_{\leq h}$ denote the image of the composition. Note that each $C^{i}(K, \D_{\ka}^{<r,\circ})$ is bounded in $C^{i}(K, \D_{\ka}^{<r})$, so $C^{i}(K, \D_{\ka}^{<r,\circ})_{\leq h}$ is bounded in $C^{i}(K, \D_{\ka}^{<r})_{\leq h}$ by continuity of the projection. Thus $C^{\bu}(K, \D_{\ka}^{<r,\circ})_{\leq h}$ is a bounded complex of finite free $\oo_{L}$-modules. Since $H^{\ast}(K, \D_{\ka}^{<r,\circ})_{\leq h} \sub {\rm Im}(H^{\ast}(C^{\bu}(K, \D_{\ka}^{<r,\circ})_{\leq h}) \ra H^{\ast}(K, \D_{\ka}^{<r})_{\leq h})$, finite generation of the former follows.
\end{proof}

We define $\T_{\ka, \leq h}^{r,\circ}= {\rm Im}(\T_{\oo_{L}} \ra \End_{\oo_{L}}(H^{\ast}(K, \D_{\ka}^{<r,\circ})_{\leq h}))$; by the above Lemma this is a finite $\oo_{L}$-algebra and hence naturally a compact Hausdorff ring. Note that if $T\in \T_{\oo_{L}}$ is $0$ in $\T_{\ka,\Fil}^{r}$, i.e. acts as $0$ on all $H^{\ast}(K, \D_{\ka}^{<r,\circ}/\Fil^{j})$, then it acts as $0$ on $H^{\ast}(K, \D_{\ka}^{<r,\circ})_{\leq h}$ and so is $0$ in $\T_{\ka,\leq h}^{r,\circ}$. In other words we have a natural (surjective) map $\T_{\ka,\Fil}^{r} \ra \T_{\ka,\leq h}^{r,\circ}$. The goal of this section is to show that this map is continuous and so extends to the completion $\wh{\T}_{\ka,\Fil}^{r}$. 

\medskip

To do this we introduce some special open sets. Let $pr_{j} : \T_{\ka,\Fil}^{r} \ra \T_{\ka,j}^{r}$ denote the projection. We put $U_{j}=\Ker(pr_{j})$; this is an open ideal of $\T_{\ka, \Fil}^{r}$. On $\T_{\ka,\leq h}^{r,\circ}$ the opens that we will use are more delicate to construct. We have a commutative diagram
$$
\xymatrix{ H^{\ast}(K, \D^{<r,\circ}_{\ka}) \ar[r]\ar[d] & H^{\ast}(K, \D^{<r}_{\ka})_{\leq h} \ar[d]\\
H^{\ast}(K, \D^{s,\circ}_{\ka}) \ar[r] & H^{\ast}(K, \D^{s}_{\ka})_{\leq h}}
$$
where the right vertical map is an isomorphism, which we think of as an equality (the reader may trace through the definitions and see that this is a natural thing to do). Defining $H^{\ast}(K, \D^{s,\circ}_{\ka})_{\leq h}$ to be the image of $H^{\ast}(K, \D^{s,\circ}_{\ka}) \ra H^{\ast}(K, \D^{s}_{\ka})_{\leq h}$ we see that $H^{\ast}(K, \D^{<r,\circ}_{\ka})_{\leq h} \sub H^{\ast}(K, \D^{s,\circ}_{\ka})_{\leq h}$ are both open lattices in $H^{\ast}(K, \D^{<r}_{\ka})_{\leq h}=H^{\ast}(K, \D^{s}_{\ka})_{\leq h}$ (Lemma \ref{finite} holds for $H^{\ast}(K, \D^{s,\circ}_{\ka})_{\leq h}$ as well, with the same proof). 

\begin{lemma}
The ideals $V_{j}=\{ T\in \T_{\ka,\leq h}^{r,\circ} \mid T(H^{\ast}(K, \D_{\ka}^{<r,\circ})_{\leq h}) \sub \vp^{j}H^{\ast}(K, \D_{\ka}^{s,\circ})_{\leq h} \}$ form a basis of open neighbourhoods of $0$ in $\T_{\ka,\leq h}^{r, \circ}$.
\end{lemma}

\begin{proof}
It's easy to check that the $V_{j}$ are ideals. By the preceding remarks the subgroups $\vp^{j}H^{\ast}(K, \D_{\ka}^{s,\circ})_{\leq h}$ form a basis of neighbourhoods of $0$ in $H^{\ast}(K, \D_{\ka}^{<r,\circ})_{\leq h}$ (for $j \gg 0$). Using this the lemma is elementary.
\end{proof}

We may now prove continuity. Denote the map $\T_{\ka,\Fil}^{r} \ra \T_{\ka, \leq h}^{r,\circ}$ by $\phi$.

\begin{proposition}\label{cty}
We have $\phi(U_{j})\sub V_{j}$. Thus $\phi$ is continuous and extends to a continuous map $\wh{\T}_{\ka,\Fil}^{r} \ra \T_{\ka, \leq h}^{r, \circ}$.
\end{proposition}

\begin{proof}
The second statement follows directly from the first (by general properties of linearly topologized groups and the fact that $\T_{\ka, \leq h}^{r,\circ}$ is complete). The first statement amounts, by the definitions, to proving that if $T$ acts as $0$ on $H^{\ast}(K, \D_{\ka}^{<r,\circ}/\Fil^{j})$, then it maps $H^{\ast}(K, \D_{\ka}^{<r,\circ})_{\leq h}$ into $\vp^{j}H^{\ast}(K, \D_{\ka}^{s,\circ})_{\leq h}$, or in other words that $T$ acts as $0$ on
$$ {\rm Im}(H^{\ast}(K, \D_{\ka}^{<r,\circ})_{\leq h} \ra H^{\ast}(K, \D_{\ka}^{s,\circ})_{\leq h}/\vp^{j} ).$$
By definition this image is equal to 
$$ {\rm Im}(H^{\ast}(K, \D_{\ka}^{<r,\circ}) \ra H^{\ast}(K, \D_{\ka}^{s,\circ})_{\leq h}/\vp^{j} ). $$
Assume now that $T$ acts as $0$ on $H^{\ast}(K, \D_{\ka}^{<r,\circ}/\Fil^{j})$. We may factor $H^{\ast}(K, \D_{\ka}^{<r,\circ}) \ra H^{\ast}(K, \D_{\ka}^{s,\circ})_{\leq h}/\vp^{j}$ through $H^{\ast}(K, \D_{\ka}^{s,\circ})/\vp^{j}$ so it suffices to prove that $T$ acts as $0$ on 
$$ {\rm Im}(H^{\ast}(K, \D_{\ka}^{<r,\circ}) \ra H^{\ast}(K, \D_{\ka}^{s,\circ})/\vp^{j} ). $$
From the long exact sequence attached to the short exact sequence 
$$ 0 \ra \D_{\ka}^{s,\circ} \overset{\vp^{j}}{\ra} \D_{\ka}^{s,\circ} \ra \D_{\ka}^{s, \circ}/\vp^{j} \ra 0$$
we see that 
$$ H^{\ast}(K, \D_{\ka}^{s,\circ})/\vp^{j} \hookrightarrow H^{\ast}(K, \D_{\ka}^{s,\circ}/\vp^{j}). $$
By definition, we have a natural map $\D_{\ka}^{<r, \circ}/\Fil^{j} \ra \D_{\ka}^{s,\circ}/\vp^{j}$. Assembling the last few sentences, we have a commutative diagram
$$
\xymatrix{ H^{\ast}(K, \D^{<r,\circ}_{\ka}) \ar[r]\ar[d] & H^{\ast}(K, \D^{s,\circ}_{\ka})/\vp^{j} \ar[d]\\
H^{\ast}(K, \D^{<r,\circ}_{\ka}/\Fil^{j}) \ar[r] & H^{\ast}(K, \D^{s,\circ}_{\ka}/\vp^{j})}
$$
where the right vertical map is injective. By assumption $T$ acts as $0$ on $H^{\ast}(K, \D^{<r,\circ}_{\ka}/\Fil^{j})$, hence it acts as $0$ on ${\rm Im}(H^{\ast}(K, \D^{<r,\circ}_{\ka}) \ra H^{\ast}(K, \D^{s,\circ}_{\ka}/\vp^{j}))$. But, since the right vertical map is injective, this image is isomorphic to ${\rm Im}(H^{\ast}(K, \D_{\ka}^{<r, \circ}) \ra H^{\ast}(K, \D_{\ka}^{s,\circ})/\vp^{j})$. Thus $T$ acts as $0$ on this, which is what we wanted to prove.
\end{proof}

\subsection{Galois representations}\label{5.3}
We now specialize to the case $\mbf{G}={\rm Res}_{\Q}^{F}\GL_{n/F}$ where $F$ is a totally real or CM number field, and $n\geq 2$. When discussing Galois representations we will, for simplicity of referencing, use the same conventions as in \cite[\S 5]{sch}. Let $S^{\prime}$ denote the set of places $w$ of $\Q$ such that either $w=\infty$, or if $w$ is finite then $K_{w}$ is \emph{not} hyperspecial. This is a finite set containing $p$ and the primes which ramify in $F$. We let $S$ denote the set of places in $F$ lying above those in $S^{\prime}$. We set
$$ \T = \bigotimes_{v\notin S} \T_{v} $$
where $\T_{v}= \Zp[ \GL_{n}(F_{v})//\GL_{n}(\oo_{F_{v}})]$ is the usual 
spherical Hecke algebra (we assume that $K\sub 
\GL_{n}(\oo_{F}\otimes_{\Z}\wh{\Z})$).  We have $\T \sub \T(\De,K)$ and we will 
use the notation and results of the previous subsection for this choice of 
$\T$. Let $q_{v}$ be the size of the residue field at $v$. We have the 
(unnormalized) Satake isomorphism
$$ \T_{v}[q_{v}^{1/2}] \cong \Zp[q_{v}^{1/2}][x_{1}^{\pm 1},\ldots ,x_{n}^{\pm 
1}]^{S_{n}}$$
where $S_{n}$ is the symmetric group on $\{1,\ldots,n\}$ permuting the 
variables $x_{1},\ldots ,x_{n}$. If we let $T_{i,v}$ denote the $i$-th 
elementary 
symmetric polynomial in $x_{1},\ldots ,x_{n}$, then $q_{v}^{i(n+1)/2}T_{i,v}\in 
\T_{v}$ and 
we define
\begin{equation}\label{Heckepoly} P_{v}(X)=1 - q_{v}^{(n+1)/2}T_{1,v}X + 
q_{v}^{n+1}T_{2,v}X^{2} - \cdots + (-1)^{n}q_{v}^{n(n+1)/2}T_{n,v}X^{n} \in 
\T_{v}[X].\end{equation}
The following theorem is essentially a special case of \cite[Theorem 5.4.1]{sch}. We let $G_{F,S}$ denote the Galois group of the maximal algebraic extension of $F$ unramified outside $S$. For the notion of a determinant we refer to \cite{che}.

\begin{theorem}\label{scholze}
(Scholze) There exists an integer $M$ depending only on $[F:\Q]$ and $n$ such 
that the following is true: for any $j$ and $r$ there exists an ideal 
$I_{\ka,j}^r\sub \T_{\ka, j}^{r}$ with $(I_{\ka,j}^r)^{M}=0$ and an 
$n$-dimensional 
continuous determinant $D$ of $G_{F,S}$ with values in 
$\T_{\ka,j}^{r}/I_{\ka,j}^r$ such that
$$ D(1 - X\, Frob_{v})=P_{v}(X) $$
for all $v\notin S$.
\end{theorem}

We remark that the local systems corresponding to the $\D_{\ka}^{<r,\circ}/\Fil^{j}$ are not necessarily included in the formulation of \cite[Theorem 5.4.1]{sch}, but the proof works the same: one first applies the Hochschild--Serre spectral sequence to reduce to the case of a trivial local system. 

\begin{corollary}\label{galois}
Keep the notation of Theorem \ref{scholze}.
\begin{enumerate}
\item There exists a closed ideal $I\sub \wh{\T}_{\ka,\Fil}^{r}$ such that $I^{M}=0$ and an $n$-dimensional continuous determinant $D$ of $G_{F,S}$ with values in $\wh{\T}_{\ka, \Fil}^{r}/I$ such that $D(1 - X\, Frob_{v})=P_{v}(X)$ for all $v\notin S$.

\medskip

\item Let $h\in \Q_{\geq 0}$. Then there exists a closed ideal $J\sub \T_{\ka,\leq h}^{r,\circ}$ such that $J^{M}=0$ and an $n$-dimensional continuous determinant $D$ of $G_{F,S}$ with values in $\T_{\ka, \leq h}^{r, \circ}/J$ such that $D(1 - X\, Frob_{v})=P_{v}(X)$ for all $v\notin S$.

\end{enumerate}
\end{corollary}

\begin{proof}
Part (1) follows from the theorem and the definition and compactness of 
$\wh{\T}_{\ka,\Fil}^{r}$ via \cite[Example 2.32]{che} (one sets 
$I=\wh{\T}_{\ka,\Fil}^{r} \cap \prod_{j} I_{\ka,j}^{r}$). Assertion (2) then 
follows from (1) and Proposition \ref{cty}.
\end{proof}

\subsection{Gluing over the reduced eigenvariety}

We now finish our construction of a Galois determinant over the reduced eigenvariety $\ms{X}^{red}=\ms{X}^{red}_{\mbf{G}}$. See Definition \ref{defred} for the definition of the reduced subspace. Keep the notation of the previous subsection. Let $(\mc{U},h)$ be a slope datum. We have a corresponding open affinoid $\ms{X}_{\U,h}\sub \ms{X}$. If $\mf{m}_{\ka}$ is a maximal ideal of $\oo(\U)$ corresponding to the weight $\ka : T_{0} \ra L^{\times}$ with $L=\oo(\U)/\mf{m}_{\ka}$ a local field, then Corollary \ref{points2} gives us a natural identification
$$ \left(\oo(\ms{X}_{\U,h})/\mf{m}_{\ka}\right)^{red} = \left(\oo(\ms{X}^{red}_{\U,h})/\mf{m}_{\ka}\right)^{red} \cong \T(\De^{p},K^{p})[U_{t}]^{red}_{\ka,h}. $$
Fix $r$. Note that $\T_{\ka, \leq h}^{r}=\T_{\ka,\leq h}^{r,\circ}[1/\vp]$ is naturally a closed $L$-subalgebra of $\T(\De^{p},K^{p})[U_{t}]_{\ka,h}$. From Corollary \ref{galois} it follows that we have a Galois determinant into $\T(\De^{p},K^{p})[U_{t}]_{\ka,h}^{red}$ and therefore into $\left(\oo(\ms{X}^{red}_{\U,h})/\mf{m}_{\ka}\right)^{red}$. We record this discussion in the following convenient form:

\begin{lemma}\label{pointwise}
Let $(\U,h)$ be a slope datum and let $\mf{m}$ be a maximal ideal of $\oo(\ms{X}_{\U,h}^{red})$. Then there exists an $n$-dimensional continuous determinant $D$ of $G_{F,S}$ with values in $\oo(\ms{X}_{\U,h}^{red})/\mf{m}$ such that $D(1 - X\, Frob_{v})=P_{v}(X)$ for all $v\notin S$.
\end{lemma}

\begin{proof}
This follows from the discussion above since the map $\Spec \oo(\ms{X}_{\U,h}^{red}) \ra \Spec \oo(\U)$ sends maximal ideals to maximal ideals (the map is finite).
\end{proof}

Before we can glue we need some preparations. Let $A$ be a reduced 
$\Zp$-algebra which is finite and free as a $\Zp$-module; we equip it with the 
$p$-adic topology. Consider the adic space $Z^{an}=\mb{A}^{1}_{S^{an}}$, where 
$S=\Spa( A[[X_{1},\ldots,X_{d}]], A[[X_{1},\ldots,X_{d}]])$; here 
$A[[X_{1},\ldots,X_{d}]]$ carries the $(p,X_{1},\ldots,X_{d})$-adic topology. 
Fix an index $i\in \{1,\ldots,d\}$. Let $T$ be a coordinate on 
$\mb{A}^{1}_{S^{an}}$. We are interested in the open affinoid subsets $V_{m}=\{ 
|p^{m}|, |X_{1}^{m}|,\ldots,|X_{d}^{m}| \leq |X_{i}|\neq 0, \, |X_{i}^{m}T|\leq 
1 \}$ of $Z^{an}$ for $m\in \Z_{\geq 1}$. Note that the union of the $V_{m}$ is 
the locus $V=\{ |X_{i}|\neq 0 \}\sub Z^{an}$. The ring $\oo(V_{m})$ has a ring 
of definition 
$$R_{m}=A[[X_{1},\ldots,X_{d}]]\left\langle 
\frac{p^{m}}{X_{i}},\frac{X_{1}^{m}}{X_{i}},\ldots,\frac{X_{d}^{m}}{X_{i}} 
\right\rangle \langle X_{i}^{m}T \rangle $$ (with the $X_i$-adic topology),
and we have natural maps $R_{m+1} \ra R_{m}$ for all $m$. Note that $\oo(V_{m})=R_{m}[1/X_{i}]$ and that $\oo^{+}(V_{m})$ is the integral closure of $R_{m}$ in $\oo(V_{m})$. By Theorem \ref{su1} (and its proof), in fact we have $\oo^{+}(V_{m})=\oo(V_m)^\circ$ and $\oo(V_m)^\circ$ is a finite $R_{m}$-algebra.

\begin{lemma}\label{profinite}
The image of $R_{m+1}$ in $R_{m}/X_{i}^{n}$ is finite for all $m$ and $n$. More generally, if $M$ is a finitely generated $R_{m+1}$-module, then the image of $M$ in $M\otimes_{R_{m+1}}R_{m}/X_{i}^{n}$ is finite.
\end{lemma}

\begin{proof}
Fix $m$ and $n$. We start with the first assertion. It suffices to check that the kernel contains the ideal 
$$I=\left( 
p^{mn},X_{1}^{mn},\ldots,X_{d}^{mn},\left(\frac{p^{m+1}}{X_{i}}\right)^{mn},\left(
 \frac{X^{m+1}_{1}}{X_{i}}\right)^{mn},\ldots,\left( 
\frac{X_{d}^{m+1}}{X_{i}}\right) ^{mn}, (X_{i}^{m+1}T)^{n}\right) $$
since it's straightforward to check that $R_{m+1}/I$ is finite. It's also straightforward to check that the given generators of $I$ are in the kernel. This finishes the (sketch of) proof of the first assertion. For the second, there is an integer $q\geq 0$ and a surjection $R_{m+1}^{q} \twoheadrightarrow M$ of $R_{m+1}$-modules, which gives us a commuting diagram
$$
\xymatrix{ R_{m+1}^{q} \ar[r]\ar[d] & R_{m}^{q}/X_{i}^{n} \ar[d]\\
M \ar[r] & M\otimes_{R_{m+1}}R_{m}/X_{i}^{n}}
$$
where the vertical maps are surjections, and the second assertion then follows from the first.
\end{proof}

\begin{proposition}\label{compact}
With notation as above, let $Y \ra V = \{ |X_{i}|\neq 0 \} \subset Z^{an}$ be a finite morphism of adic spaces, and assume that $Y$ is reduced. Then $\oo^{+}(Y)$ is compact.
\end{proposition}

\begin{proof}
Write $Y_{m}$ for the pullback of $V_{m}$; these form an increasing cover of $Y$ consisting of open affinoids. Since $Y_{m} \ra V_{m}$ is finite and $\oo^{+}(V_m)=\oo(V_m)^{\circ}$ we know that $\oo^{+}(Y_{m})$ is integral over $\oo(V_{m})^{\circ}$ (by definition of a finite morphism). Since $\oo(Y_{m})$ is reduced and $\oo(V_{m})^{\circ}$ is Nagata (it is finite over $R_m$, which is Nagata by the proof of Theorem \ref{su1}) it follows that $\oo^{+}(Y_{m})$ is finite over $\oo(V_{m})^{\circ}$. In particular, $\oo^{+}(Y_m)$ is finite over $R_m$, hence Noetherian. We may also deduce from Lemma \ref{noeth} that  $\oo^{+}(Y_m)=\oo(Y_m)^{\circ}$.

For any $n$ the map $\oo^{+}(Y_{m+1}) \ra \oo^{+}(Y_{m})/X_{i}^{n}$ factors through $\oo^{+}(Y_{m+1})\otimes_{R_{m+1}}R_{m}/X_{i}^{n}$ and hence the image is finite by Lemma \ref{profinite}. It follows that $\oo^{+}(Y_{m+1})$, when equipped with the weak topology with respect to the map $\oo^{+}(Y_{m+1}) \ra \oo^{+}(Y_{m})$, is compact. We deduce that $\oo^{+}(Y) = \varprojlim_{m}\oo^{+}(Y_{m})$ is compact, as desired.
\end{proof}

\begin{corollary}\label{compact2}
Assume that $X \ra Z^{an}$ is a finite morphism of adic spaces, with $X$ reduced. Then $\oo^{+}(X)$ is compact.
\end{corollary}

\begin{proof}
For each $i\in \{1,\ldots ,d\}$ consider the locus $Z_{i}=\{|X_{i}|\neq 0 
\}\sub 
Z^{an}$ (i.e. what was previously denoted by $V$) and set $Z_{0}=\{ |p|\neq 
0\}$. Let $X_{i}$, for $i\in \{0,\ldots,d\}$, denote the corresponding 
pullbacks to $X$. Then we have a strict inclusion $\oo^{+}(X) \sub 
\prod_{i=0}^{d}\oo^{+}(X_{i})$ with closed image and the $\oo^{+}(X_{i})$ are 
compact (for $i=1,\ldots,d$ this is Proposition \ref{compact} and for $i=0$ 
this is \cite[Lemma 7.2.11]{bc}; note that $X_{0}$ is nested in the terminology 
of \cite{bc}), so $\oo^{+}(X)$ is compact as well.
\end{proof}

We may then prove the main result of this section.

\begin{theorem}\label{gluedgalois}
There exists an $n$-dimensional continuous determinant $D$ of $G_{F,S}$ with values in $\oo^{+}(\ms{X}^{red})$ such that
$$ D(1 - X\, Frob_{v})=P_{v}(X) $$
for all $v\notin S$.
\end{theorem}

\begin{proof}
Fix a collection $\{ (\U,h) \}$ of slope data such that the $\ms{X}_{\U,h}^{red}$ cover $\ms{X}^{red}$. We then have injections
$$ \oo^{+}(\ms{X}^{red}) \hookrightarrow \prod_{(\U,h)}\oo(\ms{X}^{red}_{\U,h}) \hookrightarrow \prod_{(\U,h)} \prod_{\mf{m}} \oo(\ms{X}^{red}_{\U,h})/\mf{m}$$
where the $\mf{m}$ range over all maximal ideals of $\oo(\ms{X}^{red}_{\U,h})$. 
Note that we have an injection $\oo(\ms{X}^{red}_{\U,h}) \hookrightarrow 
\prod_{\mf{m}} \oo(\ms{X}^{red}_{\U,h})/\mf{m}$ by Lemma \ref{jacobson}, so the 
second morphism really is an injection. Then note that $\oo^{+}(\ms{X}^{red})$ 
is compact; since $\ms{X}^{red}$ is finite over 
$\Spa(\Zp[[T_{0}]],\Zp[[T_{0}]])^{an}\times \mb{A}^{1}$ this follows from 
Corollary \ref{compact2}. To see that $\Zp[[T_{0}]]$ is of the form 
$A[[X_{1},\ldots ,X_{d}]]$, we write $T_{0}\cong T_{0}^{tor}\times 
T_{0}^{free}$ 
where $T_{0}^{tor}$ is the torsion subgroup of $T_{0}$ (which is finitely 
generated) and $T_{0}^{free} \cong \Zp^{\dim T_{0}}$ is a free complement (cf. 
e.g. \cite[Proposition II.5.7]{neu}). Set $A=\Zp[T_{0}^{tor}]$, then 
$\Zp[[T_{0}]] \cong A \otimes_{\Zp} \Zp[[T_{0}^{free}]] \cong 
A[[X_{1},\ldots,X_{d}]]$. Thus we may glue the determinants from Lemma 
\ref{pointwise} into the desired determinant using \cite[Example 2.32]{che}.
\end{proof}

\section{The Coleman-Mazur eigencurve}\label{gl2}

In this section we give a short discussion of the special case of the Coleman-Mazur eigencurve and the relationship between our work and that of Andreatta--Iovita--Pilloni \cite{aip} and Liu--Wan--Xiao \cite{lwx}.

\subsection{The case $\mbf{G}=\GL_{2/\Q}$}\label{cm} Let us consider the special case $\mbf{G}=\GL_{2/\Q}$. We begin by fixing choices of groups and Hecke algebras/operators. Let $\mbf{B}$ be the upper triangular Borel, $I$ the corresponding Iwahori and $\mbf{T}$ the diagonal torus. We use the element $t=\left( \begin{smallmatrix} 1 & 0 \\ 0 & p \end{smallmatrix} \right)\in \Sigma^{cpt}$; the corresponding Hecke operator is the $U_{p}$-operator. We choose the tame level $K_{1}(N)=\{ g\in \GL_{2}(\wh{\Z}^{p}) \mid g\equiv \left(\begin{smallmatrix} \ast & \ast \\ 0 & 1 \end{smallmatrix} \right)\, {\rm mod}\, N \}$, with $N\in \Z_{\geq 5}$ prime to $p$. Put $\De_{\ell}=\GL_{2}(\Q_{\ell})$ if $\ell \nmid N$ and $\De_{\ell}=\GL_{2}(\Z_{\ell})$ if $\ell \mid N$. With these choices everything else is determined and we use the notation of the main part of the paper. 

\medskip

In this case, overconvergent modular symbols were first constructed by Stevens \cite{ste}, and the corresponding eigencurve was constructed and shown to agree with the Coleman-Mazur eigencurve by Bellaiche \cite{bel}. Stevens and Bellaiche worked with the compactly supported cohomology $H^{1}_c$, which admits a very explicit description in this case (it is given by the functor denoted by ${\rm Symb}$ in \cite{bel,ste}). Ordinary cohomology was first considered in \cite{han}. For ordinary cohomology, one has
$$ H^{\ast}(K, \ms{D}_{\ka})=H^{1}(K, \ms{D}_{\ka}) $$
for all weights $\ka$, upon noting that the $H^{i}$ vanish automatically for all $i\geq 2$ and that $H^{0}$ vanishes by (a simpler version of) the argument in the proof of \cite[Lemma 5.4]{chj}. A consequence is the following lemma:

\begin{lemma}\label{proj}
Let $\ka : T_{0} \ra R^{\times}$ be a weight, and assume that $F_{\ka}$ has a slope-$\leq h$ factorization for some $h\in \Q_{\geq 0}$ (slopes with respect to some multiplicative pseudo-uniformizer $\vp\in R$). Then $H^{1}(K, \ms{D}_{\ka})_{\leq h}$ is a finite projective $R$-module and is compatible with arbitrary base change.
\end{lemma}

\begin{proof}
Let $f : R \ra S$ be a continuous homomorphism of complete Tate rings and equip $S$ with a Banach--Tate $\Zp$-algebra norm such that $f(\vp)$ is a multiplicative pseudo-uniformizer. Put $\ka_{S}=f \circ \ka$. By the vanishing of the $H^{i}$ for $i\neq 1$ the Tor-spectral sequence collapses and gives us that
$$H^{1}(K, \ms{D}_{\ka})_{\leq h}\otimes_{R}S=H^{1}(K, \ms{D}_{\ka_{S}})_{\leq h};$$
$$\Tor_{1}^{R}(H^{1}(K, \ms{D}_{\ka})_{\leq h}, S)=0.$$
The first line is compatibility with base change. Putting $S=R/J$ for an arbitrary ideal $J\sub R$ (automatically closed) we see from the second line that $H^{1}(K, \ms{D}_{\ka})_{\leq h}$ is a finite flat $R$-module, hence finite projective.
\end{proof}

If $\ka : T_{0} \ra R^{\times}$ is a weight we write $\ka_{i}$, $i=1,2$, for the characters $\Zp^{\times} \ra R^{\times}$ defined by 
$$ \ka\left( \begin{pmatrix} a & 0 \\ 0 & d \end{pmatrix} \right) = \ka_{1}(a)\ka_{2}(d)$$
and we identify $\ol{N}_{1}$ with $p\Zp$ via $\left( \begin{smallmatrix} 1 & 0 \\ x & 1 \end{smallmatrix} \right) \mapsto x$. If we consider the eigenvariety $\ms{X}^{rig}=\ms{X}^{rig}_{\mbf{G}}$ constructed in \cite{han}, then it is equidimensional of dimension $2$ by \cite[Proposition B.1]{han} since $H^{\ast}=H^{1}$. This object is usually referred to as the `eigensurface'. If we instead do the eigenvariety construction over the part $\mc{W}_{0}^{rig}$ of weight space where $\ka_{2}=1$ we obtain an eigenvariety that turns out to equal the Coleman-Mazur eigencurve; it is in particular reduced, equidimensional of dimension $1$, and flat over $\mc{W}_{0}^{rig}$. Let us denote this eigenvariety by $\ms{E}^{rig}$; the properties of $\ms{E}^{rig}$ stated in the previous sentence are presumably well known to experts but we will give a brief sketch of the proofs below. To begin with, it is equidimensional of dimension $1$ (by the same argument as above for $\ms{X}^{rig}$). For weights with $\ka_{2}=1$ we conflate $\ka$ and $\ka_{1}$, and we may write the action on $\ms{A}_{\ka}$ explicitly as
$$ (f. \gamma)(x) = \ka(a+bx)f\left( \frac{c+dx}{a+bx} \right) $$
for $\gamma = \left( \begin{smallmatrix} a & b \\ c & d \end{smallmatrix} \right)\in M_{2}(\Zp)$ such that $a\in \Zp^{\times}$, $c\in p\Zp$ (this defines the submonoid $\De_{p,0}$ of $\De_{p}$ generated by $I$ and $t$), and $x\in p\Zp$ (cf. \cite[\S 2.2]{han}). Using the anti-involution 
$$ \begin{pmatrix} a & b \\ c & d \end{pmatrix} \mapsto \begin{pmatrix} a & c/p \\ pb & d \end{pmatrix} $$
on $\De_{p,0}$ and rescaling $f$ by $p^{-1}$ in the argument to a function on $\Zp$ one sees this right action corresponds to the left action defined in \cite{ste}. To prove that $\ms{E}^{rig}$ is reduced one uses a criterion of Chenevier \cite[Proposition 3.9]{che2}; see \cite[Theorem IV.2.1]{bel} for the analogous statement for the $H^{1}_{c}$. Unfortunately, the notion of a `classical structure' in \cite{che2} is based on the input data for the eigenvariety construction in \cite{bu} and is therefore not directly applicable to the situation in \cite{han}. Let us state a version of \cite[Proposition 3.9]{che2} applicable to the situation in \cite[Definition 4.2.1]{han}. We use the notation and terminology of \cite[\S 4-5]{han} freely; in particular we use the language of classical rigid geometry for this Proposition.

\begin{proposition}
Let $\mf{O}=(\ms{W},\ms{Z},\ms{M},\mbf{T},\psi)$ be an eigenvariety datum. If $(U,h)$ is a slope datum, assume that $\ms{M}(\ms{Z}_{U,h})$ is a projective $\oo(U)$-module. Assume moreover that there exists a very Zariski dense set $\ms{W}^{cl}\sub \ms{W}$ such that, if $(U,h)$ is a slope datum, there is a Zariski open and dense subset $W_{U,h}$ of $\ms{W}^{cl}\cap U$ such that $\ms{M}(\ms{Z}_{U,h})_{x}$ is a semisimple $\mbf{T}[T^{-1}]$-module. Here $T$ is the parameter on $\A^{1}_{\ms{W}}$, which naturally acts invertibly on $\ms{M}(\ms{Z}_{U,h})$, and the set $\ms{W}^{cl}$ is given the Zariski topology. Then the eigenvariety $\ms{X}_{\mf{O}}$ is reduced.
\end{proposition}

The proof is virtually identical to that of \cite[Proposition 3.9]{che2}; we omit it. Using this, one proves that $\ms{E}^{rig}$ is reduced in the same way as in the proof of \cite[Theorem IV.2.1(i)]{bel}, using the control theorem of Stevens (cf. \cite[Theorem 3.2.5]{han}); recall that our Hecke algebra $\T(\De^{p},K^{p})$ contains no Hecke operators at primes dividing $N$. That $\ms{E}^{rig}$ is equal to the Coleman-Mazur eigencurve is then proved using the control theorems of Coleman and Stevens and the Eichler-Shimura isomorphism together with \cite[Theorem 5.1.2]{han} (this type of argument is well known to experts, see for example the proof of \cite[Theorem IV.2.1(i)]{bel}). The argument for flatness will be given below. This finishes our review of the basic properties of $\ms{E}^{rig}$.

\medskip

Let us now return to the constructions of this paper. Our eigenvariety construction gives an extension $\ms{E}$ of $\ms{E}^{rig}$ defined over the locus $\mc{W}_{0}\sub \mc{W}$ where $\ka_{2}=1$. Another such extension $\ms{E}^{\prime}$ was constructed by Andreatta, Iovita and Pilloni \cite{aip}. Note that $\mc{W}_{0}$ is naturally the analytic locus of the formal weight space $\mf{W}_{0}$ with $\ka_{2}=1$. We have $\mf{W}_{0}\cong \Spa(\Zp[[\Zp^{\times}]], \Zp[[\Zp^{\times}]])$. When $p\neq 2$, $\Zp[[\Zp^{\times}]]$ is a regular ring. When $p=2$ this is no longer the case; we have $\Z_{2}[[\Z_{2}^{\times}]]\cong \Z_{2}[\Z/2][[X]]$ and $\Z_{2}[\Z/2]$ is not regular. We will instead work over the normalization $A$ of $\Z_{2}[[\Z_{2}^{\times}]]$, which is isomorphic to $\Z_{2}[[X]]\times \Z_{2}[[X]]$. The normalization map
$$ \Z_{2}[\Z/2][[X]] \ra \Z_{2}[[X]] \times \Z_{2}[[X]] $$
is explicitly given by
$$ \sum_{n\geq 0} (a_{n}+b_{n}g)X^{n} \mapsto \left( \sum_{n\geq 0} (a_{n}+b_{n})X^{n}, \sum_{n\geq 0} (a_{n}-b_{n})X^{n} \right) $$
where $a_{n},b_{n}\in \Z_{2}$ and $g\in \Z/2$ is the non-trivial element. Using this map we get a weight into $A$, and we let $\wt{\mf{W}}_{0} = \Spa(A,A)$ and put $\wt{\mc{W}}_{0}=\wt{\mf{W}}_{0}^{an}$. We remark that $\wt{\mc{W}}_{0}^{rig} \cong \mc{W}_{0}^{rig}$ canonically via the normalization map. To make our notation uniform, we set $\wt{\mc{W}}_{0}=\mc{W}_{0}$ when $p\neq 2$. We may perform our construction over $\wt{\mc{W}}_{0}$, and one may pull back the Banach modules that are used to construct the eigencurve in \cite{aip} to $\wt{\mc{W}}_{0}$, and construct the eigencurve over $\wt{\mc{W}}_{0}$ instead. Let us denote the corresponding eigenvarieties by $\wt{\ms{E}}$ and $\wt{\ms{E}}^{\prime}$, though we hasten to remark that these should \emph{not} be thought of as normalizations of $\ms{E}$ and $\ms{E}^{\prime}$.

\begin{lemma}
Let $\U \subset \wt{\mc{W}}_{0}$ be a connected open affinoid subset such that $\oo(U)$ is a Tate ring. Then $\oo(\U)$ is a Dedekind domain.
\end{lemma}

\begin{proof}
We first show that $\oo(\U)$ is regular of Krull dimension $1$. The connected components of $\wt{\mc{W}}_{0}$ are isomorphic to $\Spa(\Zp[[X]],\Zp[[X]])$ so we may take $\U$ to be an open affinoid subset of $\Spa(\Zp[[X]],\Zp[[X]])$. Let $\mf{q}$ be a maximal ideal of $\oo(\U)$ and let $\mf{p}$ be its preimage in $A=\Zp[[X]]$. By Proposition \ref{localrings} the natural map $A_{\mf{p}} \ra \oo(\U)_{\mf{q}}$ induces an isomorphism on completions. By Lemma \ref{maxideals} $\mf{p}$ defines a closed point of $\Spec A \setminus \{(p,X)\}$. Since $A$ is a regular local ring of dimension $2$, it follows that $A_{\mf{p}}$ is a regular local ring of dimension $1$, and hence the same is true for $\oo(\U)_{\mf{q}}$ (since if $R$ is a Noetherian local ring with completion $\wh{R}$, $R$ is regular if and only if $\wh{R}$ is regular, and $\dim R=\dim \wh{R}$). It follows that $\oo(\U)$ is a product of regular integral domains of dimension $1$. Since $\U$ is connected $\oo(\U)$ does not contain any nontrivial idempotents, so $\oo(\U)$ is an integral domain. This finishes the proof.
\end{proof}

Let $(\U,h)$ be a connected slope datum for $\wt{\ms{E}}$ (by which we mean a slope datum for the construction that produces $\wt{\ms{E}}$ such that $\U$ is connected; we will use the terminology `(connected) slope datum for $\wt{\ms{E}}^{\prime}$' similarly). Then $\oo(\wt{\ms{E}}_{\U,h})$ is, by definition, an $\oo(\U)$-submodule of $\End_{\oo(\U)}(H^{1}(K, \ms{D}_{\U})_{\leq h})$. The latter is projective by Lemma  \ref{proj}, so the former is also projective since $\oo(\U)$ is Dedekind. Thus the natural map $\wt{\ms{E}}_{\U,h} \ra \U$ is finite flat, and hence $\wt{\ms{E}} \ra \wt{\mc{W}}_{0}$ is flat. The same applies to $\wt{\ms{E}}^{\prime}$. 

\medskip

Now let $(\U,h)$ be a slope datum for $\wt{\ms{E}}$ \emph{and} $\wt{\ms{E}}^{\prime}$. Let $(\U_{i})_{i\in I}$ be an open affinoid cover of $\U^{rig}$. Then the natural map $\oo(\U) \ra \prod_{i\in I}\oo(\U_{i})$ is an injection (since $\U \setminus \U^{rig}$ does not contain any open subset of $\U$) so tensoring with the finite projective $\oo(\U)$-module $\oo(\wt{\ms{E}}_{\U,h})$ we get an injection
$$ \oo(\wt{\ms{E}}_{\U,h}) \hookrightarrow \left( \prod_{i\in I} \oo(\U_{i}) \right) \otimes_{\oo(\U)}\oo(\wt{\ms{E}}_{\U,h}) \cong \prod_{i\in I} \oo(\wt{\ms{E}}_{\U_{i},h}), $$
which in particular shows that $\oo(\wt{\ms{E}}_{\U,h})$ is reduced. The image of $\oo(\wt{\ms{E}}_{\U,h})$ inside $\prod_{i \in I} \oo(\wt{\ms{E}}_{\U_{i},h})$ is equal to the $\oo(\U)$-span of the image of $\T(\De^{p},K^{p})[U_{p}]$ in $\prod_{i \in I} \oo(\wt{\ms{E}}_{\U_{i},h})$. The same holds replacing $\wt{\ms{E}}$ by $\wt{\ms{E}}^{\prime}$, so since we have canonical isomorphisms
$$ \oo(\wt{\ms{E}}_{\U_{i},h}) \cong \oo(\wt{\ms{E}}^{\prime}_{\U_{i},h}) $$
we obtain a canonical isomorphism $\oo(\wt{\ms{E}}_{\U,h}) \cong \oo(\wt{\ms{E}}^{\prime}_{\U,h})$, compatible with the way the eigencurves are built. As a result we have a canonical isomorphism $\wt{\ms{E}} \cong \wt{\ms{E}}^{\prime}$ extending the canonical isomorphism $\wt{\ms{E}}^{rig}\cong (\wt{\ms{E}}^{\prime})^{rig}$. We summarize the discussion above in the following theorem:

\begin{theorem}\label{comparison}
The eigenvariety $\wt{\ms{E}}$ is reduced and flat over $\wt{\mc{W}}_{0}$. Moreover, it is canonically isomorphic to the eigencurve $\wt{\ms{E}}^{\prime}$ constructed by Andreatta--Iovita--Pilloni \cite{aip}.
\end{theorem}
\begin{remark}
In fact, it is possible to show that $\ms{E}$ and $\ms{E}^{\prime}$ are 
isomorphic for all $p$ (i.e~including $p=2$), using the interpolation theorem 
\cite[Theorem 3.2.1]{jn2}, since both $\ms{E}$ and $\ms{E}^{\prime}$ are 
reduced 
with Zariski dense sets of classical points which naturally match up.
\end{remark}

Fix a character $\eta : (\Z/q)^{\times} \ra \Fp^{\times}$ (recall that $q=4$ if $p= 2$ and $q=p$ otherwise). We have a natural isomorphism $\Zp^{\times}\cong (\Z/q)^{\times} \times (1+q\Zp)$ defined by $z \mapsto (\ol{z}, z/\omega(\ol{z}))$ where $\ol{-}$ denotes reduction modulo $q$ and $\omega$ denotes the Teichm\"uller lift. Let us write $\langle z \rangle :=z/\omega(\ol{z})$. Then we may define a character $\ka_{\eta} : \Zp^{\times} \ra \Zp[[X]]^{\times}$ by
$$ \ka_{\eta}(z) = \omega(\eta(\ol{z})) \sum_{\n=0}^{\infty} \begin{pmatrix} p^{-1}\log \langle z \rangle \\ n \end{pmatrix} X^{n}. $$
We let $\ol{\ka}_{\eta}$ denote its reduction modulo $p$. This is a character $\Zp^{\times} \ra \Fp((X))^{\times}$ which we may think of as a character $T_{0} \ra \Fp((X))^{\times}$ with $\ka_{2}=1$. We remark that if $p=2$ then $\eta$, and hence $\ol{\ka}_{\eta}$, is unique.

\begin{corollary}\label{inf}
There are infinitely many (non-ordinary) finite slope $U_{p}$-eigenvectors in $H^{1}(K, \ms{D}_{\ol{\ka}_{\eta}})$.
\end{corollary}

\begin{proof}
By Corollary \ref{points2}, its analogue for $\wt{\ms{E}}^{\prime}$ (which is 
simpler, and is essentially \cite[Lemma 5.9]{bu}) and Theorem \ref{comparison} 
we see that $H^{1}(K,\ms{D}_{\ol{\ka}_{\eta}})$ and the module 
$\ol{M}^{\dagger}_{\ol{\ka}_{\eta}}(N)$ of overconvergent modular forms of 
weight $\ol{\ka}_{\eta}$ and tame level $N$ constructed in \cite{aip} contain 
the same finite slope systems of Hecke eigenvalues. By \cite[Corollary 
A.1]{bp}, $\ol{M}^{\dagger}_{\ol{\ka}_{\eta}}(N)$ has infinitely many finite 
slope $U_{p}$-eigenvectors, so we are done.
\end{proof}

\begin{remark}
It is possible to prove Corollary \ref{inf} directly from \cite[Theorem A]{bp} (using the observation in the Remark following \cite[Corollary A.2]{bp}) without any reference to \cite{aip}.
\end{remark}

\subsection{Estimates for the Newton polygon of $U_p$}
In this and the following section we give a short proof of the estimates 
obtained in \cite[Theorem 3.16]{lwx} for the Newton polygon of $U_p$ acting on 
spaces of overconvergent automorphic forms for a definite quaternion algebra 
over $\Q$.

We fix an odd prime $p$ and assume that we are in the setting of Section \ref{modules}. Suppose that $\overline{N}_1 \cong \Z_p$ and fix a topological generator $\nbar$. Let $f$ be a norm-decreasing $R$-linear map \[f: \bigoplus_{i=1}^t \mc{D}^r_{\ka} \rightarrow \bigoplus_{i=1}^t \mc{D}^{r^{1/p}}_{\ka}\] and recall that we have a compact inclusion \[\iota: \bigoplus_{i=1}^t \mc{D}^{r^{1/p}}_{\ka}\hookrightarrow \bigoplus_{i=1}^t \mc{D}^r_{\ka}.\] We set $U = \iota\circ f$. Then $U$ is a compact endomorphism of $M = \bigoplus_{i=1}^t \mc{D}^r_{\ka}$. Recall that we have a potential ON-basis for $\mc{D}^r_{\ka}$ given by the elements $e_{r,\alpha}:=\varpi^{-n(r,\varpi,\alpha)}\nbar^\alpha$ for $\alpha\in\Z_{\ge 0}$. 

We consider the potential ON-basis for $M$ given by \[e_{r,0}^1=(e_{r,0},0,\ldots,0),\ldots,e_{r,0}^t=(0,\ldots,0,e_{r,0}), e_{r,1}^1=(e_{r,1},0,\ldots,0),\ldots\]
\begin{lemma}\label{Upestimate}
Assume that there is no $x\in R$ with $1<|x|<|\vp|^{-1}$. Then $U$ maps $e_{r,\alpha}^i$ to a sum $\sum_{j,\beta} a_\beta^j e_{r,\beta}^j$ with \[|a_\beta^j|\le |\varpi|^{n(r,\varpi,\beta)-n(r^{1/p},\varpi,\beta)}.\] If we define \[\lambda(0) = 0, \lambda(i+1) = \lambda(i) + n(r,\varpi,\lfloor i/t\rfloor)-n(r^{1/p},\varpi,\lfloor i/t\rfloor)\] the Fredholm series \[\det(1-TU|M) = \sum_{n\ge 0}c_nT^n \in R\{\{T\}\}\] satisfies $|c_n| \le |\varpi|^{\lambda(n)}$.
\end{lemma}
\begin{proof}
We first prove the estimate on the matrix coefficients of $U$. Apply $f$ to $e_{r,\alpha}^i$. We get a sum $\sum_{j,\beta} b_\beta^j e_{r^{1/p},\beta}^j$, and the fact that $f$ is norm-decreasing is equivalent to
$$ |b_{\beta}^{j}| \leq |\vp|^{n(r^{1/p},\vp,\beta)-n(r,\vp,\al)}r^{|\al|-|\beta|/p} $$
for all $j,\beta$. Since $|\vp|<|\vp|^{-n(r,\vp,\al)}r^{|\al|},\,|\vp|^{-n(r^{1/p},\vp,\beta)}r^{|\beta|/p} \leq 1$ by construction, we deduce that $|b_{\beta}^{j}|<|\vp|^{-1}$ for all $j,\beta$. By our assumption on $R$, it follows that $|b_{\beta}^{j}|\leq 1$ for all $j,\beta$. We then have \[e_{r^{1/p},\beta}^j = \varpi^{-n(r^{1/p},\varpi,\beta)}\nbar^\beta = \varpi^{n(r,\varpi,\beta)-n(r^{1/p},\varpi,\beta)}e_{r,\beta}^j\] so we conclude that \[U e_{r,\alpha}^i= \sum_{j,\beta} a_\beta^j e_{r,\beta}^j\] where $a_\beta^j = \varpi^{n(r,\varpi,\beta)-n(r^{1/p},\varpi,\beta)}b_\beta^j$, and therefore \[|a_\beta^j| \le |\varpi|^{n(r,\varpi,\beta)-n(r^{1/p},\varpi,\beta)}.\]

In other words, the $i$th row of the matrix for $U$ (we begin indexing rows at $i = 0$) has entries with norm $\le |\varpi|^{n(r,\varpi,\lfloor i/t\rfloor)-n(r^{1/p},\varpi,\lfloor i/t\rfloor)}$. We deduce immediately that $|c_n| \le |\varpi|^{\lambda(n)}$, since $c_n$ is an alternating sum of products of matrix entries coming from $n$ distinct rows (\cite[Proposition 7]{ser}), and each of these products has norm $\le |\varpi|^{\lambda(n)}$.
\end{proof}

\subsection{Definite quaternion algebras over $\Q$}
As an application of Lemma \ref{Upestimate}, we give a new proof of 
\cite[Theorem 3.16]{lwx}. In this section we assume that $p$ is odd. We need to 
set up things so that we can apply the machinery of sections \ref{modules}, 
\ref{evars}. Let $D/\Q$ be a definite quaternion algebra, split at $p$, and let 
$\mathbf{G}$ be the reductive group over $\Q$ defined by $\mathbf{G}(R)= 
(D\otimes_{\Q}R)^\times$, for $\Q$-algebras $R$. We fix an isomorphism $D_p 
\cong M_2(\Q_p)$ and henceforth identify $D_p$ with $M_2(\Q_p)$ via this 
isomorphism. Let $\mathbf{G}_{\Z_p}$ be the $\Z_p$-model for $\mathbf{G}$ given 
by $\mathbf{G}_{\Z_p}(R)= (M_2(\Z_p)\otimes_{\Z_p}R)^\times$, for 
$\Z_p$-algebras $R$. We let $\mathbf{B}$ be the upper triangular Borel in 
$\mathbf{G}_{\Z_p}$ and let $\mathbf{T}$ be the diagonal maximal torus. 

Fix a tame level $K^p = \prod_{l \ne p}K_l$ with $K_l \subset \mathbf{G}(\Q_l)$ compact open, let $K = K^pI$, and assume that $K$ is neat.

Fix a character $\eta :\Z_p^\times \rightarrow \F_p^\times$ and let $\Lambda = \Z_p[[\Z_p^\times]]$. We have an induced map $\pi_\eta:\Lambda \rightarrow \F_p$ and we denote its kernel by $\m_\eta$. We write $\Lambda_\eta$ for the localisation $\Lambda_{\m_\eta}$. We have a universal character \[ [\cdot]_\eta: \Z_p^\times \rightarrow \Lambda_{\eta}^\times\] 

We give the complete local ring $\Lambda_{\eta}$ the $\m_\eta$-adic topology. Fixing a topological generator $\gamma$ of $1+p\Z_p$ gives an isomorphism \begin{align*}\Z_p[[X]] &\rightarrow  \Lambda_{\eta}\\ X &\mapsto [\gamma]_\eta-1.\end{align*}

Let $\mathfrak{W}_\eta = \Spa(\Lambda_{\eta},\Lambda_{\eta})$, denote its analytic locus by $\mathcal{W}_\eta$ and let $\U_1 \subset \mathcal{W}_\eta$ be the rational subdomain of $\mathfrak{W}_\eta$ \[\U_1 = \{|p| \le |X| \ne 0\}.\] Pulling back $\U_1$ to the open unit disc $\mathcal{W}_\eta^{rig}$ gives the `boundary annulus' $|X|_p\ge p^{-1}$.
	
We let $R_\eta = \OO(\U_1)$. More explicitly, $R_\eta = R_\eta^\circ[\frac{1}{X}]$ where $R_\eta^\circ$ is a ring of definition for $R_\eta$, given by the $X$-adic completion of $\Z_p[[X]][\frac{p}{X}]$, with the $X$-adic topology.

Even more explicitly, we can describe the elements of $R^\circ_\eta$ as formal power series \[\left\{\sum_{n \in \Z} a_n X^n: a_n \in \Z_p, |a_n|_pp^{-n}\le 1, |a_n|_pp^{-n} \rightarrow 0 \hbox{ as }n\rightarrow -\infty\right\}\]

$X$ is a topologically nilpotent unit in $R_\eta$ and so equipping $R_\eta$ with the norm $|r| = \inf\{p^{-n} \mid r \in X^nR_\eta^\circ, n \in \Z\}$ makes $R_\eta$ into a Banach--Tate $\Z_p$-algebra. This norm has the explicit description: 
\begin{equation}\label{Rwnorm}|\sum_{n \in \Z} a_n X^n| = \sup\{|a_n|_pp^{-n}\}.\end{equation}

Note that if $r \in \Lambda_\eta$, we have $|r| = \inf\{p^{-n} \mid r \in \m_\eta^n, n \in \Z\}$.

We now define a continuous character \[\kappa_\eta: T_0 \rightarrow R_\eta^\times\] by \[\kappa_\eta\begin{pmatrix}
a & 0\\ 0 & d\end{pmatrix} = [a]_\eta.\]

\begin{lemma}
	The norm we have defined on $R_\eta$ is adapted to $\kappa_\eta$. Moreover, for $t \in T_1$ we have $|\kappa_\eta(t)-1|\le 1/p$.
\end{lemma}
\begin{proof}If $t \in T_1$ we have $\kappa(t) - 1 = (1+X)^\alpha - 1 = \sum_{n\ge 1}{\alpha\choose n}X^n$ for some $\alpha \in \Z_p$. So $|\kappa_\eta(t)-1|\le 1/p$.
	\end{proof}

We can now apply the theory of Section \ref{evars} to the space of overconvergent automorphic forms $H^0(K,\mathcal{D}^{1/p}_{\kappa_\eta})$. Note that we have the concrete description:

\[H^0(K,\mathcal{D}^{1/p}_{\kappa_\eta}) = \{f: D^\times\backslash (D\otimes \A^f)^\times/K^p \rightarrow \mathcal{D}^{1/p}_{\kappa_\eta}~|~ f(gk) = k^{-1}f(g)\, \hbox{for }k\in I\} \] and $H^0(K,\mathcal{D}^{1/p}_{\kappa_\eta})$ is a Banach $R_\eta$-module with norm $|f| = \sup_{g \in (D\otimes \A^f)^\times}||f(g)||_{1/p}$.

In particular, we consider the action on 
$H^0(K,\mathcal{D}^{1/p}_{\kappa_\eta})$ of the Hecke operator 
$U_{\kappa_\eta}$ attached to the element $\left( \begin{smallmatrix} 1 & 0 \\ 
0 & p \end{smallmatrix} \right) \in \Sigma^{cpt}$. As a simple consequence of 
our results, we obtain the following theorem, which is essentially due to 
Liu--Wan--Xiao --- compare with \cite[Theorem 3.16, \S5.4]{lwx}. 
\begin{theorem}
	The Hecke operator $U_{\kappa_\eta}$ is compact. Consider the Fredholm series \[F_{\kappa_\eta}(T) = \sum_{n\ge 0}c_n T^n = \det\left(1-TU_{\kappa_\eta} | H^0(K,\mathcal{D}^{1/p}_{\kappa_\eta})\right).\] 
	
	Let $t =  |D^\times\backslash (D\otimes \A^f)^\times/K|$. We have $c_n \in \Lambda_\eta$ and  moreover we have \[c_n \in \m_\eta^{\lambda(n)} \hbox{for }n\in\Z_{\ge 0},\] where $\lambda(0) = 0,\lambda(1),\ldots$ is a sequence of integers determined by \[\lambda(0) = 0, \lambda(i+1) = \lambda(i) + \lfloor i/t\rfloor-\lfloor i/pt\rfloor\]
\end{theorem}
\begin{proof}Compactness of $U_{\kappa_\eta}$ follows from Corollary \ref{cor:cpt}. The fact that $c_n \in \Lambda_\eta$ follows from Corollary \ref{global}, since $F_{\kappa_\eta}(T)$ extends to a Fredholm series over $\mathcal{W}_\eta$, and $\OO(\mathcal{W}_\eta)=\Lambda_\eta$.
	
The rest of the Theorem follows from Lemma \ref{Upestimate} (note that the norm on $R_{\eta}$ satisfies the assumption of that Lemma), using the fact that if we choose representatives $g_1,\ldots, g_t$ for the double cosets $D^\times\backslash (D\otimes \A^f)^\times/K$ and $r\in [p^{-1},1)$ we have an isomorphism of potentially ON-able $R_\eta$-modules: \begin{align*}H^0(K,\mathcal{D}^{r}_{\kappa_\eta}) &\cong \oplus_{i=1}^t \mathcal{D}^{r}_{\kappa_\eta} \\ f &\mapsto (f(g_i))_{i=1}^t
\end{align*}
We take $\varpi = X$, and compute that  $n(p^{-1},X,\lfloor i/t\rfloor) = \lfloor i/t\rfloor$ and $n(p^{-1/p},X,\lfloor i/t\rfloor) = \lfloor 1/p \lfloor i/t\rfloor\rfloor = \lfloor i/pt\rfloor$. 
\end{proof}

As in \S \ref{cm}, our eigenvariety construction, applied to the modules $H^0(K,\mathcal{D}^{r}_{\kappa})$ with $\kappa_2 = 1$, gives an eigenvariety $\ms{E}_\eta$ which is flat over $\mc{W}_{\eta}$. The open subspace $\ms{E}^{rig}_\eta$ defined by $|p|\ne 0$ is the eigenvariety constructed by Buzzard \cite{bu2}. 

We end this section with our interpretation of \cite[Theorem 1.3, Theorem 1.5]{lwx}. First we need some extra notation. For $m \le n$ positive integers we define \[\U_{m/n} = \{|p^m| \le |X^n| \ne 0\} \sub \mc{W}_{\eta}.\] We set $\U= \U_1$.

For a real number $\alpha \in (0,1]$ we denote by $\mc{W}_{\eta}^{> p^{-\alpha}}$ the open subspace of $\mc{W}_{\eta}$ obtained as the union of open affinoids \[\mc{W}_{\eta}^{> p^{-\alpha}} = \bigcup_{m/n < \alpha}\U_{m/n}.\] Note that $X$ is a topologically nilpotent unit in $\oo(\U_{m/n})$ for all $m,n$. We denote the pullback of $\ms{E}_\eta$ to $\mc{W}_{\eta}^{> p^{-\alpha}}$ by $\ms{E}_\eta^{> p^{-\alpha}}$. The eigenvariety $\ms{E}_\eta$ comes equipped with a map to the spectral variety $Z(F_{\kappa_\eta})$, and therefore it comes equipped with a map to $\A^1_{\U}$. For $h = m/n \in \Q$ with $m \in \Z$ and $n\in\Z_{\ge 1}$ we define an affinoid subspace $\mb{B}_{\U,=h} \subset \A^1_{\U}$ by \[\mb{B}_{\U,=h} = \{|T^n| = |X^{-m}|\}.\] Similarly, if $h = m/n \le h' = m'/n$ are rational numbers, we define \[\mb{B}_{\U,[h,h']} = \{|X^{-m}| \le |T^n| \le |X^{-m'}|\}.\]

\begin{lemma}\label{renorm}
	Let $L/\Q_p$ be a finite extension and let $x \in L$ with $|x|_p = p^{-\alpha}$, where $0 < \alpha \le 1$. Consider the closed immersion $\iota: \Spa(L,\oo_L) \hookrightarrow \U$ induced by the continuous $\Zp$-algebra map $R_{\eta} \ra L$ sending $X$ to $x$. Let $\mb{B}_{x,=h}$ be the pullback of $\mb{B}_{\U,=h}$ along $\iota$. Then $\mb{B}_{x,=h} \hookrightarrow \A^1_{L}$ is the affinoid open defined by \[\mb{B}_{x,=h} = \{|T|_p = p^{-\alpha h}\}.\] 
\end{lemma}
\begin{proof}
	The affinoid $\mb{B}_{x,=h}$ is given by $\{|T^n| = |x^{-m}| = |p^{\alpha m}|\}\subset \A^1_{L}$.
	\end{proof}

\begin{theorem}[\cite{lwx}]\label{lwxboundary}
The space $\ms{E}_\eta^{> p^{-1}}$ is a disjoint countable union of adic spaces finite and flat over $\mc{W}_{\eta}^{> p^{-1}}$.	
	
Moreover, there is an explicit $\alpha$ depending only on $K^p$,  with $0 < \alpha < 1$, such that \[\ms{E}_\eta^{> p^{-\alpha}} = \coprod_{i \ge 0}\ms{X}_{\eta,i}\] with $\ms{X}_{\eta,i}$ finite flat over $\mc{W}_{\eta}^{> p^{-\alpha}}$ and each piece $\ms{X}_{\eta,i}$ of the eigenvariety has constant slope, in the sense that each map $\ms{X}_{\eta,i}\ra \A^1_{\U}$ factors through the affinoid subspace $\mb{B}_{\U,=h_i} \subset \A^1_{\U}$, for some $h_i \in \Q_{\ge 0}$. In particular, if we measure slopes on $\ms{X}^{rig}_{\eta,i}$ with the usual $p$-adic valuation, then the slope of a point in $\ms{X}^{rig}_{\eta,i}$ is given by $h_iv_p(T)$.
\end{theorem}
\begin{proof}
This follows from Theorems 1.3, 1.5 and Remark 3.25 of \cite{lwx}. Remark 3.25 
shows that, after restricting to $\mc{W}_{\eta}^{> p^{-1}}$, the Fredholm 
series $F_{\kappa_\eta}$ factorises as a countable product of multiplicative 
polynomials $\prod_{i \ge 0} P_i$, with each finite product $\prod_{i \ge 0}^N 
P_i$ a factor in a slope factorization over every affinoid subspace of 
$\mc{W}_{\eta}^{> p^{-1}}$. This establishes the claim about $\ms{E}_\eta^{> 
p^{-1}}$. Theorem 1.5 shows moreover that (for some explicit $\alpha$) the 
restriction of $F_{\kappa_\eta}$ to $\mc{W}_\eta^{> p^{-\alpha}}$ factorises as 
$\prod_{i\ge 0}Q_i$, such that the specialization of $Q_i$ at every classical 
rigid analytic point of $\mc{W}_\eta^{> p^{-\alpha}}$ has constant slope equal 
to $h_i$ for some $h_i \in \Q_{\ge 0}$ (independent of the specialization). We 
obtain a decomposition of $Z(F_{\kappa_\eta})$ as a disjoint union of spaces 
$Z_i$, finite flat over $\mc{W}_\eta^{> p^{-\alpha}}$, such that every 
classical rigid analytic point of $Z_i$ is contained in $\mb{B}_{\U,=h_i}$. The 
space $\ms{X}_{\eta,i}$ is defined to be the inverse image of $Z_i$ in 
$\ms{E}_\eta^{> p^{-\alpha}}$. It now remains to show that every point of $Z_i$ 
is contained in $\mb{B}_{\U,=h_i}$. First we check this for rank $1$ points: a 
rank $1$ point of $Z_i$ which is not in $\mb{B}_{\U,=h_i}$ is contained in 
$\mb{B}_{\U,[h,h']}$ for some interval $[h,h']$ which does not contain $h_i$. 
But then $\mb{B}_{\U,[h,h']}\cap Z_i$ is a non-empty open subset in $Z_i$ which 
contains no classical rigid analytic point, which is impossible since $Z_i$ is 
finite flat over $\mc{W}_\eta^{> p^{-\alpha}}$ (in particular the map $Z_i \ra 
\mc{W}_\eta^{> p^{-\alpha}}$ is open). Let $V$ be an affinoid open in 
$\mc{W}_\eta^{> p^{-\alpha}}$. Then $Z_i|_V\cap \mb{B}_{\U,=h_i}$ is an 
affinoid open in $Z_i|_V$ such that the complement contains no rank $1$ point. 
We have $Z_i|_V = \Spa(A,A^\circ)$ for some Tate ring $A$ with a Noetherian 
ring of definition. So \cite[Lemma 3.4]{hu1} (c.f~the proof of \cite[Corollary 
4.2]{hu1}) implies that rank $1$ points are dense in the constructible topology 
of $\Spa(A,A^\circ)$ and we deduce that $Z_i|_V\cap \mb{B}_{\U,=h_i} = Z_i|_V$. 
Therefore $Z_i$ is contained in $\mb{B}_{\U,=h_i}$, as desired. The final 
sentence of the Theorem follows from Lemma \ref{renorm}.
\end{proof}
Note that \cite{lwx} proves moreover that the slopes appearing in the above theorem (with multiplicities) are given by a finite union of arithmetic progressions.

\appendix
\counterwithin{theorem}{section}
\section{Some algebraic properties of Tate rings}\label{app}
In this section we prove some properties of the kinds of Tate rings and adic spaces that we need. We start with a ring theoretic lemma.

\begin{lemma}\label{jacobson}
Let $R$ be a complete Tate ring with a Noetherian ring of definition $R_{0}$. Then $R$ is Jacobson.
\end{lemma}

\begin{proof}
Let $\vp\in R_{0}$ be a topologically nilpotent unit in $R$. $R_{0}$ is a Zariski ring when equipped with its $\vp$-adic topology (since it is complete), so $\Spec R_{0} \setminus \{\vp=0\}$ is a Jacobson scheme by \cite[(10.5.7)]{ega43}. But $\Spec R_{0} \setminus \{ \vp=0 \} =\Spec R$, so $R$ is Jacobson as desired.
\end{proof}

We record another simple lemma that will prove to be useful.

\begin{lemma}\label{noeth}
Let $R$ be a complete Tate ring with a Noetherian ring of definition $R_{0}$. If $S\sub R$ is an open and bounded subring (i.e. a ring of definition) containing $R_{0}$, then $S$ is a finitely generated $R_0$-module, hence Noetherian, and integral over $R_{0}$. Moreover, $R^{\circ}$ is the integral closure of $R_{0}$ in $R$. 
\end{lemma}

\begin{proof}
Pick a topologically nilpotent unit $\vp\in R$ contained in $R_{0}$. Since $S$ is bounded we have $S\sub \vp^{-N}R_{0}$ for some $N$, and hence $S$ is an $R_{0}$-submodule of the cyclic $R_{0}$-module $\vp^{-N}R_{0}$. The lemma now follows since $R_{0}$ is Noetherian. For the last assertion, first note that the integral closure is contained in $R^{\circ}$. Since $R^{\circ}$ is the union of all open and bounded subrings and any two open bounded subrings are contained in a third, the assertion follows from the first part.
\end{proof}

One consequence of the above lemma is a version for Tate rings (with our Noetherian hypothesis) of \cite[6.3.4/Proposition 1]{bgr}:

\begin{lemma}\label{finitepowerbounded}
Let $R$ be a complete Tate ring with a Noetherian ring of definition. Let $S$ be a finite $R$-algebra, equipped with the natural $R$-module topology. Then $S$ is a complete Tate ring with a Noetherian ring of definition.

Moreover, the integral closure of $R^\circ$ in $S$ is equal to $S^\circ$. In particular, the morphism $(R,R^\circ)\ra (S,S^\circ)$ is a finite morphism of affinoid rings \cite[1.4.2]{hu3} and $\Spa(S,S^\circ) \ra \Spa(R,R^\circ)$ is a finite morphism of adic spaces \cite[1.4.4]{hu3}.
\end{lemma}
\begin{proof}
We let $R_0$ denote a Noetherian ring of definition for $R$, let $\vp$ denote a topologically nilpotent unit in $R$ and let $s_1,\ldots,s_n$ denote $R$-module generators of $S$. Each $s_i$ is integral over $R$, and multiplying by a large enough power of $\vp$ we may assume that each $s_i$ is integral over $R_0$. Let $S_0$ be the subring of $S$ generated by $R_0$ and $s_1,\ldots,s_n$. Now $S_0$ is a finite $R_0$-module and is in particular a Noetherian ring. Moreover $S_0$ with the $\vp$-adic topology is an open subring of $S$. In particular, $S$ is an $f$-adic topological ring, and since $\vp$ is a topologically nilpotent unit we see that $S$ is a Tate ring with a Noetherian ring of definition. Completeness of $S$ follows from completeness of finitely generated modules over Noetherian adic rings (this is \cite[Lemma 2.3 (ii)]{hu2}).

Finally we show that the integral closure $R^\circ$ in $S$ is equal to $S^\circ$. It is clear from the definition of the topology on $S$ that $R^\circ$ maps to $S^\circ$ so the integral closure of $R^\circ$ in $S$ is contained in $S^\circ$. Conversely, by Lemma \ref{noeth}, $S^\circ$ is the integral closure of $S_0$ in $S$. Since $S_0$ is integral over $R_0$, we see that $S^\circ$ is integral over $R_0$, and therefore it is integral over $R^\circ$.
\end{proof}

Next we recall the notion of uniformity. If $R$ is a normed ring, then the spectral seminorm $|-|_{sp}$ on $R$ is defined by $|r|_{sp}=\lim_{n\ra \infty} |r^{n}|^{1/n}$. It is well known that this limit exists and defines a power-multiplicative seminorm. Whenever it is a norm, we will refer to it as the spectral norm on $R$. Conversely, if we mention `the spectral norm of $R$', we are implicitly stating (or assuming) that the spectral seminorm is a norm.

\begin{definition}
Let $R$ be a complete Tate ring. We say that $R$ is \emph{uniform} if the set of power-bounded elements $R^{\circ}$ is bounded. We say that $R$ is \emph{stably uniform} if any rational localization of $R$ is also uniform. If $R$ is a Banach--Tate ring, we say that $R$ is \emph{uniform} if the norm is power-multiplicative.

\medskip
Note that, if $R$ is Banach--Tate ring whose underlying complete Tate ring is uniform, then the given norm on $R$ is equivalent to the corresponding spectral norm, which is power-multiplicative. In this case, \cite[Theorem 1.3]{be} says that the spectral norm is equal to the Gelfand norm $\sup_{x\in \mc{M}(R)}|-|_{x}$. If $R$ is in addition stably uniform, then if $\vp\in R$ is a multiplicative pseudo-uniformizer and $U\sub X=\Spa(R,R^{+})$ is a rational subdomain, we may equate $\mc{M}(\oo_{X}(U))$ with the rank $1$ points in $U$ using $\vp$ and equip $\oo_{X}(U)$ with the corresponding Gelfand norm.
\end{definition}  

We may extend the definition of stable uniformity to arbitrary \emph{analytic} adic spaces, i.e. those that are locally the adic spectra of complete Tate rings. We say that such an $X$ is stably uniform if there is a cover of open affinoid subsets $U_{i}\sub X$ such that $\oo_{X}(U)$ is stably uniform. We remark that if $R$ is complete sheafy Tate ring such that $\Spa(R,R^{+})$ is stably uniform, then $R$ is stably uniform (this is a short argument; cf. \cite[Remark 2.8.12]{kl}). When $R$ has a Noetherian ring of definition, many naturally occurring complete Tate rings are stably uniform. Below we will prove some results in this direction.

\begin{theorem}\label{su1}
Let $A$ be a reduced quasi-excellent ring. Let $I$ be an ideal of $A$ and give $A$ the $I$-adic topology. If $U$ is a rational subdomain of $X=\Spa(A,A)$ and $\oo_{X}(U)$ is Tate, then $\oo_{X}(U)$ is uniform. In other words, the analytic locus $X^{an}\sub X$ is stably uniform. We also have $\oo_{X}^{+}(U)=\oo_{X}(U)^{\circ}$.

Moreover, if $U\subset X^{an}$ is an arbitrary open affinoid then $\oo_{X^{an}}^{+}(U)=\oo_{X^{an}}(U)^{\circ}$ and $\oo_{X^{an}}(U)^{\circ}$ is bounded in $\oo_{X^{an}}(U)$.
\end{theorem}

\begin{proof}
Let $f_{1},\ldots,f_{n},g\in A$ such that $f_1,\ldots,f_n$ generate an open 
ideal and let $U=\{ |f_{1}|,\ldots,|f_{n}|\leq |g|\neq 0 \}$. Put 
$R=\oo_{X}(U)$; recall that $R$ may be constructed as completion of the f-adic 
ring $T=A[1/g]$ with ring of definition $T_{0}=A[f_{1}/g,\ldots,f_{n}/g]\sub T$ 
and ideal of definition $J=I[f_{1}/g,\ldots,f_{n}/g]\sub T_{0}$. Let $R_{0}$ be 
the $J$-adic completion of $T_{0}$; this is a ring of definition of $R$, with 
ideal of definition $JR_{0}$. Since $A$ is reduced, so is $T$ and hence 
$T_{0}$, moreover $T_{0}$ is quasi-excellent since it is finitely generated 
over $A$. Recall that a Noetherian ring is reduced if and only if Serre's 
conditions $R_{0}$ and $S_{1}$ hold (we apologize for the unfortunate clash of 
notations). By \cite[(7.8.3.1)]{ega42}, $R_{0}$ inherits these properties from 
$T_{0}$ and is therefore reduced. Moreover, $T_{0}$ and hence 
$T_{0}/J=R_{0}/JR_{0}$ is Nagata (since they are finitely generated over $A$, 
which is quasi-excellent, hence Nagata). By \cite[Proposition 2.3]{ma}, $R_{0}$ 
is Nagata (note that there is a trivial misprint in the reference). 

Pick a topologically nilpotent unit $\vp\in R$ (recall that $R$ is Tate by assumption); without loss of generality assume $\vp\in R_{0}$. Then $R=R_{0}[1/\vp]$, so $R$ is contained in the total ring of fractions $Q(R_{0})$ of $R_{0}$. Since $R_{0}$ is reduced and Nagata, it follows that the integral closure $R^{\prime}$ of $R_{0}$ in $R$ is a finitely generated $R_{0}$-module, hence is bounded. Now $R^{\prime}=R^{\circ}$ by Lemma \ref{noeth}, so $R^{\circ}$ is bounded as desired. 

For the assertion about $\oo_{X}^{+}(U)$, let $T^{+}$ denote the integral closure of $T_{0}$ in $T$. By definition,  the completion of $T^{+}$ is $R^{+}:=\oo_{X}^{+}(U)$. In particular, $R^{+}$ contains $R_{0}$ and the assertion now follows from Lemma \ref{noeth}.

To check the assertion about an open affinoid $U = \Spa(\oo_X(U),\oo^+_X(U))$, note that $U$ has a finite cover by Tate rational subdomains $(U_i)_{i \in I}$ of $X$. Since the maps $\oo(U)\rightarrow \oo(U_i)$ are bounded (the $U_i$ are also rational subdomains of $U$), the strict embedding \[\oo(U)\hookrightarrow \prod_{i\in I}\oo(U_i)\] induces an embedding  \[\oo(U)^\circ \hookrightarrow \oo(U)\cap\prod_{i\in I}\oo(U_i)^\circ\] but the right hand side equals \[\oo(U)\cap\prod_{i\in I}\oo^+(U_i)=\oo^+(U)\] by the first part of the Theorem, so we are done because by definition $\oo^+(U) \sub \oo(U)^\circ$ which implies that we have equality. Finally, the boundedness of $\oo_{X^{an}}(U)^{\circ}$ follows from the boundedness of the $\oo(U_i)^\circ$.
\end{proof}

\begin{corollary}\label{su2}
Let $\oo$ be a complete discrete valuation ring and let $A$ be a reduced complete Noetherian adic ring formally of finite type over $\oo$, i.e. such that $A/A^{\circ\circ}$ is a finitely generated $\oo$-algebra. Then the analytic locus $X^{an}\sub X$ is stably uniform. Moreover, if $U\subset X^{an}$ is an open affinoid subspace then $\oo_{X^{an}}^{+}(U)=\oo_{X^{an}}(U)^{\circ}$ and $\oo_{X^{an}}(U)^{\circ}$ is bounded in $\oo_{X^{an}}(U)$.
\end{corollary}

\begin{proof}
$A$ is excellent by \cite[Proposition 7]{v1} and \cite[Theorem 9]{v2} (cf. \cite{con}, near the end of the Introduction). Thus, Theorem \ref{su1} applies. 
\end{proof}

We also note that the proof of Theorem \ref{su1} applies essentially verbatim to prove the following similar result.

\begin{theorem}\label{su3}
Let $R$ be a complete Tate ring and assume that $R$ has a ring of definition $R_{0}$ which is quasi-excellent and reduced. Then $R$ is stably uniform. Moreover, if $X=\Spa(R,R^{\circ})$, and $U\subset X$ is an open affinoid subspace, then $\oo_{X}^{+}(U)=\oo_{X}(U)^{\circ}$ and and $\oo_{X}(U)^{\circ}$ is bounded in $\oo_{X}(U)$. In particular, $\oo_{X}(U)$ is reduced.
\end{theorem}

This Theorem also allows us to develop the theory of the nilreduction of an adic space. We only give a sketch here --- one can check that everything in \cite[\S 9.5.1]{bgr} works in our setting.

\begin{definition}
Let $X$ be an adic space. Define the \emph{nilradical} $\rad \oo_X$ to be the sheaf associated to the presheaf $U \mapsto \rad(\oo_X(U))$, where $\rad(\oo_X(U))$ is the nilradical of the ring $\oo_X(U)$. 
\end{definition}

\begin{proposition}
Let $R$ be a complete Tate ring and assume that $R$ has a ring of definition $R_{0}$ which is quasi-excellent. Let $X = \Spa(R,R^{\circ})$. Then $\rad \oo_X \subset \oo_X$ is a coherent $\oo_X$-ideal, associated to the ideal $\rad(R)$ of $R$. 

More generally, if $X$ is an adic space which is locally of the form $\Spa(R,R^{\circ})$ where $R$ is a complete Tate ring with a quasi-excellent ring of definition, then $\rad \OO_X$ is a coherent $\oo_X$-ideal.
\end{proposition}
\begin{proof}
The key point is that if $U \subset X = \Spa(R,R^{\circ})$ is a rational subdomain, then \[\Spa(\OO_X(U)/\rad(R),(\OO_X(U)/\rad(R))^\circ) \ra\Spa(R^{red},(R^{red})^{\circ})\] is a rational subdomain, so Theorem \ref{su3} implies that $\oo_X(U)/\rad(R)$ is reduced, which implies that $\rad(\oo_X(U)) = \rad(R)\oo_X(U)$.
\end{proof}
\begin{definition}\label{defred}
Let $X$ be an adic space which is locally of the form $\Spa(R,R^{\circ})$, where $R$ is a complete Tate ring with a quasi-excellent ring of definition. Then we define $X^{red}$ to be the closed subspace of $X$ cut out by $\rad \oo_X$ (see \cite[1.4]{hu3}).
\end{definition}

In this paper the analytic adic spaces encountered will locally be of the form $\Spa(R,R^{\circ})$, where $R$ is a complete Tate ring with a ring of definition $R_{0}$ which is formally of finite type over $\Zp$. We will need a few properties of these rings, all of which follow from the material in \cite{ab}. We recall the following definition from \cite{ab}, specialized to our Noetherian situation.

\begin{definition}
A Noetherian adic ring $B$ is called a \emph{$1$-valuative order} if it is an integral domain which is local of Krull dimension $1$, and has no $J$-torsion, where $J$ is an ideal of definition (this is independent of the choice of ideal of definition).
\end{definition}

This is \cite[Definition 1.11.1]{ab}, except that we demand that $B$ is Noetherian. If $B$ is a $1$-valuative order, then the integral closure $\ol{B}$ in $L={\rm Frac}(B)$ is finite over $B$ and is a complete discrete valuation ring, so $L$ is a complete discrete valuation field \cite[Proposition 1.11.4]{ab}. If $A$ is any Noetherian adic ring with an ideal of definition $I$ and $\mf{p}\in \Spec A$, then $A/\mf{p}$ is a $1$-valuative order if and only if $\mf{p}$ is a closed point in $\Spec A \setminus V(I)$ \cite[Proposition 1.11.8]{ab}. 

\begin{lemma}\label{points}
Let $R$ be a complete Tate ring with a ring of definition $R_{0}$ which is formally of finite type over $\Zp$, and let $\mf{m}\sub R$ be a maximal ideal. Then $R/\mf{m}$ is a local field.
\end{lemma}

\begin{proof}
Let $\mf{p}=R_{0}\cap \mf{m}$ and let $\vp\in R_{0}$ be a topologically 
nilpotent unit. Then $\mf{p}$ is a closed point in $\Spec R_{0} \setminus 
V((\vp))=\Spec R$, so $R_{0}/\mf{p}$ is a $1$-valuative order and hence its 
fraction field $R/\mf{m}$ is a complete discrete valuation field. It remains to 
prove that the residue field is finite. For this, it suffices to show that the 
residue field of the local ring $R_{0}/\mf{p}$ is finite since the integral 
closure of $R_{0}/\mf{p}$ in $R/\mf{m}$ is finite over $R_{0}/\mf{p}$. Pick an 
adic surjection $A=\Zp[[T_{1},\ldots,T_{m}]]\langle X_{1},\ldots,X_{n} \rangle 
\twoheadrightarrow R_{0}$ for some $m,n\in \Z_{\geq 0}$. The maximal ideal of 
$R_{0}/\mf{p}$ is open and so corresponds to an open maximal ideal of $A$, and 
hence to a maximal ideal of $\Fp[X_{1},\ldots,X_{n}]$ in a way that preserves 
residue fields. It follows that $R_{0}/\mf{p}$ is finite as desired.
\end{proof}

\begin{lemma}\label{maxideals}
Let $f : A \ra B$ be a morphism of topologically finite type between Noetherian adic rings. Let $I$ be an ideal of definition of $A$ and assume that $A/I$ is Jacobson. Let $J=IB$; this is an ideal of definition of $B$. If $\mf{q}\in \Spec B \setminus V(J)$ is a closed point, then $\mf{p}=f^{-1}(\mf{q})$ is a closed point in $\Spec A \setminus V(I)$.
\end{lemma}

\begin{proof}
The morphism $A \ra B/\mf{q}$ is topologically of finite type and $B/\mf{q}$ is a $1$-valuative order, so by \cite[Proposition 1.11.2]{ab} $A \ra B/\mf{q}$ is finite. It is then easy to check that this forces $A/\mf{p}$ to be a $1$-valuative order as well, and hence $\mf{p}$ to be closed in $\Spec A \setminus V(I)$ by \cite[Proposition 1.11.8]{ab}.
\end{proof}

\begin{corollary}\label{maxideals2}
Let $g : R \ra S$ be a continuous morphism between between two complete Tate rings with a ring of definition that is formally of finite type over $\Zp$. Then $g$ is topologically of finite type\footnote{i.e.~$g$ factors through a surjective, continuous and open morphism $R\langle X_1,\ldots,X_n\rangle \rightarrow S$, see \cite[Lemma 3.3]{hu2}} and pulls back maximal ideals to maximal ideals.
\end{corollary}

\begin{proof}
Choose a ring of definition $R_{0}$ for $R$ which is formally of finite type over $\Zp$. By \cite[Proposition 1.10]{hu1} $g$ is adic, and therefore $g(R_0)$ is contained in a ring of definition for $S$. Since any two rings of definition are contained in another, we can find a ring of definition $S_0$ for $S$ such that $g(R_0)\sub S_0$ and $S_0$ contains a ring of definition $S_1$ which is formally of finite type over $\Z_p$. It follows from Lemma \ref{noeth} that $S_0$ is finite over $S_1$, and hence $S_0$ is also formally of finite type over $\Z_p$.

Let $\vp\in R_{0}$ be a topologically nilpotent unit in $R$. Then $I=\vp R_{0}$ and $J=g(\vp)S_{0}$ are ideals of definition, and $R_{0}/I \ra S_{0}/J$  is of finite type. Therefore $R_{0} \ra S_{0}$ is topologically of finite type, hence so is $g$. This proves the first assertion. The second then follows from Lemma \ref{maxideals}, since maximal ideals of $R$ (resp. $S$) correspond to closed points in $\Spec R=\Spec R_{0} \setminus V(I)$ (resp. $\Spec S=\Spec S_{0} \setminus V(J)$).
\end{proof}

\begin{proposition}\label{localrings}
Let $S$ be a complete Tate ring with a Noetherian ring of definition $S_{0}$ and a topologically nilpotent unit $\pi \in S_{0}$ such that $S_{0}/\pi S_{0}$ is Jacobson.
\begin{enumerate}
\item Let $A$ be a Noetherian adic ring with an ideal of definition $I$ such that $A/I$ is Jacobson. Let $f : A \ra S$ be a continuous morphism such that the induced map $\Spa(S,S^{\circ}) \ra \Spa(A,A)$ is an open immersion, and let $\mf{q}$ be a maximal ideal of $S$ with preimage $\mf{p} =f^{-1}(\mf{q})$ in $A$. Then the natural map $A_{\mf{p}} \ra S_{\mf{q}}$ induces an isomorphism on completions (with respect to the maximal ideals).

\item Let $R$ be a complete Tate ring with a Noetherian ring of definition $R_{0}$ and a topologically nilpotent unit $\vp\in R_{0}$ such that $R_{0}/\vp$ is Jacobson. Let $h : R \ra S$ be a continuous morphism such that the induced map $\Spa(S,S^{\circ}) \ra \Spa(R,R^{\circ})$ is an open immersion, and let $\mf{q}^{\prime}$ be a maximal ideal of $S$ with preimage $\mf{p}^{\prime} =h^{-1}(\mf{q}^{\prime})$ in $R$. Then the natural map $R_{\mf{p}^{\prime}} \ra S_{\mf{q}^{\prime}}$ induces an isomorphism on completions (with respect to the maximal ideals).
\end{enumerate}
\end{proposition}

\begin{proof}
We prove part (1); the proof of part (2) is virtually identical. Since 
$S/\mf{q}$ is a complete discretely valued field it defines a point $v$ in 
$\Spa(S,S^{\circ})\sub \Spa(A,A)$; let $U=\{|f_{1}|,\ldots,|f_{n}|\leq |g| \neq 
0 \}$ be a rational subdomain of $\Spa(A,A)$ which contains this point and is 
contained in $\Spa(S,S^{\circ})$. Let $T=A[f_{1}/g,\ldots,f_{n}/g]\sub A[1/g]$ 
and let $\wh{T}$ be the $IT$-adic completion of $T$.  Since 
$\wh{T}[1/g]=\oo(U)$ we see that the valuation $v$ extends to a valuation $w$ 
on $\wh{T}[1/g]$, and hence $\mf{q}$ extends to a maximal ideal $\mf{r}=\Ker w$ 
of $\wh{T}[1/g]$. We will abuse notation and let $\mf{r}$ denote its preimage 
in any of the rings $T, T[1/g]$ and $\wh{T}$ as well; then $\mf{r}$ is a closed 
point in $\Spec \wh{T} \setminus V(I\wh{T})$. 

\medskip

By \cite[Proposition 1.12.18]{ab}, the natural map $T_{\mf{r}} \ra \wh{T}_{\mf{r}}$ induces an isomorphism of completions. We claim that the natural maps $A_{\mf{p}} \ra T_{\mf{r}}$ and $\wh{T}_{\mf{r}} \ra \wh{T}[1/g]_{\mf{r}}$ are isomorphisms. For the second map this is clear (by the general fact that if $B$ is any ring, $f\in B$, and $P\in \Spec B[1/f]\sub \Spec B$, then the natural map $B_{P} \ra B[1/f]_{P}$ is an isomorphism). For the first map, we have natural maps $A_{\mf{p}} \ra T_{\mf{r}} \ra T[1/g]_{\mf{r}}=A[1/g]_{\mf{r}}$ and it is clear that the second map and the composite are isomorphisms, so the first map is an isomorphism as well. Summing up, we see that the natural map $A_{\mf{p}} \ra \oo(U)_{\mf{r}}$ induces an isomorphism on completions. By an almost identical argument, the natural map $S_{\mf{q}} \ra \oo(U)_{\mf{r}}$ induces an isomorphism on completions. It then follows that the natural map $A_{\mf{p}} \ra S_{\mf{q}}$ induces an isomorphism on completions, as desired.
\end{proof}

\bibliographystyle{alpha}
\bibliography{halo}

\begin{thebibliography}{DdSMS99}

\bibitem[Abb10]{ab}
Ahmed Abbes.
\newblock {\em \'{E}l\'ements de g\'eom\'etrie rigide. {V}olume {I}}, volume
  286 of {\em Progress in Mathematics}.
\newblock Birkh\"auser/Springer Basel AG, Basel, 2010.
\newblock Construction et {\'e}tude g{\'e}om{\'e}trique des espaces rigides.
  [Construction and geometric study of rigid spaces], With a preface by Michel
  Raynaud.

\bibitem[AIP]{aip}
Fabrizio Andreatta, Adrian Iovita, and Vincent Pilloni.
\newblock Le halo spectral.
\newblock To appear in Ann.~Sci.~ENS, available at
  \url{http://www.mat.unimi.it/users/andreat/SpectralHalo.pdf}.

\bibitem[AIP16]{aip2}
Fabrizio Andreatta, Adrian Iovita, and Vincent Pilloni.
\newblock The adic, cuspidal, {H}ilbert eigenvarieties.
\newblock {\em Research in the Mathematical Sciences}, 3(1):34, 2016.

\bibitem[AS]{as}
Avner Ash and Glenn Stevens.
\newblock $p$-adic deformations of arithmetic cohomology.
\newblock Preprint,
  \url{http://math.bu.edu/people/ghs/preprints/Ash-Stevens-02-08.pdf}.

\bibitem[AW13]{aw}
Konstantin Ardakov and Simon Wadsley.
\newblock On irreducible representations of compact {$p$}-adic analytic groups.
\newblock {\em Ann. of Math. (2)}, 178(2):453--557, 2013.

\bibitem[BC09]{bc}
Jo{\"e}l Bella{\"{\i}}che and Ga{\"e}tan Chenevier.
\newblock Families of {G}alois representations and {S}elmer groups.
\newblock {\em Ast\'erisque}, (324):xii+314, 2009.

\bibitem[Bel]{bel}
Jo\"{e}l Bella\"{i}che.
\newblock Eigenvarieties and $p$-adic {$L$}-functions.
\newblock Preprint.

\bibitem[Ber90]{be}
Vladimir~G. Berkovich.
\newblock {\em Spectral theory and analytic geometry over non-{Archimedean}
  fields}, volume~33 of {\em Mathematical Surveys and Monographs}.
\newblock American Mathematical Society, Providence, RI, 1990.

\bibitem[BGR84]{bgr}
Siegfried Bosch, Ulrich G{\"u}ntzer, and Reinhold Remmert.
\newblock {\em Non-{A}rchimedean analysis}, volume 261 of {\em Grundlehren der
  Mathematischen Wissenschaften [Fundamental Principles of Mathematical
  Sciences]}.
\newblock Springer-Verlag, Berlin, 1984.
\newblock A systematic approach to rigid analytic geometry.

\bibitem[BHS17]{bhs}
Christophe Breuil, Eugen Hellmann, and Benjamin Schraen.
\newblock Une interpr{\'e}tation modulaire de la vari{\'e}t{\'e} trianguline.
\newblock {\em Mathematische Annalen}, 367(3):1587--1645, 2017.

\bibitem[BK05]{bk}
Kevin Buzzard and L.~J.~P. Kilford.
\newblock The 2-adic eigencurve at the boundary of weight space.
\newblock {\em Compos. Math.}, 141(3):605--619, 2005.

\bibitem[BP16]{bp}
John Bergdall and Robert Pollack.
\newblock Arithmetic properties of {F}redholm series for $p$-adic modular
  forms.
\newblock {\em Proceedings of the London Mathematical Society}, 113(4):419,
  2016.

\bibitem[Buz04]{bu2}
Kevin Buzzard.
\newblock On {$p$}-adic families of automorphic forms.
\newblock In {\em Modular curves and abelian varieties}, volume 224 of {\em
  Progr. Math.}, pages 23--44. Birkh\"auser, Basel, 2004.

\bibitem[Buz07]{bu}
Kevin Buzzard.
\newblock Eigenvarieties.
\newblock In {\em {$L$}-functions and {G}alois representations}, volume 320 of
  {\em London Math. Soc. Lecture Note Ser.}, pages 59--120. Cambridge Univ.
  Press, Cambridge, 2007.

\bibitem[Che04]{chegln}
Ga{\"e}tan Chenevier.
\newblock Familles {$p$}-adiques de formes automorphes pour {${\rm GL}\sb n$}.
\newblock {\em J. Reine Angew. Math.}, 570:143--217, 2004.

\bibitem[Che05]{che2}
Ga{\"e}tan Chenevier.
\newblock Une correspondance de {Jacquet}--{Langlands} {$p$}-adique.
\newblock {\em Duke Math. J.}, 126(1):161--194, 2005.

\bibitem[Che14]{che}
Ga\"{e}tan Chenevier.
\newblock The $p$-adic analytic space of pseudocharacters of a profinite group
  and pseudorepresentations over arbitrary rings.
\newblock In {\em Automorphic forms and {G}alois representations, Volume 1},
  volume 414 of {\em London Math. Soc. Lecture Note Ser.}, pages 221--285.
  Cambridge Univ. Press, Cambridge, 2014.

\bibitem[CHJ17]{chj}
Przemys{\l}aw Chojecki, David Hansen, and Christian Johansson.
\newblock Overconvergent modular forms and perfectoid {S}himura curves.
\newblock {\em Doc. Math.}, 22:191--262, 2017.

\bibitem[CM98]{cm}
R.~Coleman and B.~Mazur.
\newblock The eigencurve.
\newblock In {\em Galois representations in arithmetic algebraic geometry
  ({D}urham, 1996)}, volume 254 of {\em London Math. Soc. Lecture Note Ser.},
  pages 1--113. Cambridge Univ. Press, Cambridge, 1998.

\bibitem[Col97]{co}
Robert~F. Coleman.
\newblock {$p$}-adic {Banach} spaces and families of modular forms.
\newblock {\em Invent. Math.}, 127(3):417--479, 1997.

\bibitem[Col10]{colmez}
Pierre Colmez.
\newblock Fonctions d'une variable {$p$}-adique.
\newblock {\em Ast\'erisque}, (330):13--59, 2010.

\bibitem[Con99]{con}
Brian Conrad.
\newblock Irreducible components of rigid spaces.
\newblock {\em Ann. Inst. Fourier (Grenoble)}, 49(2):473--541, 1999.

\bibitem[DdSMS99]{ddms}
J.~D. Dixon, M.~P.~F. du~Sautoy, A.~Mann, and D.~Segal.
\newblock {\em Analytic pro-{$p$} groups}, volume~61 of {\em Cambridge Studies
  in Advanced Mathematics}.
\newblock Cambridge University Press, Cambridge, second edition, 1999.

\bibitem[Gro65]{ega42}
Alexander Grothendieck.
\newblock \'{E}l\'ements de g\'eom\'etrie alg\'ebrique. {IV}. \'{E}tude locale
  des sch\'emas et des morphismes de sch\'emas. {II}.
\newblock {\em Inst. Hautes \'Etudes Sci. Publ. Math.}, (24):231, 1965.

\bibitem[Gro66]{ega43}
Alexander Grothendieck.
\newblock \'{E}l\'ements de g\'eom\'etrie alg\'ebrique. {IV}. \'{E}tude locale
  des sch\'emas et des morphismes de sch\'emas. {III}.
\newblock {\em Inst. Hautes \'Etudes Sci. Publ. Math.}, (28):255, 1966.

\bibitem[Gul19]{gulottapub}
Daniel~R. Gulotta.
\newblock Equidimensional adic eigenvarieties for groups with discrete series.
\newblock {\em Algebra Number Theory}, 13(8):1907--1940, 2019.

\bibitem[Han]{han2}
David Hansen.
\newblock Iwasawa theory of overconvergent modular forms, {I}: {C}ritical
  $p$-adic {$L$}-functions.
\newblock Preprint, \url{http://arxiv.org/abs/1508.03982}.

\bibitem[Han17]{han}
David Hansen.
\newblock Universal eigenvarieties, trianguline {G}alois representations, and
  {$p$}-adic {L}anglands functoriality.
\newblock {\em J. Reine Angew. Math.}, 730:1--64, 2017.
\newblock With an appendix by James Newton.

\bibitem[HJ]{ch}
David Hansen and Christian Johansson.
\newblock Completed and overconvergent cohomology.
\newblock In preparation.

\bibitem[Hub93]{hu1}
Roland Huber.
\newblock Continuous valuations.
\newblock {\em Math. Z.}, 212(3):455--477, 1993.

\bibitem[Hub94]{hu2}
Roland Huber.
\newblock A generalization of formal schemes and rigid analytic varieties.
\newblock {\em Math. Z.}, 217(4):513--551, 1994.

\bibitem[Hub96]{hu3}
Roland Huber.
\newblock {\em \'{E}tale cohomology of rigid analytic varieties and adic
  spaces}.
\newblock Aspects of Mathematics, E30. Friedr. Vieweg \& Sohn, Braunschweig,
  1996.

\bibitem[JN]{erratum}
Christian Johansson and James Newton.
\newblock Erratum to \cite{extended}.

\bibitem[JN17]{jn2}
Christian {Johansson} and James {Newton}.
\newblock {Irreducible components of extended eigenvarieties and interpolating
  {L}anglands functoriality}.
\newblock {\em ArXiv e-prints}, March 2017.
\newblock To appear in MRL.

\bibitem[JN18]{jn3}
C.~{Johansson} and J.~{Newton}.
\newblock {Parallel weight 2 points on Hilbert modular eigenvarieties and the
  parity conjecture}.
\newblock {\em ArXiv e-prints}, January 2018.

\bibitem[JN19]{extended}
Christian Johansson and James Newton.
\newblock Extended eigenvarieties for overconvergent cohomology.
\newblock {\em Algebra Number Theory}, 13(1):93--158, 2019.

\bibitem[KL15]{kl}
Kiran~S. Kedlaya and Ruochuan Liu.
\newblock Relative {$p$}-adic {Hodge} theory: foundations.
\newblock {\em Ast\'erisque}, (371):239, 2015.

\bibitem[Loe11]{loe}
David Loeffler.
\newblock Overconvergent algebraic automorphic forms.
\newblock {\em Proc. Lond. Math. Soc. (3)}, 102(2):193--228, 2011.

\bibitem[LWX17]{lwx}
Ruochuan Liu, Daqing Wan, and Liang Xiao.
\newblock The eigencurve over the boundary of weight space.
\newblock {\em Duke Math. J.}, 166(9):1739--1787, 06 2017.

\bibitem[Mar75]{ma}
Jean Marot.
\newblock Sur les anneaux universellement japonais.
\newblock {\em Bull. Soc. Math. France}, 103(1):103--111, 1975.

\bibitem[Neu99]{neu}
J{\"u}rgen Neukirch.
\newblock {\em Algebraic number theory}, volume 322 of {\em Grundlehren der
  Mathematischen Wissenschaften [Fundamental Principles of Mathematical
  Sciences]}.
\newblock Springer-Verlag, Berlin, 1999.
\newblock Translated from the 1992 German original and with a note by Norbert
  Schappacher, With a foreword by G. Harder.

\bibitem[PX]{px}
Jonathan Pottharst and Liang Xiao.
\newblock On the parity conjecture in finite-slope families.
\newblock Preprint, \url{http://arxiv.org/abs/1410.5050}.

\bibitem[Sch15]{sch}
Peter Scholze.
\newblock On torsion in the cohomology of locally symmetric varieties.
\newblock {\em Ann. of Math. (2)}, 182(3):945--1066, 2015.

\bibitem[Ser62]{ser}
Jean-Pierre Serre.
\newblock Endomorphismes compl\`etement continus des espaces de {B}anach
  {$p$}-adiques.
\newblock {\em Inst. Hautes \'Etudes Sci. Publ. Math.}, (12):69--85, 1962.

\bibitem[ST03]{st}
Peter Schneider and Jeremy Teitelbaum.
\newblock Algebras of {$p$}-adic distributions and admissible representations.
\newblock {\em Invent. Math.}, 153(1):145--196, 2003.

\bibitem[Ste]{ste}
Glenn Stevens.
\newblock Rigid analytic modular symbols.
\newblock Preprint,
  \url{http://math.bu.edu/people/ghs/preprints/OC-Symbs-04-94.pdf}.

\bibitem[Urb11]{urb}
Eric Urban.
\newblock Eigenvarieties for reductive groups.
\newblock {\em Ann. of Math. (2)}, 174(3):1685--1784, 2011.

\bibitem[Val75]{v1}
Paolo Valabrega.
\newblock On the excellent property for power series rings over polynomial
  rings.
\newblock {\em J. Math. Kyoto Univ.}, 15(2):387--395, 1975.

\bibitem[Val76]{v2}
Paolo Valabrega.
\newblock A few theorems on completion of excellent rings.
\newblock {\em Nagoya Math. J.}, 61:127--133, 1976.

\bibitem[Wei94]{wei}
Charles~A. Weibel.
\newblock {\em An introduction to homological algebra}, volume~38 of {\em
  Cambridge Studies in Advanced Mathematics}.
\newblock Cambridge University Press, Cambridge, 1994.

\bibitem[Xia12]{xia}
Zhengyu Xiang.
\newblock A construction of the full eigenvariety of a reductive group.
\newblock {\em J. Number Theory}, 132(5):938--952, 2012.

\end{thebibliography}

\end{document}